\documentclass[11pt]{article}
\usepackage[top=2cm, bottom=2cm, left=2cm, right=2cm]{geometry}

\usepackage{import}

\usepackage{amsmath, amssymb, bbold, mathtools}
\usepackage{algorithm,algpseudocode}
\usepackage{multirow}
\usepackage{graphicx}
\usepackage{tcolorbox}
\usepackage[]{hyperref}
\usepackage{bbding}
\usepackage[export]{adjustbox}
\usepackage{subcaption}
\usepackage{enumerate}
\usepackage{todonotes}

\usepackage{authblk}
\usepackage{cleveref} 

\usepackage{wasysym}
\usepackage{lmodern} 

\usepackage{natbib}
 \bibpunct[, ]{(}{)}{,}{a}{}{,}

\usepackage{tikz}
\usetikzlibrary{calc}
\usetikzlibrary{shapes}
\usepackage{pgfplots}
\usepgfplotslibrary{fillbetween}

\usepackage{amsthm}
\theoremstyle{plain}


\newcounter{parentnumber}

\newtheorem{theorem}{Theorem}
\newtheorem{lemma}{Lemma}

\newtheorem{corollary}{Corollary}
\newtheorem{remark}{Remark}

\newtheorem{proposition}{Proposition}
\newtheorem{assumption}{Assumption}

\usepackage{acronym}
\acrodef{wlog}[WLOG]{without loss of generality}
\acrodef{lsc}[lsc]{lower semi-continuous}

\definecolor{red}{RGB}{163, 31, 52}
\definecolor{gray}{RGB}{194, 192, 191}
\definecolor{blue}{RGB}{59, 89, 152}
\definecolor{green}{RGB}{0, 179, 0}

\newcommand{\abs}[1]{\left|#1\right|}

\newcommand{\E}[2]{{\mathbb E}_{#1} \left[ #2 \right]}

\renewcommand{\tfrac}[2]{{#1}/{#2}}

\newcommand{\set}[2]{\left\{ #1\ : \ #2 \right\}}

\newcommand{\defn}[0]{:=}

\newcommand{\mc}{\mathcal}
\newcommand{\mb}{\mathbb}

\newcommand{\mr}{\mathrm}

\renewcommand{\d}{{\mathrm{d}}}

\renewcommand{\Re}{\mathrm{R}}

\newcommand{\st}{\mr{s.t.}}

\usepackage{setspace}
\title{Optimal Estimators for Heavy-Tailed Mean Estimation via Convex Analysis}
\author[1]{Bart P.G.\ van Parys\thanks{\href{mailto:bart.van.parys@cwi.nl}{bart.van.parys@cwi.nl}}}
\author[2]{Bert Zwart\thanks{\href{mailto:bert.zwart@cwi.nl}{bert.zwart@cwi.nl}}}
\affil[1,2]{Centrum Wiskunde \& Informatica}
\affil[2]{Eindhoven University of Technology}

\date{}

\begin{document}
\maketitle

\begin{abstract}
  We study optimal estimation of the location parameter of a distribution known only to lie in a \emph{symmetric moment class} $\mc C_0$, the mean-zero distributions with bounded moment $\int\phi\,\d\mb P\le B$ for a fixed even moment function $\phi$. Our main result concerns the \emph{fixed-margin} regime, in which the error margin $\Delta$ is held fixed as $n\to\infty$: we give an exact large-deviation characterization of the smallest worst-case probability $\beta_n(\Delta)$ of an error exceeding $\Delta$ that any measurable estimator can guarantee with $n$ observations.
  We show that its exponential rate is exactly a two-point Hellinger exponent over the class shifted to means $\pm\Delta$: $-\tfrac1n\log\beta_n(\Delta)\to r(\Delta)=-\log\sup_{\mb P_{\pm\Delta}\in\mc C_{\pm\Delta}}\int\sqrt{\d\mb P_{-\Delta}\,\d\mb P_{\Delta}}$, achieved non-asymptotically, $\beta_n(\Delta)\le e^{-nr(\Delta)}$, by a monotone $M$-estimator synthesized from a two-parameter convex optimization problem. Lagrangian duality collapses the infinite-dimensional search over estimating functions to a search over only two multipliers, which in turn determine a pair of envelopes characterizing a set of optimal estimating functions. The sandwich shape posited ad hoc by \citet{catoni2012challenging} and \citet{lee2020optimal,bhatt2022nearly} emerges naturally in our framework. For the bounded-variance class ($\phi(x)=x^2$, $B=\sigma^2$) this exponent admits the closed form $r(\Delta)=\tfrac12\log(1+\Delta^2/\sigma^2)$. We turn next to the \emph{fixed-confidence} regime, which instead holds the confidence $\beta$ fixed and lets the optimal error margin $\Delta_n(\beta)$ shrink with $n$. In the context of several concrete classes, the same synthesis remains optimal to leading order in this distinct regime. In the high-confidence limit $\beta\downarrow0$ the synthesized $M$-estimator attains the sharp constant $\sqrt2$ of \citet{catoni2012challenging} for bounded variance and the constant $L(\alpha)$ of \citet{lee2020optimal,bhatt2022nearly} for bounded $\alpha$-moments with $\alpha\in(1,2)$, which we thereby establish as tight. For slowly varying moment functions it is instead leading-order minimax at every fixed confidence level. The least-favorable distributions behind the matching bounds are furthermore remarkably simple and supported on at most three atoms.
\end{abstract}

\section{Introduction}
\label{sec:intro}

We are interested in the classical problem of estimating a location $\mu\in\Re$ from independent random variables $\{X_i\in\Re\}_{i=1}^n$ identically distributed following a distribution $\mb P_\mu$. It is foundational to many downstream problems, such as stochastic optimization, robust regression, and bandit learning, and despite its classical nature has drawn renewed attention in the recent literature \citep{catoni2012challenging,devroye2015subgaussianmeanestimators,lugosi2019mean,lee2020optimal,bhatt2022nearly}. We assume here that we are in a location family, that is $X_i-\mu\sim \mb P_0$ for all $i\in\{1,\dots,n\}$, and that the distribution $\mb P_0$ generating the family is not known exactly. Instead, we know only that $\mb P_0\in \mc C_0$ for some convex subset $\mc C_0$ of the probability simplex $\mc P$, the set of probability measures on $\Re$. We focus on \emph{symmetric moment classes}
\begin{equation}
  \label{eq:moment-class-intro}
  \mc C_0 \;=\; \set{\mb P\in\mc P}{\textstyle\int X\,\d\mb P=0,\ \int\phi(X)\,\d\mb P\leq B},
\end{equation}
for some even, continuous, super-linear function $\phi:\Re\to\Re_+$ non-decreasing on $[0,\infty)$ with $\phi(0)<B$, which subsume the bounded-variance class of \citet{catoni2012challenging} ($\phi(x)=|x|^2$) and the bounded-$\alpha$-moment class of \citet{lee2020optimal,chen2021generalized,bhatt2022nearly} ($\phi(x)=|x|^\alpha$). The super-linearity $\phi(x)/|x|\to\infty$ as $|x|\to\infty$ ensures tightness of $\mc C_0$. We call $\mc C_0$ \emph{symmetric} because the evenness of $\phi$ makes it invariant under the reflection $X\mapsto-X$; the individual members of $\mc C_0$ need not be symmetric distributions. For any $\mu\in\Re$ the corresponding \emph{shifted moment class} is
\begin{equation}\label{eq:shifted-class}
  \mc C_\mu \;\defn\; \set{\mb P_\mu}{\textstyle\int X\,\d\mb P_\mu=\mu,\,\int\phi(X-\mu)\,\d\mb P_\mu\le B}.
\end{equation}
The convex synthesis framework of Section~\ref{sec:synthesis} extends to any convex $\mc C_0\subseteq\mc P$, but we focus on the moment class \eqref{eq:moment-class-intro} since it allows for a matching lower bound in Section~\ref{sec:lower-bound}.

Given an error margin $\Delta>0$ we want to design an estimator $\mu_n$ that for any $n\geq 1$ guarantees
\begin{equation}
  \label{eq:confidence-guarantee}
  \max\bigl(\mb P_\mu\!\left[\mu_n-\mu>\Delta\right], \mb P_\mu\!\left[\mu_n-\mu<-\Delta\right]\bigr) \leq \beta \quad \forall \mu\in\Re,\ \mb P_\mu\in \mc C_\mu
\end{equation}
with the smallest possible confidence level $\beta=\beta_n(\Delta)$. By a simple union bound, Equation~(\ref{eq:confidence-guarantee}) then gives the symmetric error bound $\mb P_\mu\left[|\mu_n-\mu|>\Delta\right]\leq 2\beta$ for all $\mu\in\Re$ and $\mb P_\mu\in \mc C_\mu$.

As already contrasted in the abstract, two distinct notions of optimality run through this paper and need to be carefully distinguished. Holding the confidence level $\beta$ fixed and letting $n\to\infty$, the \emph{fixed-confidence} regime seeks the smallest error margin $\Delta_n(\beta)$ that~\eqref{eq:confidence-guarantee} permits; holding the margin $\Delta$ fixed instead, the \emph{fixed-margin} regime seeks the largest exponential rate at which the confidence level decays, $r^\star(\Delta)=\limsup_{n\to\infty}-\tfrac1n\log\beta_n(\Delta)$. These are two cross-sections of the same $(\Delta,\beta,n)$ trade-off; they weigh the dependence on $n$ differently, and an estimator optimal in one regime need not be optimal in the other (Section~\ref{sec:related}). Our central results concern this fixed-margin rate $r^\star(\Delta)$, for which we obtain matching upper and lower bounds on every symmetric moment class; for the concrete classes of Section~\ref{sec:illustrations-2} we show in addition that the same synthesis is fixed-confidence optimal: it attains the sharp high-confidence leading constant for the bounded-variance and bounded $\alpha$-moment classes, and is leading-order minimax at every fixed $\beta\in(0,\tfrac12)$ for slowly varying classes.

Our central result identifies this optimal rate exactly. Recall the two-point Hellinger exponent of the shifted classes,
\begin{equation}\label{eq:intro-hellinger}
  r(\Delta)\;\defn\;-\log\sup_{\mb P_{\pm\Delta}\in\mc C_{\pm\Delta}}\int\sqrt{\d\mb P_{-\Delta}\,\d\mb P_{\Delta}}.
\end{equation}
No measurable estimator achieves a faster rate, $r^\star(\Delta)\le r(\Delta)$ (Proposition~\ref{prop:LB}), and the rate $r(\Delta)$ is attained by a monotone $M$-estimator synthesized from a two-parameter convex optimization problem (Section~\ref{sec:synthesis}), whose worst-case confidence level obeys the non-asymptotic bound $e^{-nr(\Delta)}$ for every $n$ (Theorem~\ref{thm:matching}). Hence $r^\star(\Delta)=r(\Delta)$, and the synthesized estimator is exponential-rate optimal among \emph{all} estimators, not merely among $M$-estimators.

\subsection{Related Work}
\label{sec:related}

The statistical problem class is captured by the ambiguity set $\mc C_0$, which encodes the estimator's knowledge about the distribution $\mb P_0$ generating the data. The classes reviewed below differ only in this choice of $\mc C_0$.
Most results in the literature, including those surveyed in this section, concern the \emph{fixed-confidence} regime, phrasing optimality in terms of the smallest deviation $\Delta_n(\beta)$ achievable at a fixed confidence level $\beta$, whereas Sections~\ref{sec:synthesis} and~\ref{sec:lower-bound} work in the \emph{fixed-margin} regime. Our optimality results are therefore complementary to, rather than implied by, the existing fixed-$\beta$ bounds. Nevertheless, for the select classes of Section~\ref{sec:illustrations-2} the same synthesis recovers the existing $\Delta_n(\beta)$ guarantees at no extra cost: the sharp high-confidence leading constant for the bounded-variance and bounded $\alpha$-moment classes, and fixed-$\beta$ leading-order minimaxity for the slowly varying classes.

The benchmark throughout is the Gaussian model $\mb P_0=N(0,\sigma^2)$, where the empirical mean is optimal and achieves the sub-Gaussian margin $\sqrt2\,\sigma\sqrt{\log(1/\beta)/n}$ (Section~\ref{sec:bound-vari-class}). Following \citet{devroye2015subgaussianmeanestimators}, the question is how far $\mc C_0$ can be enlarged beyond $\{\mb P_0\}$ while still admitting an estimator whose deviations match this benchmark to leading order. The next two classes give striking answers.

\subsubsection{Bounded Variance Class}
\label{sec:catoni}

Surprisingly, the normality of the distribution $\mb P_0$ is not required to allow for estimators with normal error margins.
\citet{catoni2012challenging} drops the parametric assumption and indeed considers all distributions with known mean and bounded variance,
\begin{equation}
  \label{def:catoni-model}
  \mc C_0\defn \set{\mb P\in \mc P}{\textstyle\int X \d \mb P = 0,~\int |X|^2 \d \mb P \leq \sigma^2}.
\end{equation}
This class allows for arbitrarily heavy tails and may contain distributions with $\E{\mb P_0}{\exp(sX)}=\infty$ for every $s\neq 0$, so the empirical mean is provably suboptimal at any nontrivial confidence level. Catoni's seminal contribution is to identify a family of non-decreasing estimating functions sandwiched, for a free scale parameter $\lambda>0$, as
\begin{equation}
  \label{eq:catoni-butterfly}
  \psi_{2, l}(x):=-\log\!\left(1-x/\lambda + x^2/(2\lambda^2)\right) \leq \psi(x) \leq \log\!\left(1+x/\lambda + x^2/(2\lambda^2)\right)=:\psi_{2, u}(x).
\end{equation}

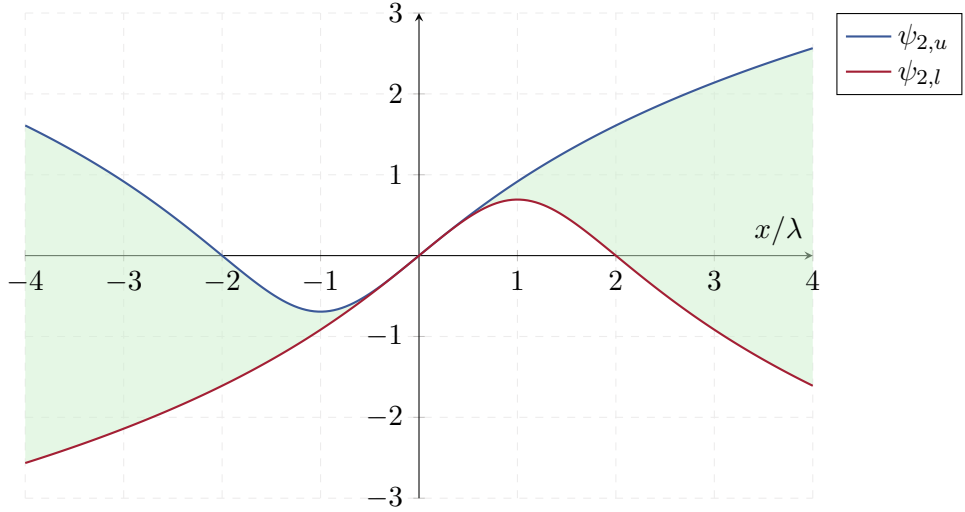
\begin{figure}[h]
  \centering
\begin{tikzpicture}[trim axis left, trim axis right]
  \begin{axis}[
        width=12cm,
        height=8cm,
  axis lines=middle,
  xlabel={$x/\lambda$},
  domain=-4:4,
  samples=200,
  xmin=-4, xmax=4,
  ymin=-3, ymax=3,
  legend pos=outer north east,
  legend cell align=left,
  grid=major,
  grid style={dashed,gray!30},
  ]

  \addplot[thick, blue, name path=upper] {ln(1+x+x^2/2)};
  \addlegendentry{$\psi_{2, u}$}

  \addplot[thick, red, name path=lower] {-ln(1-x+x^2/2)};
  \addlegendentry{$\psi_{2, l}$}
  
  \addplot[fill=green!20, opacity=0.5] fill between[of=upper and lower];

\end{axis}
\end{tikzpicture}
  \caption{\citet{catoni2012challenging} considers non-decreasing estimating functions $\psi$ in the green region sandwiched between two enveloping functions $\psi_{2, l}$ and $\psi_{2, u}$.}
  \label{fig:catoni-characteristic-functions}
\end{figure}

Every estimating function $\psi$ obeying \eqref{eq:catoni-butterfly} is sandwiched between two envelopes of logarithmic growth, so its asymptotic slope is far gentler than the linear $\psi(x)=x$ of the empirical mean; see Figure \ref{fig:catoni-characteristic-functions}.
\citet{catoni2012challenging} shows that any such $\psi$, with scale parameter $\lambda$ tuned to $\beta$ and $n$, produces an $M$-estimator achieving sub-Gaussian deviations under the variance constraint alone. Concretely, for any $\beta\in(0,1)$ the resulting estimator $\mu_n$ satisfies
\begin{equation}
  \label{eq:catoni-deviation}
  \max(\mb P_0\!\left[\mu_n > \Delta_n\right], \mb P_0\!\left[\mu_n < -\Delta_n\right] )  \leq \beta \quad \forall \mb P_0\in \mc C_0,
\end{equation}
with error margin $\Delta_n = \sqrt{2}\cdot\sigma\sqrt{\log(1/\beta)/(n-2\log(1/\beta))}$ for all $ n>2\log(1/\beta)$ \citep[Proposition~2.4]{catoni2012challenging}. For $n\to\infty$ this matches the error margin required in the normal location subfamily and hence can not be improved for $\beta\downarrow 0$.

Two classical alternatives to Catoni's $M$-estimator achieve the same sub-Gaussian rate under the variance constraint alone. The \emph{median-of-means} estimator $\mu_n^{\mathrm{mom}}$, going back to \citet{nemirovsky1983problem,jerrum1986random,alon1999space}, splits the sample into $k$ blocks, averages within each, and returns the median of the block means; with $k=\lceil 8\log(1/\beta)\rceil$ blocks it satisfies~\eqref{eq:confidence-guarantee} with $\Delta_n=\sigma\sqrt{32\log(1/\beta)/n}$ \citep[Theorem~2]{lugosi2019mean}. The \emph{trimmed mean} \citep{lugosi2021trimmed} instead discards the most extreme observations before averaging; in the variant of \citet[Theorem~6]{lugosi2019mean} it achieves $|\mu_{2n}^{\mathrm{trim}}-\mu|\le 9\sigma\sqrt{\log(8/\beta)/n}$ with probability at least $1-\beta$. Both match the $\sigma\sqrt{\log(1/\beta)/n}$ rate of the Gaussian benchmark above, and so are rate-optimal, but with suboptimal leading constants. In exchange, and unlike Catoni's estimator, neither of these alternative estimators requires knowledge of $\sigma$. Adaptivity to $\sigma$ need not cost the sharp \emph{leading} constant. \citet[Theorem~1]{leevaliant2021optimal} construct an estimator that, requiring no knowledge of the variance, attains the Gaussian margin $\sqrt2\,\sigma\sqrt{\log(1/\beta)/n}$ up to a $1+o(1)$ factor, thereby recovering Catoni's sharp constant $\sqrt2$ without the constant loss of the median-of-means and trimmed-mean estimators. The price of adaptivity surfaces only in lower-order terms: their multiplicative slack is guaranteed to vanish only asymptotically, across the moderate-deviation regime $1\ll\log(1/\beta)\ll n$ ($\beta\downarrow0$ sub-exponentially in $n$) in which the sub-Gaussian constant $\sqrt2$ is meaningful, whereas Catoni attains $\sqrt2$ through the explicit non-asymptotic margin~\eqref{eq:catoni-deviation} valid for every $n>2\log(1/\beta)$. All estimators do depend on the confidence level $\beta$, and this dependence is not an artifact: both \citet{devroye2015subgaussianmeanestimators} and \citet{lugosi2019mean} show that, absent further distributional information, an estimator achieving sub-Gaussian deviations must be tuned to $\beta$, no single estimator being simultaneously optimal over nontrivial ranges of $\beta$.

\subsubsection{Bounded \texorpdfstring{$\alpha$}{alpha}-th Moment Class}
\label{sec:Bhatt-model}

A natural extension is to replace the variance by an $\alpha$-th absolute moment, for $\alpha\in(1,2)$:
\begin{equation}
  \label{def:bhatt-model}
  \mc C_0 \defn \set{\mb P\in \mc P}{\textstyle\int X \d \mb P = 0,~\int \abs{X}^\alpha \d \mb P\leq \sigma^\alpha}.
\end{equation}
Three almost-concurrent works \citep{lee2020optimal,chen2021generalized,bhatt2022nearly} address this problem class by proposing a non-decreasing $\psi$ obeying, for a free constant $\tilde C>0$, the sandwich
\begin{equation}
  \label{eq:samorodnitsky-butterfly}
  \psi_{\alpha, l}(x) = -\log\!\left(1-x/\lambda + \tilde C \abs{x/\lambda}^\alpha\right) \leq \psi(x) \leq \log\!\left(1+x/\lambda + \tilde C \abs{x/\lambda}^\alpha\right) = \psi_{\alpha, u}(x),
\end{equation}
which recovers Catoni's sandwich \eqref{eq:catoni-butterfly} as $\alpha\to 2$ with $\tilde C\to 1/2$. The three works differ only in their choice of $\tilde C$: \citet{chen2021generalized} sets $\tilde C = 1/\alpha$, while \citet{lee2020optimal} and \citet{bhatt2022nearly} both take the value $C(\alpha) \defn \bigl(\tfrac{(\alpha-1)}{\alpha}\bigr)^{\alpha/2} \bigl(\tfrac{(2-\alpha)}{(\alpha-1)}\bigr)^{(2-\alpha)/2}$, which Section~\ref{sec:bounded-alpha-moment-1} shows to be the rate-optimal choice within the family.

Each such $\psi$ achieves the heavy-tailed margin order $\sigma\bigl(\log(1/\beta)/n\bigr)^{(\alpha-1)/\alpha}$, with a leading constant that depends on $\tilde C$. A lower bound of \citet{devroye2015subgaussianmeanestimators} matches this order but with a smaller constant, leaving the order optimal but the sharp constant open. The smallest constant in the family is $L(\alpha)$, attained at $C(\alpha)$; the precise margins and this optimal constant are taken up in Section~\ref{sec:bounded-alpha-moment-1}.

As in the variance case, the median-of-means estimator extends to this heavier-tailed setting. \citet{bubeck2013bandits} and \citet{devroye2015subgaussianmeanestimators} show that, with $k=\lceil 8\log(2/\beta)\rceil$ blocks, $\mu_n^{\mathrm{mom}}$ satisfies~\eqref{eq:confidence-guarantee} with $\Delta_n = 8\sigma\bigl(12\,\log(1/\beta)/n\bigr)^{(\alpha-1)/\alpha}$, again matching the optimal $\bigl(\log(1/\beta)/n\bigr)^{(\alpha-1)/\alpha}$ rate, though with a suboptimal leading constant. On the other hand, $\mu_n^{\mathrm{mom}}$ is independent of both the moment exponent $\alpha$ and the scale $\sigma$, depending only on the confidence $\beta$ through the block count $k$. This is unlike the $M$-estimators of \citet{lee2020optimal,bhatt2022nearly}, and unlike the optimized estimating functions we synthesize below, whose envelopes are tailored to $\alpha$ and $\sigma$.

\subsection{Contributions}
\label{ssec:contributions}

Our main contributions are the following.

\begin{enumerate}
  \item \textbf{A Convex Perspective.} We aim to design good $M$-estimators through their estimating function $\psi$; finding the one with the best exponential rate at error margin $\Delta$ is a convex optimization problem (Section~\ref{sec:synthesis}). By Lagrangian duality the infinite-dimensional search over $\psi$ collapses to a two-parameter semi-infinite convex optimization problem in the multipliers $(\lambda_1,\lambda_2)$, whose value is the best exponential rate achievable by a monotone $M$-estimator. The reformulation is not merely \emph{algorithmic} but also \emph{analytical}: convexity unlocks the standard toolbox (Lagrangian duality, KKT stationarity, strong duality, saddle-point arguments) on what was previously an ad-hoc search, and it is these tools that produce the structural and optimality results that follow.
  \item \textbf{Sandwich Structure.} The KKT conditions of the synthesis problem force every optimized monotone $\psi$ into a sandwich $\psi_l \le \psi \le \psi_u$ of log-tilted form (Section~\ref{sec:synthesis}). This recovers, from first principles, the structural shape of the estimators proposed ad hoc by \citet{catoni2012challenging}, \citet{lee2020optimal,chen2021generalized,bhatt2022nearly}; the dual multipliers parametrizing the envelopes are now \emph{optimized} rather than guessed.
  \item \textbf{Matching Exponential Rate (Main Contribution).} For symmetric moment classes we pin down the optimal fixed-margin rate $r^\star(\Delta)=r(\Delta)$ exactly (Section~\ref{sec:lower-bound}): the rate of the synthesized $M$-estimator coincides with the Chernoff exponent of a two-point Bayes test over the class, the exact large-deviation form of Le Cam's classical two-point lower bound, which in turn upper-bounds the rate of \emph{every} measurable estimator. The synthesized estimator is therefore exponential-rate optimal among all estimators, not merely among $M$-estimators. This is the main result of the paper.
  \item \textbf{Fixed-Confidence Optimality for Concrete Classes.} For the concrete classes of Section~\ref{sec:illustrations-2} the fixed-margin matching transfers to the fixed-confidence regime. For the bounded-variance and bounded $\alpha$-moment classes the synthesized $M$-estimator attains the optimal error margin $\Delta_n(\beta)$ with its sharp high-confidence leading constant; for the slowly varying classes it is leading-order minimax at every fixed confidence $\beta\in(0,\tfrac12)$. We recover Catoni's $\sqrt2$ for the bounded-variance class (Section~\ref{sec:bound-vari-class}), establish the constant $L(\alpha)$ of \citet{lee2020optimal,bhatt2022nearly} as tight for the bounded $\alpha$-moment class with $\alpha\in(1,2)$ (Section~\ref{sec:bounded-alpha-moment-1}), and, pushing beyond polynomial tails, treat the slowly varying moment classes (Section~\ref{sec:slow-growth}), where the optimal rate degenerates to sub-polynomial and the error margin decays only poly- or iterated-logarithmically in $n$.
\end{enumerate}

\paragraph{Relation to prior work.} The closest precedents are \citet{catoni2012challenging} for the bounded-variance class and \citet{lee2020optimal,chen2021generalized,bhatt2022nearly} for the bounded $\alpha$-moment class. Both lines of work \emph{posit} a sandwich-shaped one-parameter family of estimating functions and tune the scaling for the target confidence $(\beta,n)$; both yield only upper bounds on the achievable rate / margin. Our contribution is twofold relative to this prior work: (i) the sandwich structure is \emph{derived} from KKT rather than imposed, and the resulting envelopes are parametrized by two multipliers $(\lambda_1,\lambda_2)$ optimized jointly: Catoni's variance envelope coincides with a particular suboptimal slice of this family, while the $\alpha$-moment envelopes match no slice yet are still dominated by our optimized rate; (ii) for the bounded $\alpha$-moment class we prove a \emph{matching} lower bound via a Skellam/Bessel asymptotic, showing that the leading constant $L(\alpha)$ of \citet{lee2020optimal,bhatt2022nearly} is in fact tight. As with Catoni's and Bhatt's $\psi$'s, the optimal multipliers depend on the target $(\beta,n)$ regime; the convex optimization problem produces the right $(\lambda_1,\lambda_2)$ for any such target.

Our work is also close in spirit to a growing literature which views minimax estimation through a convex duality lens. \citet{polyanskiy2026dualizing} dualize Le Cam's two-point method: by convex duality and a minimax theorem the search for the tightest two-point lower bound becomes a minimization over estimators whose saddle value characterizes the minimax rate, for linear functionals strengthening the modulus-of-continuity bounds of \citet{donoho1991geometrizing}. As they observe, this continues the broader program of \citet{juditsky2020statistical} of recasting statistical inference as convex optimization. These works characterize the minimax \emph{rate} only up to universal constant factors, under quadratic risk; we instead fix the location functional and a confidence criterion and obtain the \emph{exact} large-deviation exponent. Convexity of the distribution class is the common enabler, and \citet{compton2025attainability} show what its absence costs: over the \emph{non-convex} family of translates of a fixed shape-constrained distribution, the two-point testing rate for location is attainable only up to polylogarithmic factors, and for merely symmetric shapes not at all. The classical attainability of the Hellinger modulus indeed requires a convex class \citep{donoho1991geometrizing}, precisely what our moment constraint supplies.

\paragraph{Context.} While our synthesis is phrased over the class of monotone $M$-estimators, the matching lower bound of Section~\ref{sec:lower-bound} applies to \emph{every} measurable estimator: no procedure, inside or outside the $M$-estimator family, achieves a strictly faster exponential rate on a symmetric moment class. This includes the median-of-means and trimmed-mean estimators. As discussed these estimators are \emph{order}-optimal; in the bounded variance regime, for instance, they attain the sharp $\sigma\sqrt{\log(1/\beta)/n}$ scaling of the error margin, but with a {suboptimal leading constant}. What they concede in the constant they gain in adaptivity. The median-of-means and trimmed-mean require no knowledge of the moment bound $B$ or of the exponent $\alpha$, their tuning depending only on the target confidence and sample size, whereas our envelopes are synthesized from a \emph{known} $(\phi,B)$ at a prescribed margin $\Delta$ or confidence level $\beta$. The two approaches are thus complementary rather than competing, trading exact constants on a known class against order-optimal adaptivity to an unknown one. 

\section{Optimized $M$-Estimators}
\label{sec:synthesis}

By an \emph{estimator} we always mean a statistic, i.e., a measurable map $\mu_n:\Re^n\to\Re$ of the sample $(X_1,\dots,X_n)$, where $\Re$ carries its Borel $\sigma$-algebra $\mc B(\Re)$ and the simplex $\mc P$, the set of Borel probability measures on $(\Re,\mc B(\Re))$, is equipped with the topology of weak convergence. 

In this section we aim to design good $M$-estimators by searching over their estimating function $\psi$. This search turns out to be a convex optimization problem. The prior constructions of \cite{catoni2012challenging,bhatt2022nearly,lee2020optimal} all impose a sandwich $\psi_l^\lambda\le\psi\le\psi_u^\lambda$ between log-envelopes parametrized by a single scalar $\lambda$. Our convex analysis perspective will reveal this sandwich shape as a \emph{consequence}, not an assumption. 

\subsection{A Convex Perspective}
\label{ssec:setup}

A classical choice is to consider an $M$-estimate $\mu_n$ with non-decreasing (and hence measurable) estimating function $\psi$. Writing $\psi_n(\theta):=\sum_{i=1}^n\psi(X_i-\theta)$, which is non-increasing in $\theta$, the associated $M$-estimator $\mu_n$ is any location at which $\psi_n$ changes sign:
\begin{equation}
  \label{eq:M-estimate}
  \psi_n(\theta)\;\ge\;0 \ \ \text{for } \theta<\mu_n, \qquad \psi_n(\theta)\;\le\;0 \ \ \text{for } \theta>\mu_n .
\end{equation}
This coincides with the usual estimating equation $\psi_n(\mu_n)=0$ whenever $\psi_n$ has an exact root; the question of when such a $\mu_n$ exists, and its non-uniqueness, is taken up immediately after Theorem~\ref{thm:feasibility-guarantee}. Such $M$-estimators contain the empirical mean as the special case $\psi(x)=x$, and admit efficient computation by bisection.
Whereas in previous work the search for an appropriate function $\psi$ was dominated by a certain number of ad hoc design choices such as the sandwich structure and restricted to a univariate parametric family a much more general approach can be considered. Indeed, the Chernoff bound argument employed in \cite{catoni2012challenging}, \cite{bhatt2022nearly} and \cite{lee2020optimal} allows searching over non-decreasing estimating functions $\psi$ directly.

\begin{theorem}[Convex feasibility implies the confidence guarantee]\label{thm:feasibility-guarantee}
Let $\mc C_0$ be the symmetric moment class~\eqref{eq:moment-class-intro}, $\Delta>0$, $r\in\Re$, and non-decreasing $\psi:\Re\to\Re$. If
\begin{subequations}\label{eq:symmetric-feasibility}
\begin{align}
  \sup_{\mb P\in\mc C_0}\log\E{\mb P}{\exp(\psi(X-\Delta))} \,+\, r &\;\le\; 0, \label{eq:sf-right}\\
  \sup_{\mb P\in\mc C_0}\log\E{\mb P}{\exp(-\psi(X+\Delta))} \,+\, r &\;\le\; 0, \label{eq:sf-left}
\end{align}
\end{subequations}
then any measurable $M$-estimator $\mu_n$ obeying the estimating equation~\eqref{eq:M-estimate}, should one exist, satisfies for every $n\ge 1$,
\begin{equation}
  \label{eq:feasibility-guarantee}
  \max\bigl(\mb P_\mu\!\left[\mu_n-\mu>\Delta\right],\, \mb P_\mu\!\left[\mu_n-\mu<-\Delta\right]\bigr)  \,\le\, e^{-nr}\qquad\text{for all }\mu\in\Re\text{ and }\mb P_\mu\in\mc C_\mu.
\end{equation}
\end{theorem}
\begin{proof}
Fix $\mu\in\Re$ and $\mb P_\mu\in\mc C_\mu$, and set $Y_i\defn X_i-\mu$. By the definition of $\mc C_\mu$, the common law $\mb Q$ of the $Y_i$ has $\E{\mb Q}{Y}=0$ and $\E{\mb Q}{\phi(Y)}\le B$, so $\mb Q\in\mc C_0$. Since $\mu_n$ obeys the sign-change condition~\eqref{eq:M-estimate}, $\mu_n>\mu+\Delta$ implies $\psi_n(\mu+\Delta)=\sum_{i=1}^n\psi(X_i-\mu-\Delta)=\sum_{i=1}^n\psi(Y_i-\Delta)\geq 0$, so $\mb P_\mu[\mu_n-\mu>\Delta]\le\mb P_\mu\bigl[\sum_{i=1}^n\psi(Y_i-\Delta)\ge 0\bigr]$.

The right-tail constraint~\eqref{eq:sf-right} controls the moment generating function of a single increment: since $\mb Q\in\mc C_0$, it gives $\log\E{\mb Q}{\exp(\psi(Y-\Delta))}\le -r$, and exponentiating yields $\E{\mb Q}{\exp(\psi(Y-\Delta)+r)}\le 1$. Form the nonnegative process
\[
  M_k^+\;\defn\;\exp\Bigl(\textstyle\sum_{i=1}^k\bigl(\psi(Y_i-\Delta)+r\bigr)\Bigr),\qquad M_0^+\defn 1,
\]
so that $\{\sum_{i=1}^n\psi(Y_i-\Delta)\ge0\}=\{M_n^+\ge e^{nr}\}$. By independence, $\E{\mb Q}{M_n^+}=\bigl(\E{\mb Q}{\exp(\psi(Y-\Delta)+r)}\bigr)^{n}\le 1$, so Markov's inequality gives
\[
  \mb P_\mu[\mu_n-\mu>\Delta]\;\le\;\mb P_\mu[M_n^+\ge e^{nr}]\;\le\;e^{-nr}\,\E{\mb Q}{M_n^+}\;\le\;e^{-nr}.
\]
The left-tail bound is symmetric, with $-\psi(\,\cdot\,+\Delta)$ and~\eqref{eq:sf-left} in place of $\psi(\,\cdot\,-\Delta)$ and~\eqref{eq:sf-right}.
\end{proof}

\begin{remark}[Martingale data]\label{rmk:martingale}
  The feasibility constraints~\eqref{eq:symmetric-feasibility} are a condition on $\psi$ and $r$ alone; the i.i.d.\ assumption on the data enters the proof only through $\E{\mb Q}{M_n^+}=\prod_{i=1}^n\E{\mb Q}{\exp(\psi(Y_i-\Delta)+r)}\le1$. Suppose instead that, for a filtration $(\mathcal F_k)_{k\ge0}$ to which $(X_k)$ is adapted, the data satisfy the weaker martingale property
  \[
    \E{}{X_k\mid\mathcal F_{k-1}}=\mu\qquad\text{and}\qquad\E{}{\phi(X_k-\mu)\mid\mathcal F_{k-1}}\le B\quad\text{almost surely}.
  \]
  Then~\eqref{eq:symmetric-feasibility}, being suprema over $\mc C_0$, implies the conditional bounds $\E{}{\exp(\pm\psi(Y_k\mp\Delta)+r)\mid\mathcal F_{k-1}}\le1$; hence $M_k^+$ and its mirror $M_k^-\defn\exp\bigl(\sum_{i=1}^k(-\psi(Y_i+\Delta)+r)\bigr)$ are nonnegative supermartingales started at $M_0^\pm=1$, and the guarantee~\eqref{eq:feasibility-guarantee} extends verbatim.
\end{remark}

Two remarks on the estimator $\mu_n$ are in order. First, the proof uses~\eqref{eq:M-estimate} only through the inclusion $\{\mu_n>\Delta\}\subseteq\{\psi_n(\Delta)\ge0\}$ (and its left-tail mirror), so the guarantee holds for whichever sign-change location is selected when $\psi_n$ vanishes on an interval or jumps across zero. Second, the theorem is conditional in that it presupposes at least one such $\mu_n$ exists.
Since $\psi_n$ is non-increasing, a sign change exists as soon as $\psi_n$ takes both signs. A clean sufficient condition is the classical requirement
\begin{equation}
  \label{ass:huber}
  \psi(-\infty)<0<\psi(+\infty)
\end{equation}
proposed by \citet[Section 3.2.2]{huber1981robust}. This condition delivers not merely a sign change but an estimator of the kind Theorem~\ref{thm:feasibility-guarantee} requires. Indeed, under~\eqref{ass:huber} we have $\psi_n(-\infty)>0>\psi_n(+\infty)$, so $\{\theta\in\Re:\psi_n(\theta)\le0\}$ is a nonempty proper sub-interval of $\Re$ and its left endpoint
\begin{equation}
  \label{eq:canonical-estimator}
  \widehat\mu_n\defn\inf\Bigl\{\theta\in\Re:\ \textstyle\sum_{i=1}^n\psi(X_i-\theta)\le0\Bigr\}
\end{equation}
is finite and obeys~\eqref{eq:M-estimate}. This canonical selection is Borel measurable, since $\{\widehat\mu_n<c\}=\bigcup_{q\in\mathbb Q,\,q<c}\{\psi_n(q)\le0\}$ is a countable union of measurable events. Thus~\eqref{ass:huber} guarantees a measurable $M$-estimator to which Theorem~\ref{thm:feasibility-guarantee} applies. Observe that the constraints~\eqref{eq:symmetric-feasibility} are jointly convex in $(\psi, r)$ although clearly infinite dimensional.
We characterize the largest rate achievable with a non-decreasing estimating function as the convex optimization problem
\begin{equation}
  \label{eq:r_M}
  r_M(\Delta) = \sup \set{r\in\Re}{\exists \psi~{\rm{ non-decreasing}} ~\st~ \eqref{eq:symmetric-feasibility}}.
\end{equation}
Clearly, as we may take $\psi\equiv0$ it follows that $r_M(\Delta)\geq 0$ for all $\Delta$. For this choice, any estimator satisfies (\ref{eq:M-estimate}) trivially. Any estimating function $\psi$ which fails condition \eqref{ass:huber} is feasible in \eqref{eq:symmetric-feasibility} only for $r\leq 0$ and hence dominated by the trivial choice $\psi\equiv0$. Hence, in the nontrivial case $r_M(\Delta)>0$ we are assured that there exists a sequence of estimating functions $\psi_k$, each satisfying \eqref{ass:huber}, enjoying \eqref{eq:feasibility-guarantee} with $r_k\uparrow r_M(\Delta)$.

Informally, the conditions~\eqref{eq:symmetric-feasibility} are not just sufficient but essentially tight: a Sanov-type large-deviation perspective would suggest that the largest feasible $r$ matches the true worst-case exponential rate of the corresponding error event of $\mu_n$. We do not make this precise as a rigorous Sanov argument requires the mapping $\mb P \to \E{\mb P}{\psi(X)}$ to be continuous and hence is restricted by the Portmanteau theorem to bounded and continuous estimating functions. None of our results require such topological considerations. Section~\ref{sec:lower-bound} establishes the stronger and fully rigorous statement that the synthesized rate is optimal among \emph{all} measurable estimators, not just $M$-estimators.

\subsection{Sandwich Structure}
\label{ssec:sandwich}

Rather than search the infinite-dimensional space of estimating functions $\psi$ defining $r_M(\Delta)$ in~\eqref{eq:r_M}, we carry out a finite-dimensional search over a pair of dual multipliers $(\lambda_1,\lambda_2)$.
Define the upper and lower envelopes
\begin{equation}\label{eq:envelopes}
  \psi_u^\lambda(x) \defn \log\bigl(1+\lambda_1(x+\Delta)+\lambda_2(\phi(x+\Delta)-B)\bigr),\quad
  \psi_l^\lambda(x) \defn -\log\bigl(1-\lambda_1(x-\Delta)+\lambda_2(\phi(x-\Delta)-B)\bigr),
\end{equation}
parametrized by the single pair $(\lambda_1,\lambda_2)\in\Re\times\Re_+$. These envelopes mimic the log-tilted sandwich posited ad hoc by \citet{catoni2012challenging} in~\eqref{eq:catoni-butterfly} and by \citet{lee2020optimal,bhatt2022nearly} in~\eqref{eq:samorodnitsky-butterfly}, now with two multipliers $(\lambda_1,\lambda_2)$ in place of a single scale, to be optimized rather than guessed. It is convenient to name the affine arguments of the two logarithms,
\(
g_u^\lambda(x)\defn 1+\lambda_1 x+\lambda_2(\phi(x)-B)
\)
and
\(
  g_l^\lambda(x)\defn 1-\lambda_1 x+\lambda_2(\phi(x)-B),
\)
so that $\psi_u^\lambda(x)=\log g_u^\lambda(x+\Delta)$ and $\psi_l^\lambda(x)=-\log g_l^\lambda(x-\Delta)$. As $x$ ranges over $\Re$, both envelopes are real-valued at every point exactly when $(\lambda_1,\lambda_2)$ lies in the \emph{admissible set}
\begin{equation}\label{eq:admissible-set}
  C\ \defn\ \set{(\lambda_1,\lambda_2)\in\Re\times\Re_+}{g_u^\lambda(x)>0, ~g_l^\lambda(x)>0 \ \text{for all }x\in\Re}.
\end{equation}
\begin{lemma}[Weak Lagrangian duality]\label{lem:lagrangian-dual}
Let $L:\Re\to\Re$ be measurable and bounded below. Then
\[
  \sup_{\mb P\in \mc C_0} \E{\mb P}{L(X)}\;\le\; \inf_{\lambda_1\in\Re,\,\lambda_2\geq 0}\,\sup_{x\in\Re}\bigl(L(x)-\lambda_1 x-\lambda_2(\phi(x)-B)\bigr),
\]
with both sides taking values in $(-\infty,+\infty]$.
\end{lemma}

The proof is deferred to Appendix~\ref{app:deferred}.

We may apply Lemma~\ref{lem:lagrangian-dual} with $L(x)=\exp(\psi(x-\Delta)+r)$, which is nonnegative and measurable since the non-decreasing $\psi$ is. If there exist $(\lambda_1,\lambda_2)\in\Re\times\Re_+$ with
\[
  \exp(\psi(x-\Delta)+r)-\lambda_1 x-\lambda_2(\phi(x)-B) \,\leq\, 1 \qquad\forall x\in\Re,
\]
equivalently $\psi(x)+r\le \psi_u^\lambda(x)$ for every $x\in\Re$, then the displayed pointwise bound makes the inner supremum in Lemma~\ref{lem:lagrangian-dual} at most $1$, so  $\sup_{\mb P\in\mc C_0}\E{\mb P}{\exp(\psi(X-\Delta)+r)}\le 1$ and the right-tail constraint~\eqref{eq:sf-right} holds at rate $r$. The same reasoning applied to $L(x)=\exp(-\psi(x+\Delta)+r)$ shows that a pointwise lower envelope $\psi_l^{\lambda}(x) \leq \psi(x)-r$ implies~\eqref{eq:sf-left} at rate $r$. Thus exhibiting multipliers $(\lambda_1,\lambda_2)\in C$ for which
\[
  \psi_l^\lambda(x)+r\;\le\;\psi(x)\;\le\;\psi_u^\lambda(x)-r\qquad\text{for all }x\in\Re
\]
implies \eqref{eq:symmetric-feasibility}.
Denote by $C_\Delta$ the set of pairs $(r,\lambda)\in\Re\times C$ admitting such a non-decreasing $\psi$.
We may hence define
\begin{align}
  r_M'(\Delta)  = &\sup_{(r,\lambda)\in C_\Delta}r \label{eq:r_M_prime} \\
  =&\sup\set{r\in\Re}{\begin{array}{l}\exists \lambda \in C, ~\exists \psi~{\rm{non-decreasing}}:\\
    \psi_l^\lambda(x)+r \,\leq\, \psi(x)\,\leq\, \psi_u^\lambda(x)-r~~ \forall x\in\Re 
  \end{array}}.\nonumber
\end{align}
\begin{lemma}[Finite-dimensional representation]\label{lem:representation}
The synthesis rate~\eqref{eq:r_M_prime} is a finite optimization over the admissible pair $(\lambda_1,\lambda_2)\in C$ alone,
\begin{equation}\label{eq:r_M_pairwise}
  r_M'(\Delta)\;=\;\sup\set{r\in\Re}{\exists\,(\lambda_1,\lambda_2)\in C:\ g_l^\lambda(x-\Delta)\,g_u^\lambda(y+\Delta)\ge e^{2r}\ \ \forall\,x\le y}.
\end{equation}
\end{lemma}
\begin{proof}
The estimating function can be eliminated from~\eqref{eq:r_M_prime}. For $(\lambda_1,\lambda_2)\in C$, a non-decreasing $\psi$ obeying $\psi_l^\lambda(x)+r\le\psi(x)\le\psi_u^\lambda(x)-r$ for all $x$ exists if and only if
\begin{equation}\label{eq:pairwise-feasibility}
  \psi_l^\lambda(x)+r\ \le\ \psi_u^\lambda(y)-r\qquad\text{for all }x\le y .
\end{equation}
Necessity is immediate, as any non-decreasing interpolant has $\psi_l^\lambda(x)+r\le\psi(x)\le\psi(y)\le\psi_u^\lambda(y)-r$ whenever $x\le y$; for sufficiency, $\psi(y)\defn\sup_{x\le y}\bigl(\psi_l^\lambda(x)+r\bigr)$ is non-decreasing, dominates $\psi_l^\lambda+r$ (take $x=y$), and is dominated by $\psi_u^\lambda-r$ by~\eqref{eq:pairwise-feasibility}. Exponentiating~\eqref{eq:pairwise-feasibility} gives the equivalent
\begin{equation}\label{eq:product-feasibility}
  g_l^\lambda(x-\Delta)\,g_u^\lambda(y+\Delta)\ \ge\ e^{2r}\qquad\text{for all }x\le y .
\end{equation}
Thus $(r,\lambda)\in C_\Delta$ if and only if $(\lambda_1,\lambda_2)\in C$ and~\eqref{eq:product-feasibility} holds; substituting into~\eqref{eq:r_M_prime} gives~\eqref{eq:r_M_pairwise}.
\end{proof}

The representation~\eqref{eq:r_M_pairwise} is immediately recognized as a finite convex optimization problem. Introducing a scalar variable $s$, maximizing $r$ in~\eqref{eq:r_M_pairwise} is equivalent to maximizing $r$ over $(\lambda_1,\lambda_2)\in C$, $s\ge0$ subject to the single exponential-cone constraint $e^{r}\le s$ and the semi-infinite family of rotated second-order cone (hyperbolic) constraints $s^2\le g_l^\lambda(x-\Delta)\,g_u^\lambda(y+\Delta)$, $x\le y$ (which are convex because $g_u^\lambda,g_l^\lambda$ are affine in $\lambda$). 

Taking $(\lambda_1,\lambda_2)=(0,0)\in C$ (whence $g_u^0\equiv g_l^0\equiv1$) and $r=0$ satisfies~\eqref{eq:product-feasibility}, so $r_M'(\Delta)\ge0$. By construction $r_M'(\Delta)\le r_M(\Delta)$, and the matching lower bound of Section~\ref{sec:lower-bound} will show that in fact $r_M'(\Delta)=r_M(\Delta)$. It remains to show that the supremum~\eqref{eq:r_M_pairwise} is attained; the description~\eqref{eq:r_M_pairwise} reduces this to a short compactness argument over $C$.

\begin{lemma}\label{lem:attainment}
If $r_M'(\Delta)<\infty$, then the supremum~\eqref{eq:r_M_pairwise} is attained at multipliers $\lambda^\star\in C$, equivalently $(r_M'(\Delta),\lambda_1^\star,\lambda_2^\star)\in C_\Delta$.
\end{lemma}
\begin{proof}[Proof of Lemma~\ref{lem:attainment}]
For $(\lambda_1,\lambda_2)\in C$ we have $g_u^\lambda(x)>0$ for all $x$. At $x=0$, $g_u^\lambda(0)=1-\lambda_2(B-\phi(0))>0$ gives, using the standing slack $\phi(0)<B$, $0\le\lambda_2<\bar\lambda_2:=(B-\phi(0))^{-1}$; at $x=\pm1$, with $\phi(-1)=\phi(1)$, the conditions $g_u^\lambda(\pm1)>0$ read $|\lambda_1|<1+\lambda_2|\phi(1)-B|\le1+\bar\lambda_2|\phi(1)-B|=:\bar\lambda_1$. Hence $C\subseteq[-\bar\lambda_1,\bar\lambda_1]\times[0,\bar\lambda_2)$, with compact closure $K:=[-\bar\lambda_1,\bar\lambda_1]\times[0,\bar\lambda_2]\subset\Re\times\Re_+$.
Choose feasible $(r_k,\lambda^k)$ with $r_k\uparrow r_M'(\Delta)$, so $\lambda^k\in C\subseteq K$; along a subsequence $\lambda^k\to\lambda^\star\in K$. The arguments $g_u^\lambda,g_l^\lambda$ are affine, hence continuous, in $\lambda$, and each $\lambda^k\in C$, so for every $x\le y$ the positive values $g_l^{\lambda^k}(x-\Delta),g_u^{\lambda^k}(y+\Delta)$ converge to $g_l^{\lambda^\star}(x-\Delta),g_u^{\lambda^\star}(y+\Delta)\ge0$. Passing~\eqref{eq:product-feasibility} to the limit, with $r_k\to r_M'(\Delta)$, gives
\[
  g_l^{\lambda^\star}(x-\Delta)\,g_u^{\lambda^\star}(y+\Delta)\ \ge\ e^{2r_M'(\Delta)}>0\qquad\text{for all }x\le y .
\]
Taking $y=x$, the two nonnegative factors have strictly positive product, hence are strictly positive; as $x$ is arbitrary, $g_u^{\lambda^\star}(x)>0$ for all $x$, i.e.\ $\lambda^\star\in C$. Therefore $(r_M'(\Delta),\lambda^\star)$ is feasible and the supremum~\eqref{eq:r_M_pairwise} is attained; equivalently, by Lemma~\ref{lem:representation}, $(r_M'(\Delta),\lambda_1^\star,\lambda_2^\star)\in C_\Delta$.
\end{proof}

The estimating function enters~\eqref{eq:r_M_prime} only through the sandwich inequality, and the sandwich structure imposed ad hoc in Equations \eqref{eq:catoni-butterfly} and \eqref{eq:samorodnitsky-butterfly} thus appears naturally in our framework: whenever $r_M'(\Delta)<\infty$, Lemma~\ref{lem:attainment} furnishes an optimiser $(r_M'(\Delta),\lambda_1^\star,\lambda_2^\star)\in C_\Delta$, and any non-decreasing $\psi$ satisfying the sandwich
\[
  \psi_l^{\lambda^\star}(x)+r_M'(\Delta)\le \psi(x)\le\psi_u^{\lambda^\star}(x)-r_M'(\Delta)
\]
achieves the exponential rate $r_M'(\Delta)$.
Our envelope functions are structurally similar to those of \citet{catoni2012challenging} and \citet{bhatt2022nearly}, but carry two multipliers rather than one, and those multipliers are optimized rather than chosen ad hoc.
We note that the sandwich~\eqref{eq:envelopes} pins down $\psi$ only at the points where the upper and lower envelopes \emph{kiss}, i.e.\ where $\psi_u^{\lambda^\star}(x)-r_M'(\Delta) = \psi_l^{\lambda^\star}(x)+r_M'(\Delta)$; elsewhere any non-decreasing interpolant between the two envelopes is equally optimal. This non-uniqueness is not an artefact: the support of a pair of least-favorable distributions identified in Section~\ref{sec:lower-bound} is contained in the kissing points, so on that support $\psi$ is already pinned. Off the support $\psi$ must still respect the envelope sandwich, but the precise interpolant chosen does not affect the worst-case rate.

It remains to check that the optimal envelopes yield a genuine estimator. By the discussion following~\eqref{ass:huber}, this holds as soon as the synthesized $\psi$ obeys Huber's sign condition, which the following corollary guarantees in the nontrivial case.

\begin{corollary}\label{cor:estimator-exists}
If $r_M'(\Delta)>0$, then at the optimal multipliers $(\lambda_1^\star,\lambda_2^\star)$ of Lemma~\ref{lem:attainment} every non-decreasing $\psi$ feasible in~\eqref{eq:r_M_prime} satisfies Huber's condition~\eqref{ass:huber}.
\end{corollary}
\begin{proof}
Write $r\defn r_M'(\Delta)>0$. Such a $\psi$ witnesses $(r,\lambda_1^\star,\lambda_2^\star)\in C_\Delta$ and hence satisfies the feasibility constraints~\eqref{eq:symmetric-feasibility} at rate $r$. If $\psi\le0$ everywhere then $e^{-\psi(x+\Delta)}\ge1$ for all $x$, so $\max_{\mb P\in\mc C_0}\E{\mb P}{e^{-\psi(X+\Delta)}}\ge1>e^{-r}$, contradicting the left-tail constraint~\eqref{eq:sf-left}; hence $\psi$ takes a strictly positive value and, being non-decreasing, $\psi(+\infty)>0$. The right-tail constraint~\eqref{eq:sf-right} forces $\psi(-\infty)<0$ symmetrically.
\end{proof}

\section{Matching Lower Bound}
\label{sec:lower-bound}

We establish two results: a two-point hypothesis-testing lower bound on the achievable exponential rate (Proposition~\ref{prop:LB}), and a matching theorem showing that this rate is attained by the synthesized $M$-estimator of Section~\ref{ssec:sandwich} (Theorem~\ref{thm:matching}). Together they imply that for symmetric moment classes the synthesized $M$-estimator is exponential-rate optimal among all measurable estimators.

\subsection{A Two-point Hypothesis Testing Lower Bound}
\label{ssec:two-point-LB}

We show by a two-point hypothesis testing argument that the best achievable exponential rate by \emph{any} measurable estimator is upper bounded by the value of an optimization problem in the same form as the convex synthesis problem~\eqref{eq:symmetric-feasibility}. The reduction to testing two indistinguishable distributions in the class is the classical \emph{two-point method}, formalized through Le Cam's two-point bound and the modulus-of-continuity calculus of \citet{donoho1991geometrizing}.
For any estimator $\mu_n$ define the worst-case confidence level
\begin{equation}\label{eq:beta-n-def}
  \beta_n(\mu_n;\Delta) \defn \max\!\left\{\sup_{\mu\in\Re,\,\mb P_\mu\in \mc C_\mu}\!\!\mb P_\mu\!\left[\mu_n>\mu+\Delta\right],\,
    \sup_{\mu\in\Re,\,\mb P_\mu\in \mc C_\mu}\!\!\mb P_\mu\!\left[\mu_n<\mu-\Delta\right]\right\}.
\end{equation}
For nonnegative measures $\mb P_-,\mb P_+$ on $\Re$ their mirror Hellinger affinity is
\begin{equation}\label{eq:rho-def}
  \rho(\mb P_-,\mb P_+)\defn\int\sqrt{\d\mb P_-\,\d\mb P_+}.
\end{equation}
We allow $\mb P_-,\mb P_+$ to be arbitrary nonnegative measures, not only probability measures, because the strong-duality relaxation below dualizes the unit-mass constraint together with the moment constraints and so evaluates $\rho$ over sub-probability measures.
For any $\Delta>0$, recall the two-point problem~\eqref{eq:intro-hellinger} defining $r(\Delta)$ as minus the logarithm of the supremum of the mirror affinity $\rho$ of~\eqref{eq:rho-def} over pairs $(\mb P_{-\Delta},\mb P_\Delta)\in\mc C_{-\Delta}\times\mc C_\Delta$; we call such a pair \emph{feasible}.

The two suprema of interest, in~\eqref{eq:intro-hellinger} and in its relaxation below, are both over pairs drawn from convex classes that are interchanged by reflection. The following reduction, used repeatedly, lets us restrict such a supremum to mirror pairs.

\begin{proposition}[Reduction to mirror pairs]\label{prop:mirror}
Let $\mc D_-,\mc D_+$ be convex sets of nonnegative measures on $\Re$ that are \emph{reflection-conjugate}: $\bar{\mb P}\defn\mb P\circ(-\mathrm{Id})\in\mc D_\mp$ whenever $\mb P\in\mc D_\pm$. Then the supremum $$\sup\set{\rho(\mb P_-,\mb P_+)}{\mb P_-\in\mc D_-,\,\mb P_+\in\mc D_+}$$ is unchanged when restricted to \emph{mirror pairs}, those with $\mb P_+=\bar{\mb P}_-$. Concretely, every feasible pair $(\mb P_-,\mb P_+)$ is dominated in affinity by the feasible mirror pair $\bigl(\tfrac12(\mb P_-+\bar{\mb P}_+),\,\tfrac12(\mb P_++\bar{\mb P}_-)\bigr)$.
\end{proposition}
\begin{proof}
The affinity is invariant under the common bijection $x\mapsto-x$, so $\rho(\bar{\mb P}_+,\bar{\mb P}_-)=\rho(\mb P_-,\mb P_+)$, and $(\bar{\mb P}_+,\bar{\mb P}_-)\in\mc D_-\times\mc D_+$ by reflection-conjugacy. Since $\mc D_-,\mc D_+$ are convex, the average $\mb Q_-\defn\tfrac12(\mb P_-+\bar{\mb P}_+)\in\mc D_-$ and $\mb Q_+\defn\tfrac12(\mb P_++\bar{\mb P}_-)\in\mc D_+$ is feasible; it is a mirror pair, since $\bar{\mb Q}_-=\tfrac12(\bar{\mb P}_-+\mb P_+)=\mb Q_+$; and by joint concavity of $\rho$,
\(
\rho(\mb Q_-,\mb Q_+)\;\ge\;\tfrac12\rho(\mb P_-,\mb P_+)+\tfrac12\rho(\bar{\mb P}_+,\bar{\mb P}_-)\;=\;\rho(\mb P_-,\mb P_+).
\)
\end{proof}

\begin{proposition}\label{prop:LB}
For any measurable estimator $\mu_n$ and any $\Delta>0$,
\[
  \liminf_{n\to\infty} \frac{1}{n}\log\beta_n(\mu_n;\Delta) \,\geq\, -\,r(\Delta).
\]
\end{proposition}

\begin{proof}
Fix $\Delta'>\Delta$ and any $\mb P_{-\Delta'}\in\mc C_{-\Delta'}$, $\mb P_{+\Delta'}\in\mc C_{+\Delta'}$. Since $\Delta'>\Delta$ we have $-\Delta'+\Delta<0<\Delta'-\Delta$ and consequently
\[
  \{\mu_n\geq 0\}\subseteq\{\mu_n>-\Delta'+\Delta\},\qquad
  \{\mu_n<0\}\subseteq\{\mu_n<\Delta'-\Delta\}.
\]
Let $\alpha_n\defn\mb P_{-\Delta'}[\mu_n\ge 0]$ and $\gamma_n\defn\mb P_{\Delta'}[\mu_n<0]$ denote the type-I and type-II errors of the test $\phi_n=\mathbf 1\{\mu_n\ge 0\}$ between the hypotheses $H_-:X_1,\dots,X_n\overset{iid}{\sim}\mb P_{-\Delta'}$ and $H_+:X_1,\dots,X_n\overset{iid}{\sim}\mb P_{\Delta'}$. Specializing the suprema in the definition of $\beta_n(\mu_n;\Delta)$ to $(\mu,\mb P)=(-\Delta',\mb P_{-\Delta'})$ in the right-tail term and to $(\mu,\mb P)=(\Delta',\mb P_{\Delta'})$ in the left-tail term yields
\begin{equation}
\label{eq:LB-reduction}
  \beta_n(\mu_n;\Delta)\,\geq\,\max\!\left\{\mb P_{-\Delta'}[\mu_n\geq 0],\,\mb P_{\Delta'}[\mu_n<0]\right\}=\max(\alpha_n,\gamma_n).
\end{equation}

\emph{Bayes-risk bound.} For two probability measures $\mb P_-,\mb P_+$ on $\Re$ and a sample size $n\ge 1$, denote by
\begin{equation}\label{eq:bayes-error-def}
  \mc R^\star_n(\mb P_-,\mb P_+) \,\defn\, \inf_{\phi:\Re^n\to\{0,1\}}\tfrac12\Bigl(\mb P_-[\phi=1]+\mb P_+[\phi=0]\Bigr)
\end{equation}
the Bayes error of testing $n$ independent observations from $\mb P_-$ versus $\mb P_+$ under a uniform prior, i.e.\ the smallest weighted error attainable by any test.
Since $\mc R^\star_n$ is an infimum over all tests and $\phi_n=\mathbf 1\{\mu_n\ge 0\}$ is one particular test, we get the lower bound
\[
\beta_n(\mu_n;\Delta)\;\geq\; \max(\alpha_n,\gamma_n)\;\geq\;\tfrac12(\alpha_n+\gamma_n)\;\geq\; \mc R^\star_n(\mb P_{-\Delta'},\mb P_{\Delta'})
\]

\emph{Chernoff rate.} \citet{chernoff1952} shows that the Bayes error of the optimal test decays exponentially with exponent the Chernoff information,
\[
  \lim_{n\to\infty} -\tfrac{1}{n}\log\mc R^\star_n(\mb P_{-\Delta'},\mb P_{\Delta'})\,=\, C(\mb P_{-\Delta'},\mb P_{\Delta'})\defn -\log\inf_{s\in[0,1]}\int \d\mb P_{-\Delta'}^{1-s}\,\d\mb P_{\Delta'}^{s}.
\] Combined with the Bayes-risk bound $\max(\alpha_n,\gamma_n)\ge\mc R^\star_n(\mb P_{-\Delta'},\mb P_{\Delta'})$, this gives
\begin{equation}
  \label{eq:chernoff-lb}
  \liminf_{n\to\infty} \tfrac{1}{n}\log\max(\alpha_n,\gamma_n)\,\geq\, -C(\mb P_{-\Delta'},\mb P_{\Delta'}).
\end{equation}

\emph{Lower bound.} By~\eqref{eq:LB-reduction} and~\eqref{eq:chernoff-lb}, $\liminf_{n\to\infty}\tfrac1n\log\beta_n(\mu_n;\Delta)\ge -C(\mb P_{-\Delta'},\mb P_{\Delta'})$ for every feasible pair $\mb P_{\pm\Delta'}\in\mc C_{\pm\Delta'}$; optimizing this bound over all feasible pairs gives $\liminf_{n\to\infty}\tfrac1n\log\beta_n(\mu_n;\Delta)\ge -\inf_{\mb P_{\pm\Delta'}\in\mc C_{\pm\Delta'}}C(\mb P_{-\Delta'},\mb P_{\Delta'})$. Minimizing the Chernoff information over all feasible pairs is, by symmetrization, the same as minimizing $-\log\rho$ over all feasible pairs. Indeed, taking $s=\tfrac12$ in the infimum defining $C$ gives $C\ge-\log\rho$ for every pair, so $\inf C\ge-\log\sup\rho$. For the reverse inequality, note that since $\phi$ is even the reflection $\mb P\mapsto\mb P\circ(-\mathrm{Id})$ maps $\mc C_{\pm\Delta'}$ onto $\mc C_{\mp\Delta'}$, so the convex classes $\mc C_{-\Delta'},\mc C_{\Delta'}$ are reflection-conjugate and Proposition~\ref{prop:mirror} shows that $\sup\rho$ is unchanged when restricted to mirror pairs $\mb P_+=\mb P_-\circ(-\mathrm{Id})$. On such a pair the infimum defining $C$ is attained at $s=\tfrac12$: the function $s\mapsto Z(s)\defn\int\d\mb P_-^{1-s}\,\d\mb P_+^{s}$ is convex and, since $x\mapsto-x$ interchanges $\mb P_-$ and $\mb P_+$, satisfies $Z(s)=Z(1-s)$, so $Z(\tfrac12)\le\tfrac12 Z(s)+\tfrac12 Z(1-s)=Z(s)$ for every $s\in[0,1]$, giving $\inf_{s\in[0,1]}Z(s)=Z(\tfrac12)=\rho(\mb P_-,\mb P_+)$ and hence $C(\mb P_-,\mb P_+)=-\log\rho(\mb P_-,\mb P_+)$. Minimizing $C$ over mirror pairs therefore gives $\inf C\le-\log\sup\rho$. Hence
\[
  \inf_{\mb P_{\pm\Delta'}\in\mc C_{\pm\Delta'}}C(\mb P_{-\Delta'},\mb P_{\Delta'})\;=\;-\log\!\!\sup_{\mb P_{\pm\Delta'}\in\mc C_{\pm\Delta'}}\!\!\rho(\mb P_{-\Delta'},\mb P_{\Delta'})\;=\;r(\Delta'),
\]
and therefore $\liminf_{n\to\infty}\tfrac1n\log\beta_n(\mu_n;\Delta)\ge -r(\Delta')$ for every $\Delta'>\Delta$, whence $$\liminf_{n\to\infty}\tfrac1n\log\beta_n(\mu_n;\Delta)\ge -\inf_{\Delta'>\Delta}r(\Delta').$$

\emph{Removing the auxiliary margin.} It remains to show $\inf_{\Delta'>\Delta}r(\Delta')\le r(\Delta)$, so that the bound becomes $-r(\Delta)$. Fix $\eta\in(0,e^{-r(\Delta)})$. By Lemma~\ref{lem:compact-strict} there is a compactly supported pair $\mb P_{\pm\Delta}\in\mc C_{\pm\Delta}$ with strict moment constraints $\int\phi(x+\Delta)\,\d\mb P_{-\Delta}<B$ and $\int\phi(x-\Delta)\,\d\mb P_\Delta<B$ and affinity $\rho(\mb P_{-\Delta},\mb P_\Delta)\ge e^{-r(\Delta)}-\eta$. For $h>0$ put $t_h=(\Delta+h)/\Delta>1$ and $S_{t_h}(x)=t_h x$. By Lemma~\ref{lem:scaling} the common pushforward $\bigl((S_{t_h})_*\mb P_{-\Delta},(S_{t_h})_*\mb P_\Delta\bigr)$ is feasible in~\eqref{eq:intro-hellinger} at level $t_h\Delta=\Delta+h$ for all $h>0$ small enough, and leaves the affinity unchanged. Hence $e^{-r(\Delta+h)}\ge e^{-r(\Delta)}-\eta$, i.e.\ $r(\Delta+h)\le-\log(e^{-r(\Delta)}-\eta)$, for all small $h>0$, so $\inf_{\Delta'>\Delta}r(\Delta')\le-\log(e^{-r(\Delta)}-\eta)$; letting $\eta\downarrow0$ gives $\inf_{\Delta'>\Delta}r(\Delta')\le r(\Delta)$ and therefore $\liminf_{n\to\infty}\tfrac1n\log\beta_n(\mu_n;\Delta)\ge -r(\Delta)$.
\end{proof}

\subsection{Zero Suboptimality Gap}
\label{ssec:matching}

We now show that the lower-bound rate $r(\Delta)$ of Proposition~\ref{prop:LB} coincides with the synthesis rates $r_{\mathrm{M}}'(\Delta)$ and $r_{\mathrm{M}}(\Delta)$ of Section~\ref{ssec:sandwich} for symmetric moment classes. Proposition~\ref{prop:LB} already gives the upper bound $r_{\mathrm{M}}(\Delta)\le r(\Delta)$: any feasible $(\psi,r)$ in~\eqref{eq:symmetric-feasibility} yields, through Theorem~\ref{thm:feasibility-guarantee}, an $M$-estimator with $\beta_n(\mu_n;\Delta)\le e^{-nr}$ and hence decay rate at least $r$, while Proposition~\ref{prop:LB} caps the decay rate of every measurable estimator at $r(\Delta)$; thus $r\le r(\Delta)$, and the supremum over feasible rates gives $r_{\mathrm{M}}(\Delta)\le r(\Delta)$. The matching theorem supplies the reverse.

\begin{theorem}[Matching]\label{thm:matching}
For the symmetric moment class~\eqref{eq:moment-class-intro} and every $\Delta>0$,
\[
  r(\Delta) \;\leq\; r_{\mathrm{M}}'(\Delta) \;\leq\; r_{\mathrm{M}}(\Delta) \;\leq\; r(\Delta).
\]
\end{theorem}

\begin{proof}
The bridge is a \emph{relaxed} lower-bound problem, easier to manipulate because its constraints are one-sided inequalities rather than equalities, which is what lets us identify dual multipliers with the synthesis envelopes. Relaxing both the mean equality and the unit-mass normalization defines
\begin{equation}
  \label{eq:rstar-relax}
  r_{\mathrm{relax}}(\Delta) \;:=\; -\log\sup\set{\rho(\mb P_{-\Delta},\mb P_\Delta)}{\mb P_{\pm\Delta}\in\mc C_{\pm\Delta}^{\mathrm{relax}}},
\end{equation}
where, over nonnegative measures,
\begin{align*}
  \mc C_{-\Delta}^{\mathrm{relax}}&:=\Bigl\{\mb P\ge0:\ \textstyle\int\d\mb P\le1,\ \int(X+\Delta)\,\d\mb P\le0,\ \int\bigl(\phi(X+\Delta)-B\bigr)\,\d\mb P\le0\Bigr\},\\
  \mc C_{\Delta}^{\mathrm{relax}}&:=\Bigl\{\mb P\ge0:\ \textstyle\int\d\mb P\le1,\ \int(X-\Delta)\,\d\mb P\ge0,\ \int\bigl(\phi(X-\Delta)-B\bigr)\,\d\mb P\le0\Bigr\}.
\end{align*}
On a unit-mass $\mb P$ these reduce to $\E{\mb P}{X}\le-\Delta$ (resp.\ $\ge\Delta$) and $\E{\mb P}{\phi(X\mp\Delta)}\le B$, the constraints of $\mc C_{\pm\Delta}$; the relaxation widens them to sub-probability measures with one-sided inequalities, the three constraints dualized in Step~2. The proof then proceeds in three steps.

\paragraph{Step 1: $r(\Delta) = r_{\mathrm{relax}}(\Delta)$.}

\begin{proposition}[No relaxation gap]\label{prop:relax-gap}
$r_{\mathrm{relax}}(\Delta) = r(\Delta)$.
\end{proposition}

The proof is deferred to Appendix~\ref{app:deferred}.

\paragraph{Step 2: KKT recovers the synthesis sandwich envelopes globally.}
Write $r:=r_{\mathrm{relax}}(\Delta)=r(\Delta)$ (Proposition~\ref{prop:relax-gap}).

\emph{Reduction to the Hellinger affinity.} By the definition~\eqref{eq:rstar-relax} of $r_{\mathrm{relax}}$ and Proposition~\ref{prop:relax-gap} (which identifies it with $r$),
\begin{equation}
  \label{eq:hellinger-max}
  e^{-r}\;=\;\max\set{\rho(\mb P_-,\mb P_+)}{\mb P_-\in\mc C^{\mathrm{relax}}_{-\Delta},\ \mb P_+\in\mc C^{\mathrm{relax}}_\Delta}.
\end{equation}
\emph{Least-favorable pair.} The maximum in~\eqref{eq:hellinger-max} is attained. Let $(\mb P^k_{-\Delta},\mb P^k_\Delta)_k$ be a maximizing sequence, so $\rho_k\defn\rho(\mb P^k_{-\Delta},\mb P^k_\Delta)\to e^{-r}$. By the unit-mass reduction of Step~1, we may take each $\mb P^k_{\pm\Delta}$ to be a \emph{probability} measure with $\E{\mb P^k_{-\Delta}}{X}\le-\Delta$, $\E{\mb P^k_\Delta}{X}\ge\Delta$ and $\E{\mb P^k_{\pm\Delta}}{\phi(X\mp\Delta)}\le B$.
The super-linearity of $\phi$ turns the uniform moment bound $\E{\mb P^k_{\pm\Delta}}{\phi(X\mp\Delta)}\le B$ into uniform control of the tails. Writing $\eta(M)\defn\sup_{|x|>M}|x|/\phi(x\mp\Delta)$, super-linearity $\phi(y)/|y|\to\infty$ gives $\eta(M)\to0$ as $M\to\infty$, so by Markov's inequality, for every $k$ and both signs,
\begin{equation}\label{eq:tight-ui-markov}
  \E{\mb P^k_{\pm\Delta}}{|X|\,\mathbf 1\{|X|>M\}}\;\le\;\eta(M)\,\E{\mb P^k_{\pm\Delta}}{\phi(X\mp\Delta)}\;\le\;\eta(M)\,B\;\xrightarrow[M\to\infty]{}\;0
\end{equation}
uniformly in $k$. Thus $\{\mb P^k_{\pm\Delta}\}_k$ is uniformly integrable, and a fortiori tight ($\mb P^k_{\pm\Delta}(|X|>M)\le\eta(M)B/M\to0$). By Prokhorov's theorem~\citep[Theorem~5.1]{billingsley1999} a subsequence converges weakly, $\mb P^k_{\pm\Delta}\rightharpoonup\mb P^\star_{\pm\Delta}$, the limit again a probability measure. The limit is feasible: $\phi$ is continuous and non-negative, so the moment functional is weakly lower semicontinuous (portmanteau), giving $\E{\mb P^\star_{\pm\Delta}}{\phi(X\mp\Delta)}\le\liminf_{k\to\infty}\E{\mb P^k_{\pm\Delta}}{\phi(X\mp\Delta)}\le B$; and the uniform integrability~\eqref{eq:tight-ui-markov} upgrades weak convergence to convergence of the mean~\citep[Theorem~3.5]{billingsley1999}, $\E{\mb P^k_{\pm\Delta}}{X}\to\E{\mb P^\star_{\pm\Delta}}{X}$, whence $\E{\mb P^\star_{-\Delta}}{X}\le-\Delta$ and $\E{\mb P^\star_\Delta}{X}\ge\Delta$. Thus $(\mb P^\star_{-\Delta},\mb P^\star_\Delta)\in\mc C^{\mathrm{relax}}_{-\Delta}\times\mc C^{\mathrm{relax}}_\Delta$. Finally, the affinity is the Rényi divergence of order $\tfrac12$ through $\rho=e^{-D_{1/2}(\mb P_-\|\mb P_+)/2}$ \citep[eq.~(5)]{vanerven2014renyi}, and $D_{1/2}$ is weakly lower semicontinuous in the pair of measures \citep[Theorem~19]{vanerven2014renyi}, so $\rho$ is weakly upper semicontinuous, as also shown directly in \cite[Theorem 2]{coquet1977mesure}; thus $\rho(\mb P^\star_{-\Delta},\mb P^\star_\Delta)\ge\limsup_{k\to\infty}\rho_k=e^{-r}$.
We call such a maximiser a \emph{least-favorable pair}.

\emph{Global sandwich via duality.} The multipliers and the global envelope separation come together from the Lagrangian dual of the maximization \eqref{eq:hellinger-max}, whose three constraints (mass, mean and moment) are all one-sided inequalities. Attaching multipliers $\alpha_\pm,\beta_\pm\ge0$ to the mean and moment constraints and $\nu_\pm\ge0$ to the mass constraints $\int\d\mb P_\pm\le1$, and writing
\begin{align*}
  & \tilde g_-(x;\alpha_-,\beta_-,\nu_-)=\nu_-+\alpha_-(x+\Delta)+\beta_-(\phi(x+\Delta)-B),\\
  & \tilde g_+(x;\alpha_+,\beta_+,\nu_+)=\nu_++\alpha_+(\Delta-x)+\beta_+(\phi(x-\Delta)-B),
\end{align*}
the Lagrangian is
\[
  L(\mb P_-,\mb P_+;\,\alpha_\pm,\beta_\pm,\nu_\pm)\;=\;\int\bigl[\sqrt{\d\mb P_-\,\d\mb P_+}-\tilde g_-\,\d\mb P_- -\tilde g_+\,\d\mb P_+\bigr]+\nu_-+\nu_+,
\]
whose supremum over $\mb P_\pm\ge0$ defines the dual function
\[
  h(\alpha_\pm,\beta_\pm,\nu_\pm)\;\defn\;\sup_{\mb P_-,\mb P_+\ge0}L(\mb P_-,\mb P_+;\,\alpha_\pm,\beta_\pm,\nu_\pm).
\]

\begin{lemma}\label{lem:amgm-dual}
  For $\tilde g_-,\tilde g_+\in\Re$,
  \[
    \sup_{p_-,p_+\geq 0}\bigl(\sqrt{p_{-} p_{+}} - \tilde g_- p_- -\tilde g_+ p_+\bigr) =
    \begin{cases}
      0 & \text{if } \tilde g_-,\tilde g_+\ge0 \text{ and } \tilde g_-\,\tilde g_+\ge\tfrac14,\\
      +\infty & \text{otherwise.}
    \end{cases}
  \]
\end{lemma}
\begin{proof}
If $\tilde g_-<0$, taking $p_-\to\infty$ with $p_+=0$ makes $-\tilde g_-p_-\to+\infty$, and symmetrically if $\tilde g_+<0$; so assume $\tilde g_-,\tilde g_+\ge0$. The best constant in the arithmetic--geometric mean inequality is $$\inf_{p_-,p_+>0}\tfrac{(\tilde g_-p_-+\tilde g_+p_+)}{\sqrt{p_-p_+}}=2\sqrt{\tilde g_-\tilde g_+}.$$ If $\tilde g_-\tilde g_+\ge\tfrac14$ this is $\ge1$, so $\sqrt{p_-p_+}\le\tilde g_-p_-+\tilde g_+p_+$ for all $p_\pm\ge0$; the expression is then $\le0$, attaining $0$ at $p_-=p_+=0$. If $\tilde g_-\tilde g_+<\tfrac14$ the constant is $<1$, so $\sqrt{p_-p_+}>\tilde g_-p_-+\tilde g_+p_+$ at some $(p_-,p_+)$, and scaling that point up sends the expression to $+\infty$ by $1$-homogeneity.
\end{proof}

By Lemma~\ref{lem:amgm-dual} applied pointwise, the dual function simplifies to
\[
  h(\alpha_\pm,\beta_\pm,\nu_\pm)\;=\;\begin{cases}\nu_-+\nu_+ & \text{if }\ \tilde g_-(x),\tilde g_+(x)\ge0\ \text{and}\ \tilde g_-(x)\,\tilde g_+(x)\ge\tfrac14\ \ \forall x\in\Re,\\[2pt] +\infty & \text{otherwise.}\end{cases}
\]
Indeed, if $\tilde g_\pm(x)\ge0$ and $\tilde g_-(x)\,\tilde g_+(x)\ge\tfrac14$ for all $x$, then $\sqrt{\d\mb P_-\,\d\mb P_+}\le\tilde g_-\,\d\mb P_-+\tilde g_+\,\d\mb P_+$ pointwise by the arithmetic--geometric mean inequality of Lemma~\ref{lem:amgm-dual}, so the integral in $L$ is nonpositive and the supremum over $\mb P_\pm\ge0$ equals $\nu_-+\nu_+$, attained at $\mb P_\pm=0$. If instead $\tilde g_-(x_0)<0$ or $\tilde g_-(x_0)\,\tilde g_+(x_0)<\tfrac14$ at some $x_0$, then placing atoms $\mb P_\pm=p_\pm^\star\delta_{x_0}$ at a pair $(p_-^\star,p_+^\star)$ violating that inequality and scaling them up drives the supremum to $+\infty$ by $1$-homogeneity.

Since this no-gap identity is the crux of the matching theorem, we verify the hypotheses of the Lagrange duality theorem for convex optimization problems~\citep[Theorem~1, p.~224]{luenberger1969} rather than assert them. That theorem requires a convex domain $\Omega$, a convex objective with finite infimum, a constraint map convex with respect to the positive cone of a normed space $Z$ \emph{with nonempty interior}, and a Slater point at which every constraint is strict; under these it gives a zero duality gap with the dual optimum attained. We take $\Omega=\set{(\mb P_-,\mb P_+)}{\mb P_-,\mb P_+\ge0}$, so the positivity requirements are kept inside the domain and are not dualized; the constraints sent to $Z$ are then only the six scalar inequalities (two mass, two mean, two moment), giving $Z=\Re^6$. On this domain $f=-\rho$ is convex ($\rho$ jointly concave) with finite infimum $-e^{-r}\in[-1,0]$, and each of the six functionals is affine in $(\mb P_-,\mb P_+)$, so the constraint map is convex into $\Re^6$. For the Slater point, take $t>0$ and the pair $\mb P_-=\tfrac12\delta_{-(\Delta+t)}$, $\mb P_+=\tfrac12\delta_{\Delta+t}$: its masses are $\tfrac12<1$, its means $\int(X+\Delta)\,\d\mb P_-=-\tfrac t2<0$ and $\int(X-\Delta)\,\d\mb P_+=\tfrac t2>0$, and its moments $\int(\phi(X+\Delta)-B)\,\d\mb P_-=\int(\phi(X-\Delta)-B)\,\d\mb P_+=\tfrac12\bigl(\phi(t)-B\bigr)<0$ (as $\phi(0)<B$ and $\phi$ continuous), so all six constraints hold strictly. The hypotheses being met, we obtain
\begin{align*}
  & \max\set{\rho(\mb P_-,\mb P_+)}{\mb P_-\in\mc C^{\mathrm{relax}}_{-\Delta},\ \mb P_+\in\mc C^{\mathrm{relax}}_\Delta}\\
  = & \min_{\alpha_\pm,\beta_\pm,\nu_\pm\ge0}h(\alpha_\pm,\beta_\pm,\nu_\pm)\\
  = & \min\set{\nu_-+\nu_+}{\begin{array}{l}\exists \alpha_\pm,\beta_\pm,\nu_\pm\ge0 ~\st~ \tilde g_-(x;\alpha_-,\beta_-,\nu_-),\, \tilde g_+(x;\alpha_+,\beta_+,\nu_+)\ge0,\\ \tilde g_-(x;\alpha_-,\beta_-,\nu_-)\,\tilde g_+(x;\alpha_+,\beta_+,\nu_+)\ge\tfrac14\ \ \forall x\in\Re \end{array}}\\
  = & e^{-r}.
\end{align*}
The same theorem~\citep[Theorem~1, p.~224]{luenberger1969} furnishes dual attainment: the minimization over multipliers displayed above is achieved, not merely approached. Hence the dual is attained at an optimal multiplier tuple $(\alpha_\pm,\beta_\pm,\nu_\pm)\ge0$ with $\nu_-+\nu_+=e^{-r}$. Because $\phi$ is even, the reflection $x\mapsto-x$ interchanges $\tilde g_-$ and $\tilde g_+$ (indeed $\tilde g_-(-x;\alpha_-,\beta_-,\nu_-)=\tilde g_+(x;\alpha_-,\beta_-,\nu_-)$), so swapping the two triples $(\alpha_-,\beta_-,\nu_-)\leftrightarrow(\alpha_+,\beta_+,\nu_+)$ maps the constraint $\tilde g_-\tilde g_+\ge\tfrac14$ to itself and leaves the objective $\nu_-+\nu_+$ unchanged: the feasible set is invariant under this swap and the objective is too. The feasible set is also convex (each constraint $\tilde g_-(x)\tilde g_+(x)\ge\tfrac14$, $\tilde g_\pm(x)\ge0$ is convex in the multipliers, $\tilde g_\pm$ being affine in them). Averaging our optimal point with its swap therefore stays feasible and keeps the objective at $e^{-r}$ while making it symmetric: $\alpha_-=\alpha_+=:\alpha$, $\beta_-=\beta_+=:\beta$, $\nu_-=\nu_+=:\nu$, with $2\nu=e^{-r}$; in particular $\nu>0$. Write $\lambda=(\lambda_1,\lambda_2)$ with $\lambda_1=\alpha/\nu\ge0$ and $\lambda_2=\beta/\nu\ge0$. Factoring $\tilde g_-=\nu\,e^{\psi_u^\lambda}$ and $\tilde g_+=\nu\,e^{-\psi_l^\lambda}$ through the synthesis envelopes~\eqref{eq:envelopes}, the dual-feasibility constraint $\nu^2 e^{\psi_u^\lambda(x)-\psi_l^\lambda(x)}\ge\tfrac14$ reads, with $-\log(2\nu)=r$,
\begin{equation}\label{eq:global-sandwich}
  \psi_l^{\lambda}(x)+r\;\le\;\psi_u^{\lambda}(x)-r\qquad\forall x\in\Re.
\end{equation}
Thus the dual furnishes multipliers $(\lambda_1,\lambda_2)\ge0$ separating the envelopes by exactly $2r$.

\emph{Existence of a non-decreasing interpolant.} Since $\lambda_1\ge 0$ and $\phi$ is even and non-decreasing on $[0,\infty)$, the upper envelope $\psi_u^{\lambda}(x)=\log[1+\lambda_1(x+\Delta)+\lambda_2(\phi(x+\Delta)-B)]$ is non-decreasing on $[-\Delta,\infty)$, and by the same logic the lower envelope $\psi_l^{\lambda}$ is non-decreasing on $(-\infty,\Delta]$. Define
\[
  \psi(x) \;:=\; \begin{cases}\psi_l^{\lambda}(x)+r & x\le -\Delta,\\[2pt] \psi_u^{\lambda}(x)-r & x>-\Delta.\end{cases}
\]
Each piece is non-decreasing on its domain, and at the join the global sandwich~\eqref{eq:global-sandwich} gives $\psi_l^{\lambda}(-\Delta)+r\le \psi_u^{\lambda}(-\Delta)-r$, so $\psi$ is non-decreasing on $\Re$. By construction it satisfies $\psi_l^\lambda+r\le\psi\le\psi_u^\lambda-r$ pointwise.

\paragraph{Step 3: the LB multipliers lie in the synthesis feasibility set $C_\Delta$.}
By Step~2 the non-decreasing $\psi$ and the dual multipliers $(\lambda_1,\lambda_2)\ge 0$ furnished by~\eqref{eq:global-sandwich} satisfy
\[
  \psi_l^{\lambda}(x)\;\le\;\psi(x)-r\quad\text{and}\quad \psi(x)+r\;\le\;\psi_u^{\lambda}(x)\qquad\forall x\in\Re,
\]
which is exactly the membership condition of $C_\Delta$ in~\eqref{eq:r_M_prime} at rate $r$. Hence $(r,\lambda_1,\lambda_2)\in C_\Delta$, so $r_{\mathrm{M}}'(\Delta)\ge r=r_{\mathrm{relax}}(\Delta)$ and a fortiori $r_{\mathrm{M}}(\Delta)\ge r_{\mathrm{relax}}(\Delta)$. Combined with $r_{\mathrm{M}}'(\Delta)\le r_{\mathrm{M}}(\Delta)$, Step~1, and the upper bound $r_{\mathrm{M}}(\Delta)\le r(\Delta)$ from Proposition~\ref{prop:LB}, this completes the chain; in particular $r_{\mathrm{M}}'(\Delta)=r_{\mathrm{M}}(\Delta)=r(\Delta)$.
\end{proof}

Theorem~\ref{thm:matching} is the main result of the paper: for symmetric moment classes, the synthesized $M$-estimator is exponential-rate optimal among all measurable estimators. On the upper-bound side, the synthesis problem of Section~\ref{ssec:sandwich} produces an $M$-estimator $\mu_n$ (Corollary~\ref{cor:estimator-exists}) whose worst-case tail decays at rate $r_M'(\Delta)$,
\[
  \frac{1}{n}\log\,\sup_{\mu\in\Re}\sup_{\mb P_\mu\in\mc C_\mu}\max(\mb P_\mu[\mu_n-\mu>\Delta], \mb P_\mu[\mu_n-\mu<-\Delta]) \;\leq\; -r_M'(\Delta),
\]
for every $n\geq 1$. On the lower-bound side, Proposition~\ref{prop:LB} forces every measurable estimator $\widetilde\mu_n$ to satisfy
\[
  \liminf_{n\to\infty}\,\frac{1}{n}\log\,\sup_{\mu\in\Re}\sup_{\mb P_\mu\in\mc C_\mu}\max(\mb P_\mu[\widetilde\mu_n-\mu>\Delta], \mb P_\mu[\widetilde\mu_n-\mu<-\Delta]) \;\geq\; -r(\Delta).
\]
Thus the synthesized $M$-estimator is exponential-rate optimal among all estimators, and the two-parameter semi-infinite convex optimization problem~\eqref{eq:r_M_prime} characterizes that optimal rate.

Based on the proof of Theorem~\ref{thm:matching} one may construct a least-favorable pair $(\mb P^\star_{-\Delta},\mb P^\star_\Delta)\in\mc C_{-\Delta}\times\mc C_\Delta$ together with the optimal dual variables, and read off where the pair is supported. After the symmetrization of Step~2 the dual is a minimization over a single triple $(\alpha,\beta,\nu)\ge0$; with the dual-feasibility functions
\begin{align*}
  & g_-(x;\alpha,\beta,\nu)\defn\nu+\alpha(x+\Delta)+\beta(\phi(x+\Delta)-B),\\
  & g_+(x;\alpha,\beta,\nu)\defn\nu+\alpha(\Delta-x)+\beta(\phi(x-\Delta)-B),
\end{align*}
the symmetric dual reads
\[
  e^{-r}=\min\set{2\nu\!\!}{\!\!\alpha,\beta,\nu\ge0,\  g_-(x;\alpha,\beta,\nu),\,g_+(x;\alpha,\beta,\nu)\ge0,\  g_-(x;\alpha,\beta,\nu)\,g_+(x;\alpha,\beta,\nu)\ge\tfrac14\ \ \forall x\in\Re}.
\]
Let $(\alpha^\star,\beta^\star,\nu^\star)$ denote an optimal triple, necessarily with $2\nu^\star=e^{-r}$, write $\lambda^\star=(\lambda_1,\lambda_2)=(\alpha^\star/\nu^\star,\beta^\star/\nu^\star)$ for the induced envelope multipliers, and abbreviate $g^\star_\pm(x)\defn g_\pm(x;\alpha^\star,\beta^\star,\nu^\star)$, which factor through the synthesis envelopes as $g^\star_-=\nu^\star\,e^{\psi_u^{\lambda^\star}}$ and $g^\star_+=\nu^\star\,e^{-\psi_l^{\lambda^\star}}$. Because the optimal dual value $2\nu^\star=e^{-r}$ equals the optimal primal value, $(\mb P^\star_{-\Delta},\mb P^\star_\Delta)$ and the optimal triple $(\alpha^\star,\beta^\star,\nu^\star)$ form a saddle of the Lagrangian \citep[Theorem~1, p.~224]{luenberger1969}, so $(\mb P^\star_{-\Delta},\mb P^\star_\Delta)$ maximizes over the domain $\Omega=\set{(\mb P_-,\mb P_+)}{\mb P_-,\mb P_+\ge0}$ the Lagrangian $L$ evaluated at these optimal dual variables, that is (using $\tilde g_\pm=g^\star_\pm$ and $\nu_-+\nu_+=2\nu^\star$ at the symmetric optimum),
\[
  L(\mb P_-,\mb P_+;\alpha^\star,\beta^\star,\nu^\star)\;=\;\int\bigl[\sqrt{\d\mb P_-\,\d\mb P_+}-g^\star_-\,\d\mb P_- -g^\star_+\,\d\mb P_+\bigr]+2\nu^\star,
\]
with maximal value $2\nu^\star$. Fix a common dominating measure $\omega$ (for instance $\omega=\mb P^\star_{-\Delta}+\mb P^\star_\Delta$) and write $p^\star_\pm=\d\mb P^\star_{\pm\Delta}/\d\omega$, so the integrand at the optimum is $q(x)\defn\sqrt{p^\star_-(x)p^\star_+(x)}-g^\star_-(x)p^\star_-(x)-g^\star_+(x)p^\star_+(x)$. Dual feasibility forces $g^\star_-,g^\star_+>0$ and $g^\star_-g^\star_+\ge\tfrac14$, so by the arithmetic--geometric mean inequality (Lemma~\ref{lem:amgm-dual})
\[
  q(x)\;\le\;\sqrt{p^\star_-(x)p^\star_+(x)}-2\sqrt{g^\star_-(x)g^\star_+(x)}\,\sqrt{p^\star_-(x)p^\star_+(x)}\;=\;\sqrt{p^\star_-(x)p^\star_+(x)}\bigl(1-2\sqrt{g^\star_-(x)g^\star_+(x)}\bigr)\;\le\;0\quad\forall x\in\Re.
\]
Since $\mb P^\star_{\pm\Delta}$ attains the maximum $2\nu^\star$, the integral $\int q\,\d\omega=0$, and $q\le0$ then forces $q=0$ for $\omega$-almost every $x$. Let $N\defn\set{x\in\Re}{g^\star_-(x)g^\star_+(x)>\tfrac14}$ be the set where the contact condition is strict. Consider any $x\in N$ at which moreover $q(x)=0$; by the previous sentence this excludes only an $\omega$-null set of $x\in N$. There the factor $1-2\sqrt{g^\star_-(x)g^\star_+(x)}$ is strictly negative, so the displayed bound becomes the squeeze
\[
  0\;=\;q(x)\;\le\;\sqrt{p^\star_-(x)p^\star_+(x)}\,\bigl(1-2\sqrt{g^\star_-(x)g^\star_+(x)}\bigr)\;\le\;0,
\]
which, the factor being nonzero, is possible only if $\sqrt{p^\star_-(x)p^\star_+(x)}=0$. Substituting $\sqrt{p^\star_-(x)p^\star_+(x)}=0$ into $q(x)=\sqrt{p^\star_-(x)p^\star_+(x)}-g^\star_-(x)p^\star_-(x)-g^\star_+(x)p^\star_+(x)=0$ leaves $g^\star_-(x)p^\star_-(x)+g^\star_+(x)p^\star_+(x)=0$; both summands are nonnegative and $g^\star_\pm(x)>0$, so $p^\star_-(x)=p^\star_+(x)=0$. Thus $p^\star_\pm(x)=0$ for $\omega$-almost every $x\in N$, i.e.\ $\mb P^\star_{\pm\Delta}(N)=0$. This is complementary slackness: $\mb P^\star_{\pm\Delta}$ puts probability mass only on the contact set $\set{x\in\Re}{g^\star_-(x)\,g^\star_+(x)=\tfrac14}$. Since $g^\star_-(x)\,g^\star_+(x)=(\nu^\star)^2 e^{\psi_u^{\lambda^\star}(x)-\psi_l^{\lambda^\star}(x)}$ and $-\log(2\nu^\star)=r$, this contact condition is exactly where the upper and lower synthesis envelopes \emph{kiss}, i.e., $\psi_u^{\lambda^\star}(x)-r=\psi_l^{\lambda^\star}(x)+r$.
Figure~\ref{fig:kissing} depicts this for the bounded-variance class.

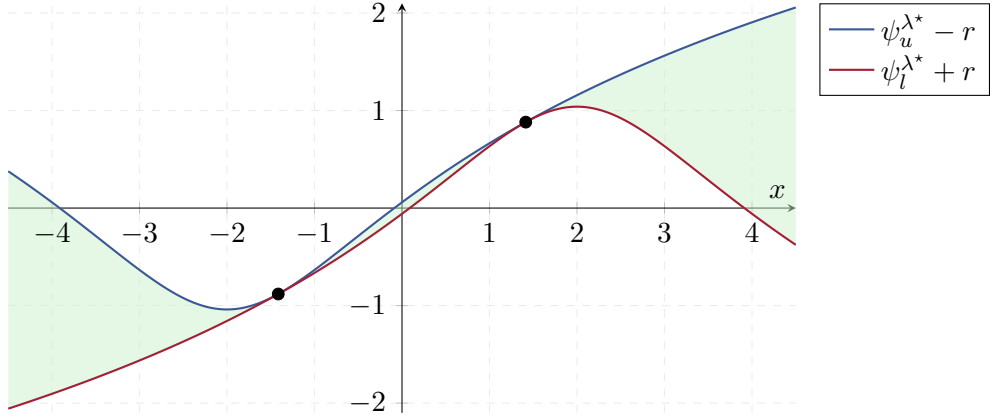
\begin{figure}[h]
  \centering
\begin{tikzpicture}[trim axis left, trim axis right]
  \pgfmathsetmacro{\sig}{1}
  \pgfmathsetmacro{\Dl}{1}
  \pgfmathsetmacro{\uu}{sqrt(1+(\Dl/\sig)^2)}
  \pgfmathsetmacro{\rr}{ln(\uu)}
  \pgfmathsetmacro{\lone}{\Dl/(\sig^2+\Dl^2)}
  \pgfmathsetmacro{\ltwo}{\Dl^2/(2*\sig^2*(\sig^2+\Dl^2))}
  \pgfmathsetmacro{\aa}{\sig*\uu}
  \pgfmathsetmacro{\kh}{ln(1+\lone*(\aa+\Dl)+\ltwo*((\aa+\Dl)^2-\sig^2))-\rr}
  \begin{axis}[
    width=12cm,
    height=7cm,
    axis lines=middle,
    xlabel={$x$},
    ylabel={},
    domain=-4.5:4.5,
    samples=200,
    xmin=-4.5, xmax=4.5,
    ymin=-2.1, ymax=2.1,
    grid=major,
    grid style={dashed,gray!30},
    legend pos=outer north east,
    legend cell align=left,
  ]
  \addplot[thick, blue, name path=U] {ln(1+\lone*(x+\Dl)+\ltwo*((x+\Dl)^2-\sig^2)) - \rr};
  \addlegendentry{$\psi_u^{\lambda^\star}-r$}
  \addplot[thick, red, name path=Lo] {-ln(1-\lone*(x-\Dl)+\ltwo*((x-\Dl)^2-\sig^2)) + \rr};
  \addlegendentry{$\psi_l^{\lambda^\star}+r$}
  \addplot[fill=green!20, opacity=0.5] fill between[of=U and Lo];

  \addplot[only marks, mark=*, black, mark size=2.2pt] coordinates {(-\aa,-\kh) (\aa,\kh)};

\end{axis}
\end{tikzpicture}
\caption{The kissing geometry for the bounded-variance class for $\sigma=1$ and $\Delta=1$. In Section~\ref{sec:bound-vari-class} we show that the optimal multipliers are $\lambda_1^\star=\Delta/(\sigma^2+\Delta^2)$, $\lambda_2^\star=\Delta^2/(2\sigma^2(\sigma^2+\Delta^2))$ and achieve the rate $r=\log u$ with $u=\sqrt{1+(\Delta/\sigma)^2}$ of Proposition~\ref{prop:bounded-var-synthesis}. The upper envelope $\psi_u^{\lambda^\star}-r$ (blue) and lower envelope $\psi_l^{\lambda^\star}+r$ (red) bound the admissible band (shaded); any non-decreasing estimating function threading the band is optimal. They touch only at the two kissing points $a_{\pm}=\pm\sigma\sqrt{1+(\Delta/\sigma)^2}$ which support the least-favorable pair $(\mb P^\star_{-\Delta},\mb P^\star_\Delta)$, also given in Proposition~\ref{prop:bounded-var-synthesis}.}
  \label{fig:kissing}
\end{figure}

\section{Three Symmetric Moment Classes}
\label{sec:illustrations-2}

The matching theorem of Section~\ref{ssec:matching} identifies the optimal exponential rate $r(\Delta)$ with the Hellinger exponent of a two-point testing problem between the oppositely shifted moment classes $\mc C_{-\Delta}$ and $\mc C_\Delta$. We now make this concrete in three settings: the bounded-variance class of \citet{catoni2012challenging} (Section~\ref{sec:bound-vari-class}), and the bounded $\alpha$-moment class in the light-tailed regime $\alpha\ge2$ (Section~\ref{sec:bounded-alpha-moment-large}) and the heavy-tailed regime $\alpha\in(1,2)$ (Section~\ref{sec:bounded-alpha-moment-1}) studied by \citet{bhatt2022nearly,lee2020optimal}; a fourth, slowly varying, class is treated in Section~\ref{sec:slow-growth}. In each setting we identify, at least for small $\Delta$, an \emph{almost least-favorable} pair, i.e., a pair of distributions $(\mb P^\star_{-\Delta},\mb P^\star_\Delta)\in\mc C_{-\Delta}\times\mc C_\Delta$ whose Hellinger affinity matches the optimal rate to leading order, and, remarkably, in every case these take the same embarrassingly simple form: mirror-symmetric and supported on at most three points $\{-a,0,a\}$,
\begin{equation}\label{eq:almost-lf-pair}
  \mb P^\star_{\Delta} \,=\, p_-\,\delta_{-a} + p_0\,\delta_0 + p_+\,\delta_{a},\qquad \mb P^\star_{-\Delta} \,=\, \mb P^\star_{\Delta}\circ(-\mathrm{Id}) \,=\, p_+\,\delta_{-a} + p_0\,\delta_0 + p_-\,\delta_{a}.
\end{equation}
Two atoms at $\pm a$ already suffice for the variance class and for $\alpha\ge2$ (there $p_0=0$ as can be seen in Figure \ref{fig:kissing}); the heavy-tailed regime $\alpha\in(1,2)$ charges a third atom at the origin to absorb mass ($p_0>0$); and the slow-growth class concentrates on $0$ and a single distant atom ($p_-=0$).

The discreteness of these pairs is not merely a convenience: it is what lets us push past the fixed-margin regime into the fixed-confidence regime. Fixed-margin optimality does not, in general, transfer to the fixed-$\beta$ regime, where the two limits weigh $n$ differently (Section~\ref{sec:related}). However, as each almost least-favorable pair is carried by at most three atoms, the sufficient statistic for testing $\mb P^\star_{-\Delta}$ against $\mb P^\star_\Delta$ is a low-dimensional trinomial (or binomial) count, whose Bayes error we can analyze directly to recover the exact leading constant of the error margin $\Delta_n$ as $\beta\downarrow0$. The same least-favorable distributions thus do double duty, certifying the fixed-margin exponent and, in the fixed-confidence regime, the sharp high-confidence leading constant for the bounded-variance and bounded $\alpha$-moment classes and fixed-$\beta$ leading-order minimaxity for the slowly varying classes.

\subsection{Bounded Variance Class}
\label{sec:bound-vari-class}

We consider the bounded variance class $\mc C_0$ introduced in~\eqref{def:catoni-model} and briefly outline the fixed-margin result of \citet{catoni2012challenging}, which we will contrast with the synthesized estimator proposed in Section~\ref{ssec:sandwich}.

\paragraph{Fixed-Margin Regime.}
\citet{catoni2012challenging} first considers an envelope function $\psi(x)\leq \psi_{2,u}(x) = \log(1+x/\lambda+\abs{x/\lambda}^2/2)$ with
\begin{align*}
  \E{\mb P_0}{\exp(\psi(X-\Delta))} \leq & \E{\mb P_0}{1+(X-\Delta)/\lambda+(X-\Delta)^2/(2\lambda^2)}\\
                                        \leq & 1-\frac{\Delta}{\lambda}+\frac{\sigma^2+\Delta^2}{2\lambda^2}.
\end{align*}
Using $\log(1+t)\leq t$ and $\lambda = \Delta+\tfrac{\sigma^2}{\Delta}$, the same Chernoff bound argument as in the proof of Theorem~\ref{thm:feasibility-guarantee} implies for all $\mu\in\Re$ and $\mb P_\mu\in\mc C_\mu$,
\begin{align}
  \label{eq:catoni-rate-lb}
  \frac 1n \log \mb P_\mu \left[\mu_n-\mu> \Delta\right]\leq -\frac{\Delta^2}{2(\sigma^2+\Delta^2)}.
\end{align}
The symmetric bound on the left tail, $\tfrac1n\log\mb P_\mu\left[\mu_n-\mu<-\Delta\right]\leq-\tfrac{\Delta^2}{2(\sigma^2+\Delta^2)}$, follows similarly from $\psi(x)\geq \psi_{2,l}(x) = -\log(1-x/\lambda+\abs{x/\lambda}^2/2)$.
We may now identify Catoni's analysis with a feasible point in the set $C_\Delta$ defined in Section~\ref{ssec:sandwich} whose envelopes reproduce Catoni's $\psi_{2,u}$ and $\psi_{2,l}$ up to the additive rate $r_{\mathrm C}$. Indeed, setting $\lambda=\Delta+\sigma^2/\Delta$ in
\[
  \log\bigl(1+x/\lambda+x^2/(2\lambda^2)\bigr) \,=\, \log\bigl(1+\lambda_1^{\mathrm{C}}(x+\Delta)+\lambda_2^{\mathrm{C}}((x+\Delta)^2-\sigma^2)\bigr) - r_{\mathrm{C}}
\]
yields the dual multipliers and the exact rate $r_{\mathrm C}$ of Catoni's estimating function, sharper than his reported bound~\eqref{eq:catoni-rate-lb}:
\[
  \lambda_1^{\mathrm{C}} \,=\, e^{r_{\mathrm{C}}}\Bigl(\tfrac{1}{\lambda}-\tfrac{\Delta}{\lambda^2}\Bigr), \qquad
  \lambda_2^{\mathrm{C}} \,=\, \tfrac{e^{r_{\mathrm{C}}}}{(2\lambda^2)}, \qquad
  r_{C}(\Delta) \,=\,\log\Bigl(\frac{2((\sigma/\Delta)^2+1)}{2(\sigma/\Delta)^2+1}\Bigr) > \frac{\Delta^2}{2(\sigma^2+\Delta^2)}.
\]

However, $r_{\mathrm{C}}$ is in turn dominated by $r_M'$ once the multipliers are optimized directly rather than chosen ad hoc. For the bounded variance class everything is available in closed form: the optimal rate, the optimal dual multipliers, and the primal least-favorable pair on whose support the envelopes kiss (Figure~\ref{fig:kissing}); see also Figure~\ref{fig:catoni-rate}.

\begin{proposition}\label{prop:bounded-var-synthesis}
  For the bounded variance class $\mc C_0=\set{\mb P}{\E{\mb P}{X}=0,\,\E{\mb P}{X^2}\le\sigma^2}$ and every $\Delta>0$, the optimal exponential rate is
  \[
    r_M'(\Delta) \;=\; r_M(\Delta) \;=\; r(\Delta) \;=\; \tfrac{1}{2}\log\Bigl(1+\bigl(\tfrac{\Delta}{\sigma}\bigr)^{2}\Bigr).
  \]
  It is certified on the dual side by the synthesis multipliers
  \[
    \lambda_1^\star \;=\; \frac{\Delta}{\sigma^2+\Delta^2},\qquad \lambda_2^\star \;=\; \frac{\Delta^2}{2\sigma^2(\sigma^2+\Delta^2)},
  \]
  whose envelopes~\eqref{eq:envelopes} sandwich the synthesized estimating function at rate $r(\Delta)$, and attained on the primal side by the mirror two-atom least-favorable pair
  \[
    \mb P^\star_{\pm\Delta} \;=\; \frac{1\mp s^\star}{2}\,\delta_{-a} + \frac{1\pm s^\star}{2}\,\delta_{a},\qquad s^\star=\frac{\Delta/\sigma}{\sqrt{1+(\Delta/\sigma)^2}}\in(0,1),\quad a=\sigma\sqrt{1+(\Delta/\sigma)^2}=\Delta/s^\star.
  \]
\end{proposition}
\begin{proof}
  Write $\varepsilon\defn\Delta/\sigma$, $u\defn\sqrt{1+\varepsilon^2}$, and $r\defn\log u=\tfrac12\log(1+\varepsilon^2)$. By Theorem~\ref{thm:matching} the three rates coincide, $r_M'(\Delta)=r_M(\Delta)=r(\Delta)$, so it suffices to bound their common value below by $r$ (dual) and above by $r$ (primal).

  \emph{Dual lower bound $r_M'(\Delta)\ge r$.} The multipliers $\lambda^\star=(\lambda_1^\star,\lambda_2^\star)$ are non-negative, and with $\phi(x)=x^2$, $B=\sigma^2$ the synthesis envelopes~\eqref{eq:envelopes} satisfy the algebraic identity
  \[
    e^{\psi_u^{\lambda^\star}(x)}\,e^{-\psi_l^{\lambda^\star}(x)}-e^{2r}\;=\;(\lambda_2^\star)^2\bigl(x^2-\sigma^2 u^2\bigr)^2\;\ge\;0\qquad\forall x\in\Re,
  \]
  which is exactly the gap inequality $\psi_u^{\lambda^\star}(x)-\psi_l^{\lambda^\star}(x)\ge 2r$, with equality iff $x=\pm\sigma u$. Sandwiching any non-decreasing $\psi$ between the envelopes exhibits $(r,\lambda_1^\star,\lambda_2^\star)\in C_\Delta$, so $r_M'(\Delta)\ge r$.

  \emph{Primal upper bound $r(\Delta)\le r$.} For the pair above, $\E{\mb P^\star_\Delta}{X}=a s^\star=\Delta$ and $\E{\mb P^\star_\Delta}{(X-\Delta)^2}=a^2+\Delta^2-2\Delta a s^\star=a^2-\Delta^2=\sigma^2$, so $\mb P^\star_\Delta\in\mc C_\Delta$ and by mirror symmetry $\mb P^\star_{-\Delta}\in\mc C_{-\Delta}$. Hence, using $1-(s^\star)^2=1/(1+\varepsilon^2)$,
  \[
    r(\Delta) \;=\; -\log\!\sup_{(\mb P_-,\mb P_+)\in\mc C_{-\Delta}\times\mc C_\Delta}\!\!\rho(\mb P_-,\mb P_+) \;\le\; -\log\rho(\mb P^\star_{-\Delta},\mb P^\star_\Delta) \;=\; -\log\!\sqrt{1-(s^\star)^2} \;=\; r.
  \]
  The two bounds and the matching identity give $r_M'(\Delta)=r_M(\Delta)=r(\Delta)=r$, attained by the displayed dual and primal points.
\end{proof}

By Theorem~\ref{thm:matching} this common value is the optimal large-deviation exponent achievable by any measurable estimator, and it is attained by the synthesized $M$-estimators of Section~\ref{ssec:sandwich}:
\[
  \tfrac{1}{2}\log\Bigl(1+\bigl(\tfrac{\Delta}{\sigma}\bigr)^{2}\Bigr) \;=\; \lim_{n\to\infty}-\frac{1}{n}\log\,\beta_n(\mu_n;\Delta).
\]
The optimal multipliers $(\lambda_1^\star,\lambda_2^\star)$ defining this estimator's sandwich envelopes are given in Proposition~\ref{prop:bounded-var-synthesis}.

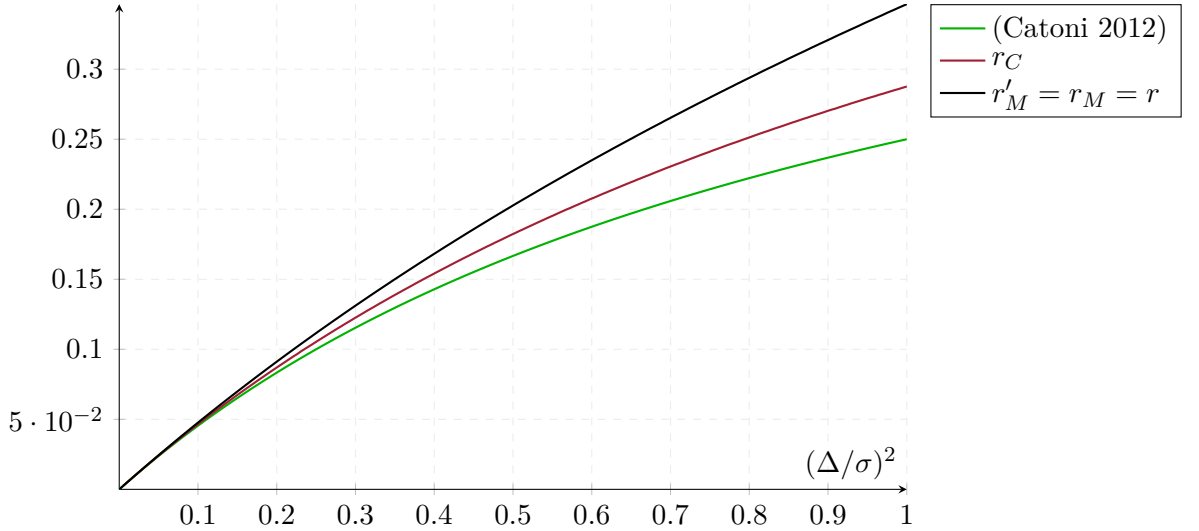
\begin{figure}[h]
  \centering
\begin{tikzpicture}[trim axis left, trim axis right]
  \begin{axis}[
    width=12cm,
    height=8cm,
  axis lines=middle,
  xlabel={$(\Delta/\sigma)^2$},
  ylabel={},
  domain=0:1,
  samples=200,
  xmin=0, xmax=1,
  ymin=0,
  grid=major,
  grid style={dashed,gray!30},
  legend,
  legend pos=outer north east,
  legend cell align=left,
  ]
  \addplot[thick, green] {1/2*1/(1/x+1)};
  \addlegendentry{\citep{catoni2012challenging}}

  \addplot[thick, red] {ln((2*(1/x+1))/(2/x+1))};
  \addlegendentry{$r_C$}

  \addplot[thick, black] {1/2*ln(1+x)};
  \addlegendentry{$r_M'=r_M=r$}

\end{axis}
\end{tikzpicture}
\caption{We consider the bounded variance class. The exact exponential rate $r_C$ at which the estimator of \citet{catoni2012challenging} drives down the probability of an error exceeding $\Delta$, as a function of $\Delta^2/\sigma^2$, is depicted in red. This rate exceeds the rate reported by \citet{catoni2012challenging} (green) yet is itself dominated by the optimized rate $r_M'=r_M=r$ (black).
}
  \label{fig:catoni-rate}
\end{figure}

\paragraph{Fixed-Confidence Regime.}
The optimality in the fixed-confidence regime is delivered by an elementary route: $N(0,\sigma^2)$ itself lies inside the bounded-variance class, so any estimator satisfying~\eqref{eq:confidence-guarantee} inherits the Gaussian lower bound $\Delta_n\ge\sigma\,\Phi^{-1}(1-\beta)/\sqrt n$ \citep[Proposition~6.1]{catoni2012challenging}; since $\Phi^{-1}(1-\beta)\sim\sqrt{2\log(1/\beta)}$ as $\beta\downarrow 0$, this pins the leading constant of $\Delta_n$ at $\sqrt2$, matching Catoni's upper bound. This Gaussian argument, however, does \emph{not} extend to the bounded $\alpha$-moment class when $\alpha>2$, because $N(0,\sigma^2)$ then falls outside the class. For $X\sim N(0,\tau^2)$, writing $X=\tau Z$ with $Z\sim N(0,1)$, the moment constraint reads $\E{}{\abs X^\alpha}=\tau^\alpha\,\E{}{\abs Z^\alpha}\le\sigma^\alpha$, so only Gaussians with
\[
  \tau\;\le\;\frac{\sigma}{m_\alpha},\qquad m_\alpha\defn\bigl(\E{}{\abs Z^\alpha}\bigr)^{1/\alpha}=\Bigl(\frac{2^{\alpha/2}\,\Gamma(\tfrac{(\alpha+1)}{2})}{\sqrt\pi}\Bigr)^{1/\alpha},
\]
lie inside. Feeding the largest such inscribed Gaussian, $N\!\bigl(0,(\sigma/m_\alpha)^2\bigr)$, into the deviation argument yields the fixed-$\beta$ lower bound $\Delta_n\ge(\sqrt2+o(1))\,\tfrac{\sigma}{m_\alpha}\sqrt{\log(1/\beta)/n}$, short of the upper bound $(\sqrt2+o(1))\,\sigma\sqrt{\log(1/\beta)/n}$ by the factor $m_\alpha>1$ strictly for $\alpha>2$. To resolve $\alpha>2$, we need a genuinely heavier-tailed least-favorable distribution, which we now derive.

\subsection{Bounded $\alpha$-Moment Class for \texorpdfstring{$\alpha\ge 2$}{alpha>=2}}
\label{sec:bounded-alpha-moment-large}

For $\alpha\ge 2$ the bounded $\alpha$-moment class is contained in Catoni's bounded-variance class so the synthesized bounded-variance $M$-estimator of Proposition~\ref{prop:bounded-var-synthesis} applies verbatim, delivering Catoni's rate $\tfrac12\log(1+(\Delta/\sigma)^2)$ and the sub-Gaussian margin $\Delta_n\le(\sqrt 2+o(1))\sigma\sqrt{\log(1/\beta)/n}$. The next proposition shows this dual certificate is sharp: a primal two-atom \emph{near} least-favorable pair matches its rate to leading order by appealing to a Lindeberg-Feller CLT.

\begin{proposition}\label{prop:two-atom-large-alpha}
Let $\alpha\ge 2$ and $\varepsilon\defn\Delta/\sigma$. Since $\mc C_0\subseteq\set{\mb P}{\E{\mb P}{X}=0,\ \E{\mb P}{X^2}\le\sigma^2}$ by Jensen, the bounded-variance multipliers of Proposition~\ref{prop:bounded-var-synthesis},
\[
  \lambda_1^\star=\frac{\Delta}{\sigma^2+\Delta^2},\qquad \lambda_2^\star=\frac{\Delta^2}{2\sigma^2(\sigma^2+\Delta^2)},
\]
remain feasible and their synthesized $M$-estimator certifies the rate
\[
  r_M'(\Delta)\;\ge\;\tfrac12\log\bigl(1+(\Delta/\sigma)^2\bigr).
\]
Furthermore, there is $\Delta_0>0$ such that for every $\Delta\in(0,\Delta_0]$ the system
\begin{equation}\label{eq:large-alpha-saddle}
  a\,s=\Delta,\qquad \tfrac{(1-s)}{2}(a+\Delta)^\alpha+\tfrac{(1+s)}{2}(a-\Delta)^\alpha=\sigma^\alpha
\end{equation}
has a solution $(a(\Delta),s(\Delta))\in(0,\infty)\times(0,1)$, depending continuously on $\Delta$, with $a(\Delta)=\sigma(1+O(\varepsilon^2))$ and $s(\Delta)=\varepsilon+O(\varepsilon^3)$, and the mirror two-atom near least-favorable pair
\[
  \mb P^\star_{\pm\Delta}=\frac{1\mp s(\Delta)}{2}\,\delta_{-a(\Delta)}+\frac{1\pm s(\Delta)}{2}\,\delta_{+a(\Delta)}\;\in\;\mc C_{\pm\Delta}
\]
gives
\[
  r(\Delta)\;\le\;-\tfrac12\log\bigl(1-s(\Delta)^2\bigr)\;=\;\tfrac12\log\bigl(1+(\Delta/\sigma)^2\bigr)\,(1+o(1)).
\]
\end{proposition}
\begin{proof}
For $\alpha\ge2$ and any $\mb P\in\mc C_{\pm\Delta}$, Jensen's inequality gives $\E{\mb P}{(X\mp\Delta)^2}\le\bigl(\E{\mb P}{|X\mp\Delta|^\alpha}\bigr)^{2/\alpha}\le\sigma^2$, so $\mb P$ also obeys the bounded-variance constraint at the same margin. The monotone estimating function synthesized in Proposition~\ref{prop:bounded-var-synthesis} from $\lambda^\star=(\lambda_1^\star,\lambda_2^\star)$ therefore controls it too, and the corresponding $M$-estimator attains worst-case rate $\tfrac12\log(1+\varepsilon^2)$ on $\mc C_0$; hence $r_M'(\Delta)\ge\tfrac12\log(1+\varepsilon^2)$.

The pair $\mb P^\star_{\pm\Delta}$ has mean $\E{\mb P^\star_{\pm\Delta}}{X}=\pm\Delta$ for any $(a,s)$ with $as=\Delta$, and its moment constraint is the second equation of~\eqref{eq:large-alpha-saddle},
\[
  \E{\mb P^\star_\Delta}{|X-\Delta|^\alpha}=\tfrac{(1-s)}{2}(a+\Delta)^\alpha+\tfrac{(1+s)}{2}(a-\Delta)^\alpha=\sigma^\alpha .
\]
Substituting $s=\Delta/a$ and rescaling by $u\defn a/\sigma$, $\varepsilon\defn\Delta/\sigma$ (so $a=\sigma u$, $\Delta=\sigma\varepsilon$, and dividing through by $\sigma^\alpha$) turns this into $F(u;\varepsilon)=0$ with
\[
  F(u;\varepsilon)\defn(1-\varepsilon/u)(u+\varepsilon)^\alpha+(1+\varepsilon/u)(u-\varepsilon)^\alpha-2 .
\]
At $\varepsilon_0=0$, $F(u;0)=2u^\alpha-2$, so $u_0=1$. On a neighborhood of $(u_0, \varepsilon_0) = (1,0)$ both bases $u\pm\varepsilon$ are bounded away from $0$; since $t\mapsto t^\alpha$ is $C^2$ on $(0,\infty)$ and $u\mapsto\varepsilon/u$ is smooth for $u$ near $1$, $F$ is $C^\infty$ around $(u_0, \varepsilon_0)$. As $\partial F/\partial u\,|_{(u_0,\varepsilon_0)}=2\alpha> 0$, the $C^2$ implicit function theorem \citep{krantz2013implicit} yields a unique $C^2$ solution $u(\varepsilon)$ near $\varepsilon=\varepsilon_0=0$ with $u(\varepsilon_0)=u_0$.
Differentiating the identity $F(u(\varepsilon),\varepsilon)\equiv0$ and applying the chain rule gives $\partial_u F\cdot u'(\varepsilon)+\partial_\varepsilon F=0$, that is $u'(\varepsilon)=-\partial_\varepsilon F/\partial_u F$ along the branch. At $(1,0)$ we have $\partial F/\partial\varepsilon=0$ (the odd-in-$\varepsilon$ contributions cancel), so $u'(0)=0$, and a second-order Taylor expansion of the branch gives $u(\varepsilon)=u(0)+u'(0)\,\varepsilon+O(\varepsilon^2)=1+O(\varepsilon^2)$. Consequently $a(\Delta)=\sigma(1+O(\varepsilon^2))$ and $s(\Delta)=\Delta/a(\Delta)=\varepsilon+O(\varepsilon^3)$, and $\mb P^\star_{\pm\Delta}$ then satisfies both constraints, so $\mb P^\star_{\pm\Delta}\in\mc C_{\pm\Delta}$.

\emph{Rate.} The pairs $\mb P^\star_\Delta,\mb P^\star_{-\Delta}$ share the atoms $\pm a(\Delta)$ with masses interchanged, so $\rho(\mb P^\star_{-\Delta},\mb P^\star_\Delta)=2\cdot\tfrac12\sqrt{(1-s)(1+s)}=\sqrt{1-s(\Delta)^2}$, and by~\eqref{eq:intro-hellinger} we have $r(\Delta)\le-\log\rho(\mb P^\star_{-\Delta},\mb P^\star_\Delta)=-\tfrac12\log(1-s(\Delta)^2)=\tfrac12 s(\Delta)^2(1+o(1))=\tfrac{\Delta^2}{(2\sigma^2)}(1+o(1))=\tfrac12\log(1+\varepsilon^2)(1+o(1))$, using $s(\Delta)=\varepsilon+O(\varepsilon^3)$. 
\end{proof}

\begin{proposition}[Fixed-confidence lower bound at fixed $\beta$, $\alpha\ge 2$]\label{prop:CLT-lb-fixed-beta}
For every $\alpha\ge 2$, every $\beta\in(0,\tfrac12)$, and every sequence of estimators $(\mu_n)$ with worst-case confidence level~\eqref{eq:beta-n-def} obeying $\beta_n(\mu_n;\Delta_n)\le\beta$ for all $n$, we have
\[
\liminf_{n\to\infty}\,\frac{\sqrt{n}\,\Delta_n}{\sigma}\;\ge\;\Phi^{-1}(1-\beta).
\]
\end{proposition}

The proof is deferred to Appendix~\ref{app:deferred}.

The bound of Proposition~\ref{prop:CLT-lb-fixed-beta} is the exact analogue of the Gaussian lower bound of Section~\ref{sec:bound-vari-class}. For each fixed $\beta$ it gives $\liminf_{n\to\infty}\sqrt n\,\Delta_n/\sigma\ge\Phi^{-1}(1-\beta)$, and since $\Phi^{-1}(1-\beta)=\sqrt{2\log(1/\beta)}\,(1+o(1))$ as $\beta\downarrow0$, this matches Catoni's upper bound $\Delta_n\le(\sqrt 2+o(1))\,\sigma\sqrt{\log(1/\beta)/n}$ to leading order. The optimal high-confidence constant is therefore $\sqrt 2$ for the bounded $\alpha$-moment class for every $\alpha\ge 2$.

The argument proving Proposition~\ref{prop:CLT-lb-fixed-beta} is special to the light-tailed regime $\alpha\ge2$. Its engine is the two-atom least-favorable pair, whose bias $s(\Delta)\sim\Delta/\sigma$ \emph{vanishes} as $\Delta\downarrow0$; in the relevant scaling $\Delta_n\asymp\sigma\sqrt{\log(1/\beta)/n}$ this gives $\sqrt n\,s_n=\Theta(1)$, exactly the regime in which a binomial count is approximately Gaussian, so a central limit theorem inverts the Bayes error. For $\alpha\in(1,2)$ the class is heavy-tailed and the least-favorable distribution changes shape: it must place a \emph{third} atom at the origin to carry mass, and its bias no longer vanishes but settles at the constant $s^\star=\sqrt{\alpha(2-\alpha)}$. A binomial with non-vanishing bias is no longer approximately Gaussian; instead the trinomial sufficient statistic, after a Poisson coupling, converges to a Skellam law (a difference of two Poissons), and the Bayes error must be inverted through Bessel-function asymptotics rather than a CLT. The next subsection develops this more delicate analysis.

\subsection{Bounded $\alpha$-Moment Class for \texorpdfstring{$\alpha\in(1,2)$}{alpha in (1,2)}}
\label{sec:bounded-alpha-moment-1}

The bounded $\alpha$-moment class~\eqref{def:bhatt-model} is treated by \citet{lee2020optimal,chen2021generalized,bhatt2022nearly} through the one-parameter $M$-estimators built from the sandwich~\eqref{eq:samorodnitsky-butterfly}. At the Lee/Bhatt constant $C(\alpha)$ of Section~\ref{sec:Bhatt-model} this family attains the leading margin constant
\begin{equation}
  \label{eq:L-def}
  L(\alpha) \,\defn\, \frac{\alpha\,C(\alpha)^{1/\alpha}}{(\alpha-1)^{(\alpha-1)/\alpha}},
\end{equation}
which we show below is the smallest constant this family, and indeed any estimator, can achieve; throughout, $q\defn\alpha/(\alpha-1)\in(2,\infty)$ denotes the conjugate exponent. We contrast their construction with the synthesized estimator of Section~\ref{ssec:sandwich}.

\paragraph{Fixed-Margin Regime.}
We first pin the rate of the bounded $\alpha$-moment class to leading order, the exact analogue of Proposition~\ref{prop:two-atom-large-alpha} for the heavy-tailed regime: on the dual side the one-parameter estimator of \citet{lee2020optimal,bhatt2022nearly} certifies a closed-form lower bound on $r_M$, and on the primal side an explicit three-atom near least-favorable pair certifies a matching upper bound on $r$.

\begin{proposition}\label{prop:bhatt-synthesis}
Let $\alpha\in(1,2)$, $q=\alpha/(\alpha-1)$, and $\varepsilon\defn\Delta/\sigma$. The one-parameter estimating function~\eqref{eq:samorodnitsky-butterfly} of \citet{lee2020optimal,bhatt2022nearly}, at their constant $\tilde C=C(\alpha)$, is a feasible non-decreasing $M$-estimator and certifies
\[
  r_M(\Delta)\;\ge\;r_B(\Delta)\;\defn\;\frac{\alpha-1}{\alpha}\,\frac{\Delta^{q}}{\bigl(\alpha\,C(\alpha)\,M_\alpha(\Delta)\bigr)^{1/(\alpha-1)}},\qquad M_\alpha(\Delta):=\sup_{\mb P\in\mc C_0}\E{\mb P}{\abs{X-\Delta}^\alpha}.
\]
Consider the mirror three-atom pair
\[
  \mb P^\star_{\pm\Delta}=p_-\delta_{\mp a}+p_0\delta_0+p_+\delta_{\pm a},\qquad p_\pm=(1\pm s^\star)m,\ \ p_0=1-2m,\ \ s^\star=\sqrt{\alpha(2-\alpha)},
\]
in which the scale $a=\Delta/(p_+-p_-)$ fixes the means at $\pm\Delta$, and the mass parameter $m=m(\varepsilon)\in(0,\tfrac12)$ is chosen so that the $\alpha$-moment constraint binds, $\E{\mb P^\star_\Delta}{\abs{X-\Delta}^\alpha}=\sigma^\alpha$; then $\mb P^\star_{\pm\Delta}\in\mc C_{\pm\Delta}$ for every sufficiently small $\varepsilon$. By Lemma~\ref{lem:F-star-asymp}, $p_0+2\sqrt{p_+p_-}=1-K(\alpha)\varepsilon^q+o(\varepsilon^q)$, so
\begin{equation}\label{eq:K-def}
  r(\Delta)\;\le\;-\log\bigl(p_0+2\sqrt{p_+p_-}\bigr)\;=\;K(\alpha)\,\varepsilon^q+o(\varepsilon^q),\qquad K(\alpha)\defn\frac{\alpha-1}{\alpha\bigl(\alpha\,C(\alpha)\bigr)^{1/(\alpha-1)}}.
\end{equation}
\end{proposition}
\begin{proof}
  The envelope~\eqref{eq:samorodnitsky-butterfly} satisfies $\psi(x)\le\psi_{\alpha,u}(x)=\log(1+x/\lambda+C(\alpha)\abs{x/\lambda}^\alpha)$, so using $\E{\mb P_0}{X}=0$,
\[
  \E{\mb P_0}{\exp(\psi(X-\Delta))}\le 1-\Delta/\lambda+C(\alpha)\,\E{\mb P_0}{\abs{X-\Delta}^\alpha}/\lambda^\alpha.
\]
Using $\log(1+t)\le t$ and taking the supremum over $\mb P_0\in\mc C_0$ (so that $\E{\mb P_0}{\abs{X-\Delta}^\alpha}\le M_\alpha(\Delta)$), $\sup_{\mb P_0\in\mc C_0}\log\E{\mb P_0}{\exp(\psi(X-\Delta))}\le-\Delta/\lambda+C(\alpha)M_\alpha(\Delta)/\lambda^\alpha$ for every $\lambda>0$; minimizing the right-hand side at $\lambda^{\alpha-1}=\alpha C(\alpha)M_\alpha(\Delta)/\Delta$ gives $\sup_{\mb P_0\in\mc C_0}\log\E{\mb P_0}{\exp(\psi(X-\Delta))}\le-r_B(\Delta)$, i.e.\ the feasibility constraint~\eqref{eq:sf-right} holds at rate $r_B(\Delta)$. The left tail is symmetric: from the lower envelope $\psi(x)\ge\psi_{\alpha,l}(x)=-\log(1-x/\lambda+C(\alpha)\abs{x/\lambda}^\alpha)$ of~\eqref{eq:samorodnitsky-butterfly} and the reflection invariance of $\mc C_0$ under $X\mapsto-X$, the same computation gives $\sup_{\mb P_0\in\mc C_0}\log\E{\mb P_0}{\exp(-\psi(X+\Delta))}\le-r_B(\Delta)$, the constraint~\eqref{eq:sf-left}. As $\psi$ is non-decreasing and $(\psi,r_B(\Delta))$ thus satisfies~\eqref{eq:symmetric-feasibility}, the definition~\eqref{eq:r_M} of $r_M$ gives $r_M(\Delta)\ge r_B(\Delta)$.

For the three-atom pair above, the mirror Hellinger affinity~\eqref{eq:rho-def} is $\rho(\mb P^\star_{-\Delta},\mb P^\star_\Delta)=\int\sqrt{\d\mb P^\star_{-\Delta}\,\d\mb P^\star_\Delta}=p_0+2\sqrt{p_+p_-}$ (the shared atom at $0$ contributes $p_0$, the swapped atoms at $\pm a$ contribute $2\sqrt{p_+p_-}$), so by~\eqref{eq:intro-hellinger},
\[
  r(\Delta)\;\le\;-\log\rho(\mb P^\star_{-\Delta},\mb P^\star_\Delta)=-\log\bigl(p_0+2\sqrt{p_+p_-}\bigr).
\]
Lemma~\ref{lem:F-star-asymp} furnishes a triple with $p_0+2\sqrt{p_+p_-}=1-K(\alpha)\varepsilon^q+o(\varepsilon^q)$, whence $r(\Delta)\le-\log(1-K(\alpha)\varepsilon^q+o(\varepsilon^q))=K(\alpha)\varepsilon^q+o(\varepsilon^q)$.
\end{proof}

The two certificates pin the rate to leading order. Since $M_\alpha(\Delta)\to\sigma^\alpha$ as $\Delta\downarrow0$, the dual rate is $$r_B(\Delta)=\frac{\alpha-1}{\alpha}\,\Delta^q/\bigl(\alpha\,C(\alpha)\,\sigma^\alpha\bigr)^{1/(\alpha-1)}(1+o(1))=K(\alpha)\varepsilon^q(1+o(1))$$ (using $\sigma^{\alpha/(\alpha-1)}=\sigma^q$), matching the primal upper bound~\eqref{eq:K-def}. As $r_M(\Delta)\le r(\Delta)$, the dual lower bound and the primal upper bound sandwich both rates, so
\[
  r(\Delta)=r_M(\Delta)\,(1+o(1))=K(\alpha)\,(\Delta/\sigma)^q\,(1+o(1))\qquad(\Delta\downarrow0).
\]

\begin{lemma}\label{lem:F-star-asymp}
  Let $\alpha\in(1,2)$ and conjugate exponent $q=\alpha/(\alpha-1)$, with $C(\alpha)$ as in Section~\ref{sec:Bhatt-model}. For every sufficiently small $\varepsilon>0$ there exist $p_0,p_+,p_-\ge 0$ with $p_0+p_++p_-=1$ and $a>0$ such that the pair
  \[
    \mb P^\star_{\Delta} = p_{-}\delta_{-a} + p_0\delta_{0} + p_{+}\delta_{a},\qquad
    \mb P^\star_{-\Delta} = p_{+}\delta_{-a} + p_0\delta_{0} + p_{-}\delta_{a}
  \]
  is feasible (i.e., $\mb P^\star_{\pm\Delta}\in\mc C_{\pm\Delta}$ for $\Delta=\sigma\varepsilon$) and satisfies
  \[
    p_0+2\sqrt{p_+p_-} \,=\, 1 - K(\alpha)\,\varepsilon^q + o(\varepsilon^q)
    \qquad\text{as } \varepsilon\downarrow 0,
  \]
  with $K(\alpha)$ as in~\eqref{eq:K-def}.
\end{lemma}

The proof is deferred to Appendix~\ref{app:deferred}.

Leading order aside, Theorem~\ref{thm:matching} gives the exact rate at any \emph{fixed} margin: the synthesis problem~\eqref{eq:r_M_prime} and the two-point lower bound~\eqref{eq:intro-hellinger} can each be solved numerically and agree on the common value $r_M'(\Delta)=r_M(\Delta)=r(\Delta)$. Figure~\ref{fig:bhatt-rate} reports this rate for $\alpha=1.5$: the optimized rate $r_M'$ dominates the one-parameter rate $r_B$ by feasibility, and across the depicted range of margins this domination is strict, a genuine gap rather than a numerical artefact.

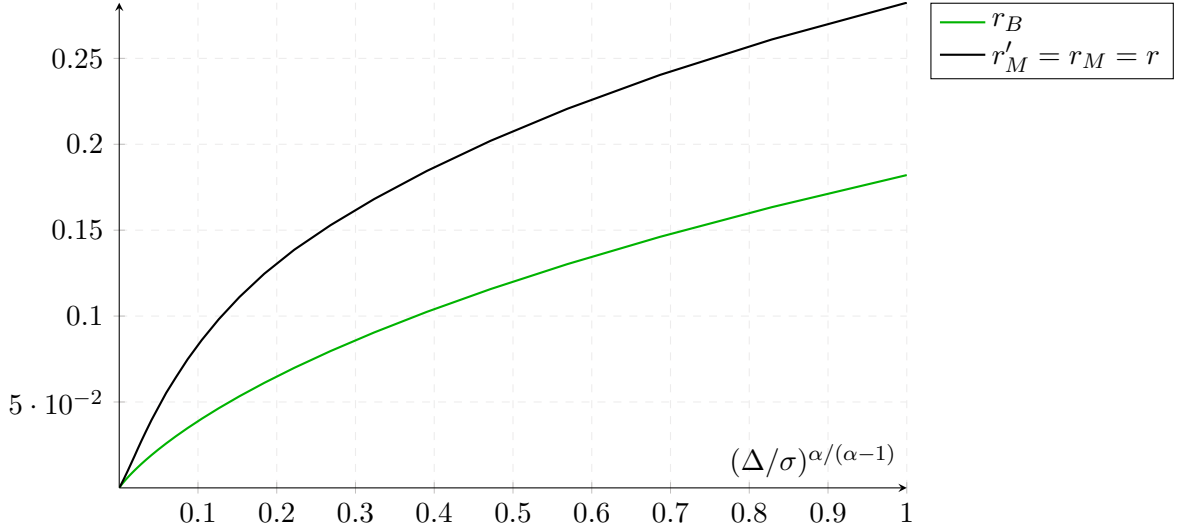
\begin{figure}[h]
  \centering
\begin{tikzpicture}[trim axis left, trim axis right]
  \begin{axis}[
    width=12cm,
    height=8cm,
  axis lines=middle,
  xlabel={$(\Delta/\sigma)^{\alpha/(\alpha-1)}$},
  ylabel={},
  domain=0:1,
  samples=200,
  xmin=0, xmax=1,
  ymin=0,
  grid=major,
  grid style={dashed,gray!30},
  legend,
  legend pos=outer north east,
  legend cell align=left,
  ]

  \addplot[thick, green] table[col sep=comma, x=xs, y=r_bhatt] {data/bhatt-rate-a-1.5.csv};
  \addlegendentry{$r_B$}

  \addplot[thick, black] table[col sep=comma, x=xs, y=rmax] {data/bhatt-rate-a-1.5.csv};
  \addlegendentry{$r_M'=r_M=r$}

\end{axis}
\end{tikzpicture}
\caption{We consider the bounded $\alpha$-moment class with $\alpha=1.5$.
  The exponential rate $r_B$ at which the estimator of \citet{lee2020optimal,bhatt2022nearly} drives down the probability of an error exceeding $\Delta$, as a function of $(\Delta/\sigma)^{\alpha/(\alpha-1)}$, is depicted in green. This rate is numerically dominated by the optimized rate $r_M'$ (black) for all $\Delta\neq 0$ over the depicted range. }
  \label{fig:bhatt-rate}
\end{figure}

\paragraph{Fixed-Confidence Regime.}
The fixed-confidence guarantee of the one-parameter $M$-estimator of \citet{lee2020optimal,bhatt2022nearly} is immediate from the fixed-margin analysis above: Proposition~\ref{prop:bhatt-synthesis} shows its $\psi$ satisfies the feasibility constraints~\eqref{eq:symmetric-feasibility} at the rate $r_B(\Delta)=K(\alpha)(\Delta/\sigma)^q(1+o(1))$, so Theorem~\ref{thm:feasibility-guarantee} gives $\beta_n(\mu_n;\Delta)\le e^{-n\,r_B(\Delta)}$ for all $\mu$ and $\mb P_\mu\in\mc C_\mu$. Inverting $n\,r_B(\Delta_n)=\log(1/\beta)$ recovers their heavy-tailed counterpart of the variance-class margin~\eqref{eq:catoni-deviation}, with leading constant $L(\alpha)=K(\alpha)^{-(\alpha-1)/\alpha}$: for any $\beta\in(0,1)$ and $n$ sufficiently large,
\begin{equation}
  \label{eq:bhatt-deviation}
  \max(\mb P_\mu\!\left[\mu_n-\mu > \Delta_n\right], \mb P_\mu\!\left[\mu_n-\mu < -\Delta_n\right])  \leq \beta \quad \forall \mu\in\Re,\ \mb P_\mu\in \mc C_\mu,
\end{equation}
for the error margin
\begin{equation}
  \label{eq:bhatt-rate}
  \Delta_n =  \bigl(L(\alpha)+o(1)\bigr) \sigma \left(\tfrac{\log(1/\beta)}{n}\right)^{(\alpha-1)/\alpha}.
\end{equation}
\citet[Theorem 3.1]{devroye2015subgaussianmeanestimators} implies that any estimator satisfying~\eqref{eq:confidence-guarantee} requires
\begin{equation}
  \label{eq:lb-devroye}
  \Delta_n\geq \left(1+o(1)\right) \sigma \left(\tfrac{\log(1/\beta)}{n}\right)^{(\alpha-1)/\alpha},
\end{equation}
matching~\eqref{eq:bhatt-rate} in rate but with a strictly smaller leading constant, $1<L(\alpha)$ for $\alpha\in(1,2)$: by~\eqref{eq:L-def}, $\log L(\alpha)$ has derivative $-\log(2-\alpha)/\alpha^2>0$ on $(1,2)$ and $\log L(1)=0$, so $\log L(\alpha)=\int_1^\alpha -\log(2-u)/u^2\,\d u>0$, whence $L(\alpha)>1$.
The previous result will now be used to establish that the margin~\eqref{eq:bhatt-rate} is asymptotically minimax-optimal as $\beta\downarrow 0$ for the bounded $\alpha$-moment class, $\alpha\in(1,2)$, in direct analogy with the $\sqrt 2$-matching of \citet{catoni2012challenging} for the variance class. The argument is split into three pieces: a non-asymptotic binary hypothesis test lower bound (Proposition~\ref{prop:skellam-lb-nonasymp}) that holds at every $\beta\in(0,1/2)$ and every $n\ge 1$; its sequential limit at a fixed confidence level (Proposition~\ref{prop:skellam-lb-fixed-beta}); and the asymptotic analysis of that limit as $\beta\downarrow 0$ (Corollary~\ref{cor:skellam-lb-asymp}), which recovers the upper-bound constant exactly.

We recall the quantities from Lemma~\ref{lem:F-star-asymp} that we will use repeatedly. Setting $s^\star\defn\sqrt{\alpha(2-\alpha)}\in(0,1)$ and $\varepsilon\defn\Delta/\sigma$, the three-atom near least-favorable pair $\mb P^\star_{\pm\Delta}=p_-\delta_{\mp a}+p_0\delta_0+p_+\delta_{\pm a}$ has $p_\pm=(1\pm s^\star)m$, $p_0=1-2m$, with feasible mass
\begin{equation}\label{eq:m-scaling}
  m=m(\varepsilon)=\tfrac12(s^\star)^{-q}\varepsilon^q\,(1+o(1)) \qquad\text{as }\varepsilon\downarrow 0,
\end{equation}
and the Hellinger-shape parameter
\begin{equation}\label{eq:rho-s}
  \rho_s\defn\sqrt{1-(s^\star)^2}=\alpha-1, \qquad 1-\rho_s=2-\alpha.
\end{equation}

\begin{proposition}\label{prop:skellam-lb-nonasymp}
  Let $\alpha\in(1,2)$. There exists $\varepsilon_0>0$ (depending only on $\alpha$) such that for every $n\ge 1$, every $\Delta_n>0$ with $\Delta_n/\sigma\le\varepsilon_0$, and every measurable estimator $\mu_n$, its worst-case confidence level $\beta_n(\mu_n;\Delta_n)$ of~\eqref{eq:beta-n-def} obeys
  \begin{equation}\label{eq:skellam-lb-nonasymp}
    \beta_n(\mu_n;\Delta_n) \;\geq\; \mc B_n(\Delta_n) \defn \mb P^\star_{\Delta_n}[S_n<0]+\tfrac12\,\mb P^\star_{\Delta_n}[S_n=0],
  \end{equation}
  where $\mb P^\star_{\Delta_n}$ is the three-atom near least-favorable pair of Lemma~\ref{lem:F-star-asymp} at margin $\Delta_n$ and $S_n\defn\sum_{i=1}^n X_i/a(\Delta_n)$ is the sufficient statistic.
\end{proposition}
\begin{proof}
  Let $\varepsilon_*>0$ be the threshold from Lemma~\ref{lem:F-star-asymp} below which the three-atom near least-favorable pair is constructible, and take the constant of the statement to be $\varepsilon_0\defn\varepsilon_*/2$. Since $\Delta_n/\sigma\le\varepsilon_0<\varepsilon_*$, the interval $(\Delta_n,\sigma\varepsilon_*]$ is nonempty; fix any $\Delta'$ in it, so that $\Delta'/\sigma\le\varepsilon_*$ and the pair $\mb P^\star_{\pm\Delta'}$ exists (as does $\mb P^\star_{\Delta_n}$, since $\Delta_n/\sigma\le\varepsilon_0<\varepsilon_*$). Specializing the right-hand side of the non-asymptotic inequality~\eqref{eq:LB-reduction} from the proof of Proposition~\ref{prop:LB} to the pair $(\mu,\mb P)=(\pm\Delta',\mb P^\star_{\pm\Delta'})$ at margin $\Delta'$ and applying the Bayes-risk bound at uniform prior (in the notation $\mc R^\star_n(\cdot,\cdot)$ of~\eqref{eq:bayes-error-def}) gives
  \[
    \beta_n(\mu_n;\Delta_n) \;\geq\; \mc R^\star_n\!\bigl(\mb P^\star_{-\Delta'},\mb P^\star_{\Delta'}\bigr) \;=\; \mc B_n(\Delta'),
  \]
  where the last equality uses sufficiency. Indeed, by sign-symmetry $\mb P^\star_{-\Delta'}=\mb P^\star_{\Delta'}\circ(-\mathrm{Id})$, so under the change of variable $Y_i\defn X_i/a(\Delta')\in\{-1,0,+1\}$ the two hypotheses become $Y_i\sim(p_-,p_0,p_+)$ and $Y_i\sim(p_+,p_0,p_-)$ respectively, where $p_\pm=p_\pm(\Delta')$. The single-observation log-likelihood ratio is then
  \[
    \log\frac{\d\mb P^\star_{\Delta'}}{\d\mb P^\star_{-\Delta'}}(Y_i)\;=\;
    \begin{cases}
      \log(p_+/p_-), & Y_i=+1,\\[2pt]
      \;\;\;0, & Y_i=0,\\[2pt]
      -\log(p_+/p_-), & Y_i=-1,
    \end{cases}
  \]
  using $\log(p_-/p_+)=-\log(p_+/p_-)$; summing in $i$ gives the log-likelihood ratio $\log(p_+/p_-)\,S_n$, proportional to $S_n=\sum_{i=1}^n Y_i$. The Bayes-optimal test therefore thresholds $S_n$ at $0$ with randomized tie-breaking on $\{S_n=0\}$, which produces exactly the trinomial expression for $\mc B_n(\Delta')$. The atom masses $(p_-,p_0,p_+)$ of the pair constructed in Lemma~\ref{lem:F-star-asymp} are continuous functions of the margin, so the trinomial probabilities defining $\mc B_n(\cdot)$ are continuous in $\Delta'$. Letting $\Delta'\downarrow\Delta_n$ in the bound above yields $\beta_n(\mu_n;\Delta_n)\geq\mc B_n(\Delta_n)$.
\end{proof}

The bound~\eqref{eq:skellam-lb-nonasymp} is fully non-asymptotic: $\mc B_n(\Delta_n)$ is the \emph{exact} trinomial Bayes error $\mc R^\star_n$, evaluable at any $(n,\Delta_n)$. The lower bound on the optimal margin now amounts to evaluating $\mc B_n$ asymptotically and inverting the resulting relation (first $n\to\infty$ at fixed $\beta$, then $\beta\downarrow 0$) to recover the leading constant in~\eqref{eq:bhatt-rate}. In the $\alpha\ge2$ regime this evaluation went through a \emph{central limit theorem}: the centered sign-count, rescaled by $\sqrt n$, converges to a Gaussian and its Bayes error to a Gaussian tail (Proposition~\ref{prop:CLT-lb-fixed-beta}). That route is unavailable here. The informative atoms $\pm a$ carry vanishing mass $p_\pm=(1\pm s^\star)m$ with $m=m(\varepsilon)\downarrow0$, so the operative scaling is not $\sqrt n$ but the Poisson regime $nm=O(1)$: the counts $n_\pm$ become asymptotically independent Poissons, and their difference $S_n=n_+-n_-$ converges instead to a Skellam law\footnote{The Skellam distribution $\mathrm{Skellam}(\mu_1,\mu_2)$, $\mu_1,\mu_2>0$, is the law of the difference $N_1-N_2$ of two independent Poisson variables $N_1\sim\mathrm{Poisson}(\mu_1)$ and $N_2\sim\mathrm{Poisson}(\mu_2)$, supported on $\mathbb Z$ with mass function
\begin{equation}\label{eq:skellam-pmf}
  \Pr[N_1-N_2=k]\;=\;e^{-(\mu_1+\mu_2)}(\mu_1/\mu_2)^{k/2}\,I_{|k|}\bigl(2\sqrt{\mu_1\mu_2}\bigr),\qquad k\in\mathbb Z,
\end{equation}
with $I_\nu$ the modified Bessel function of the first kind.}. Lemma~\ref{lem:tv-poisson} below makes this precise and will play the role that the Lindeberg--Feller central limit theorem played for $\alpha\ge2$: it bounds the total variation distance between the law of $S_n$ to a Skellam law by $4nm^2=4(nm)\,m$, which vanishes in the Poisson regime $nm=O(1)$ since then $4nm^2=O(1)\cdot m\to0$ as $m\downarrow0$.

\begin{lemma}\label{lem:tv-poisson}
  Let $(n_+,n_-,n_0)\sim\mathrm{Multinomial}(n;p_+,p_-,p_0)$ with $p_\pm=(1\pm s^\star)m$, $p_0=1-2m$, and write $S_n=n_+-n_-$. Then
  \begin{equation}\label{eq:tv-skellam}
    \mathrm{TV}\bigl(\mathcal L(S_n),\,\mathrm{Skellam}(n(1+s^\star)m,n(1-s^\star)m)\bigr) \;\leq\; 4nm^2.
  \end{equation}
\end{lemma}

The proof is deferred to Appendix~\ref{app:deferred}.

Another feature further distinguishes the $\alpha\in (1, 2)$ regime from the $\alpha\geq 2$ regime: the tie term $\tfrac12\mb P^\star_{\Delta_n}[S_n=0]$ in~\eqref{eq:skellam-lb-nonasymp} is retained, since the asymptotic Skellam law carries positive mass at $0$ (in the Gaussian limit it vanishes and was discarded); and the inversion $\lambda^\star(\beta)=\mc B_S^{-1}(\beta)$ below is calibrated to the exact $\mc B_S$. We now establish that this inversion is well-defined and continuous on $(0,1/2)$ using a continuous Markov chain argument.

\begin{lemma}\label{lem:BS-properties}
  Let $\alpha\in(1,2)$, $s^\star\defn\sqrt{\alpha(2-\alpha)}\in(0,1)$, and, for $\lambda\ge0$, let $S\sim\mathrm{Skellam}(\lambda(1+s^\star),\lambda(1-s^\star))$ and
  \[
    \mc B_S(\lambda) \,\defn\, \mb P_S[S<0]+\tfrac12\,\mb P_S[S=0].
  \]
  Then, with Hellinger shape parameter $\rho_s=\sqrt{1-(s^\star)^2}$ and $I_0$ the modified Bessel function of the first kind,
  \begin{equation}\label{eq:BS-integral}
    \mc B_S(\lambda) \;=\; s^\star\!\int_\lambda^\infty e^{-2u}I_0(2u\rho_s)\,\d u .
  \end{equation}
  In particular $\mc B_S$ is continuous and strictly decreasing on $[0,\infty)$ from $\mc B_S(0)=1/2$ to $\mc B_S(\infty)=0$, and its inverse $\lambda^\star(\beta)\defn\mc B_S^{-1}(\beta)$ is well-defined and continuous on $(0,1/2)$.
\end{lemma}

The proof is deferred to Appendix~\ref{app:deferred}.

The non-asymptotic bound of Proposition~\ref{prop:skellam-lb-nonasymp}, the Skellam approximation of Lemma~\ref{lem:tv-poisson}, and the Bayes-error properties of Lemma~\ref{lem:BS-properties} now combine into the fixed-confidence lower bound, first in its sequential limit at a fixed confidence level $\beta$.

\begin{proposition}[Fixed-confidence lower bound at fixed $\beta$]\label{prop:skellam-lb-fixed-beta}
  Let $\alpha\in(1,2)$ and let $\lambda^\star(\beta)\defn\mc B_S^{-1}(\beta)$ be the inverse of Lemma~\ref{lem:BS-properties}. For each $\beta\in(0,1/2)$ and every sequence of estimators $(\mu_n)$ with worst-case confidence level~\eqref{eq:beta-n-def} obeying $\beta_n(\mu_n;\Delta_n)\le\beta$ for all $n$, we must have
  \begin{equation}\label{eq:fixed-beta-liminf}
    \liminf_{n\to\infty}\, n^{(\alpha-1)/\alpha}\,\frac{\Delta_n}{\sigma} \;\geq\; s^\star\,\bigl(2\lambda^\star(\beta)\bigr)^{(\alpha-1)/\alpha}.
  \end{equation}
\end{proposition}

The proof is deferred to Appendix~\ref{app:deferred}.

It remains to evaluate the fixed-$\beta$ constant $s^\star(2\lambda^\star(\beta))^{(\alpha-1)/\alpha}$ as $\beta\downarrow0$, which we read off from the large-$\lambda$ behavior of $\mc B_S$ via its exact derivative in Lemma~\ref{lem:BS-properties}.

\begin{corollary}[Leading constant as $\beta\downarrow0$]\label{cor:skellam-lb-asymp}
  Let $\alpha\in(1,2)$. The fixed-$\beta$ constant of Proposition~\ref{prop:skellam-lb-fixed-beta} satisfies
  \begin{equation}\label{eq:beta-zero-asymp}
    s^\star\bigl(2\lambda^\star(\beta)\bigr)^{(\alpha-1)/\alpha} \;=\; L(\alpha)\,\bigl(\log(1/\beta)\bigr)^{(\alpha-1)/\alpha}\,(1+o(1)) \quad\text{as }\beta\downarrow 0.
  \end{equation}
\end{corollary}

The proof is deferred to Appendix~\ref{app:deferred}.

Combining the asymptotic~\eqref{eq:beta-zero-asymp} with the fixed-$\beta$ lower bound~\eqref{eq:fixed-beta-liminf} of Proposition~\ref{prop:skellam-lb-fixed-beta} gives the iterated-limits lower bound
\begin{equation}\label{eq:iterated-liminf}
  \liminf_{\beta\downarrow 0}\;\liminf_{n\to\infty}\,\frac{n^{(\alpha-1)/\alpha}\Delta_n}{\sigma\,\bigl(\log(1/\beta)\bigr)^{(\alpha-1)/\alpha}} \;\geq\; L(\alpha).
\end{equation}
This matches the achievable margin~\eqref{eq:bhatt-rate} to leading order: the synthesized $M$-estimator attains a confidence margin with the same constant $L(\alpha)$. The leading constant $L(\alpha)$ of \citet{lee2020optimal,bhatt2022nearly} is therefore tight, the direct $\alpha\in(1,2)$ analogue of the $\sqrt 2$-matching of \citet{catoni2012challenging} and \citet{devroye2015subgaussianmeanestimators} for the bounded-variance class.

\subsection{Slowly Varying Moment Classes}
\label{sec:slow-growth}

The bounded $\alpha$-moment class is characterized by a one-parameter family $\phi(x)=\abs{x}^\alpha$, indexed by $\alpha\in(1,2]$. Lowering the polynomial rate $\alpha$ enlarges the class and slows the confidence exponent: Proposition~\ref{prop:bhatt-synthesis} gives $r(\Delta)\asymp(\Delta/\sigma)^{\alpha/(\alpha-1)}$, an exponent whose order $\alpha/(\alpha-1)$ blows up as $\alpha\downarrow 1$. At the endpoint $\alpha=1$ the class degenerates to a bare first-moment bound, on which no positive concentration exponent is possible at small margins: the mirror pair $\mb P_{\pm\Delta}=(1-\varepsilon)\delta_0+\varepsilon\,\delta_{\pm\Delta/\varepsilon}$ lies in $\mc C_{\pm\Delta}$ whenever its centered first absolute moment $2(1-\varepsilon)\Delta\le B$ and has Hellinger affinity $1-\varepsilon$, so letting $\varepsilon\downarrow0$ gives $r(\Delta)=0$ for every $\Delta\le B/2$.

A natural intermediate regime replaces $\abs{x}^\alpha$ by a moment function $\phi$ that is superlinear but grows only marginally faster than linear at infinity. Such $\phi$ admit tails far heavier than any $\abs x^{1+\epsilon}$ with $\epsilon>0$: a distribution with $\mb P[\abs X>x]\asymp x^{-1}(\log x)^{-\gamma-2}$ has a finite $\abs x(\log\abs x)^\gamma$-moment but an infinite $\abs{x}^{1+\epsilon}$-moment for every $\epsilon>0$. We refer to any such class as a \emph{slow-growth} class.
The natural model is one in which the marginal growth $\phi(x)/\abs x$ is \emph{slowly varying}: it grows to infinity but does so subpolynomially, so that $\phi$ sits just above the critical linear scale. We make this precise in the following standing assumption, which augments the standing properties of $\phi$ (Section~\ref{ssec:contributions}) with an origin-regularity condition and a slowly-varying tail.

\begin{assumption}[Slowly varying moment class]\label{assu:slow-growth}
In addition to the standing properties of $\phi$ (Section~\ref{ssec:contributions}), assume:
\begin{enumerate}
\item[(A1)] The function $\phi$ vanishes on a neighborhood of the origin: $\phi(x)=0$ for $\abs x\le x_0$, for some $x_0>0$.
\item[(A2)] The function $\Lambda(x)\defn\phi(x)/\abs x$ admits a von Mises representation in the tail: there exist $x_\infty>1$, $\Lambda_0>0$, and a continuously differentiable function $\eta:[x_\infty,\infty)\to(0,\infty)$ such that
\[
  \Lambda(x)\;=\;\Lambda_0\exp\Bigl\{\int_{x_\infty}^{x}\tfrac{\eta(u)}{u}\,\d u\Bigr\},\qquad x\ge x_\infty ,
\]
where the index $\eta$ satisfies
\begin{enumerate}
  \item[(i)] $\eta(x)\to 0$;
  \item[(ii)] $\displaystyle\int_{x_\infty}^\infty\tfrac{\eta(u)}{u}\,\d u=\infty$;
  \item[(iii)] $x\eta'(x)/\eta(x)\to 0$ as $x\to\infty$.
\end{enumerate}
\end{enumerate}
\end{assumption}

Condition (A1) is a mild origin condition: $\phi$ is flat near the origin. Condition (A2) describes the tail entirely through the single infinitesimal index $\eta$: its positivity makes $\Lambda$ strictly increasing (so $\Lambda^{-1}$ is well-defined), while~(i)--(iii) read, respectively, as slow variation of $\Lambda=\phi(x)/\abs x$ (from $\eta(x)\to0$), the divergence $\Lambda(x)\to\infty$, and slow variation of $\eta$ itself. Both (A1) and (A2) hold for the compact zero-modifications of the loggamma tails $\abs x(\log\abs x)^\gamma$ ($\gamma>0$) and the iterated-logarithm tails described later.

\subsubsection{The Fixed-Margin Regime}
\label{sec:slow-margin}

Define the \emph{scale} $a(\Delta)>0$ and \emph{mass} $m(\Delta)\in(0,1)$ by
\begin{equation}\label{eq:slow-critical}
  m(\Delta)\,a(\Delta)\;=\;\Delta,\qquad m(\Delta)\,\phi(a(\Delta))\;=\;B,
\end{equation}
or equivalently, eliminating $m(\Delta)$ via $\phi(a(\Delta))/a(\Delta)=\Lambda(a(\Delta))=B/\Delta$,
\begin{equation}\label{eq:slow-critical-explicit}
  a(\Delta)\;=\;\Lambda^{-1}(B/\Delta),\qquad m(\Delta)\;=\;\frac{\Delta}{\Lambda^{-1}(B/\Delta)} .
\end{equation}
Fix once and for all a margin threshold $\Delta_0>0$ small enough that
\begin{equation}\label{eq:slow-Delta0}
  \Delta_0\;\le\;\min\{x_0,\,B/\Lambda(x_\infty)\}\qquad\text{and}\qquad m(\Delta_0)<1 .
\end{equation}
Such a $\Delta_0$ exists because $m(\Delta)=\Delta/\Lambda^{-1}(B/\Delta)$ is continuous and strictly increasing with $m(\Delta)\to0$ as $\Delta\downarrow0$ (the scale $a(\Delta)$ blows up while $\Delta\downarrow0$). For $\Delta\in(0,\Delta_0]$ the pair $(a(\Delta),m(\Delta))$ is then well-defined and unique with $m(\Delta)\in(0,1)$: under Assumption~(A2) the map $\Lambda$ is a continuous strictly increasing bijection of $[x_\infty,\infty)$ onto $[\Lambda(x_\infty),\infty)$, so $B/\Delta\ge\Lambda(x_\infty)$ gives $a(\Delta)=\Lambda^{-1}(B/\Delta)\ge x_\infty$ and, by monotonicity, $m(\Delta)\le m(\Delta_0)<1$; while $\Delta\le x_0$ ensures $\phi(\Delta)=0$, so that the origin atom contributes nothing to the moment budget. Moreover, the slow variation of $\Lambda$ makes $\Lambda^{-1}$ rapidly varying, so $a(\Delta)=\Lambda^{-1}(B/\Delta)\to\infty$ very fast as $\Delta\downarrow 0$, faster than any power of $1/\Delta$ (Lemma~\ref{lem:slow-inverse-growth}); consequently $m(\Delta)=\Delta/a(\Delta)\to 0$ faster than any power of $\Delta$.
For $\Delta\in(0,\Delta_0]$ consider the pair
\begin{equation}\label{eq:slow-pair}
  \mb P_{\pm\Delta}\defn(1-m(\Delta))\,\delta_0+m(\Delta)\,\delta_{\pm a(\Delta)}\in \mc C_{\pm\Delta}.
\end{equation}
Since $m(\Delta)\in(0,1)$ on $(0,\Delta_0]$ by~\eqref{eq:slow-Delta0}, each $\mb P_{\pm\Delta}$ is a genuine probability measure (a convex combination of $\delta_0$ and $\delta_{\pm a(\Delta)}$); every appearance of this pair below is confined to that range of margins, so this holds throughout.

\begin{lemma}[Upper bound via the mirror pair]\label{prop:slow-upper}
Under Assumption~\ref{assu:slow-growth}, $r(\Delta)\le -\log(1-m(\Delta))$ for every $\Delta\in(0,\Delta_0]$.
\end{lemma}

\begin{proof}
We first check that our proposed pair~\eqref{eq:slow-pair} is feasible, $\mb P_{\pm\Delta}\in\mc C_{\pm\Delta}$. The mean condition $\E{\mb P_{\pm\Delta}}{X}=\pm\Delta$ is $m(\Delta) a(\Delta)=\Delta$, satisfied by~\eqref{eq:slow-critical}. The moment budget is
\[
  \E{\mb P_{\pm\Delta}}{\phi(X\mp\Delta)}\;=\;(1-m(\Delta))\,\phi(\mp\Delta)\;+\;m(\Delta)\,\phi(\pm a(\Delta)\mp\Delta).
\]
Since $\Delta\le\Delta_0\le x_0$, the origin atom contributes nothing, as $\phi(\Delta)=0$ by~(A1); and because $\phi$ is even and non-decreasing on $[0,\infty)$ with $0<a(\Delta)-\Delta<a(\Delta)$,
\[
  \E{\mb P_{\pm\Delta}}{\phi(X\mp\Delta)}=m(\Delta)\,\phi(a(\Delta)-\Delta)\;\le\;m(\Delta)\,\phi(a(\Delta))\;=\;B,
\]
the last equality being~\eqref{eq:slow-critical}. Hence $\mb P_{\pm\Delta}\in\mc C_{\pm\Delta}$ for every $\Delta\in(0,\Delta_0]$, and with the affinity $\rho(\mb P_{-\Delta},\mb P_\Delta)=1-m(\Delta)$, \eqref{eq:intro-hellinger} gives
\[
  r(\Delta)\;\le\;-\log\rho(\mb P_{-\Delta},\mb P_\Delta)=-\log(1-m(\Delta)).\qedhere
\]
\end{proof}

The two-atom hypothesis pair $\mb P_{\pm\Delta}$ should here be contrasted with its counterpart in the $\alpha$-moment regime $\alpha\in(1,2)$ of Section~\ref{sec:bounded-alpha-moment-1}. There the least-favorable pair is genuinely \emph{three}-atom, $\mb P^\star_{\pm\Delta}=p_-\delta_{\mp a}+p_0\delta_0+p_+\delta_{\pm a}$, splitting its vanishing rare mass over \emph{both} scales $\pm a$ as $p_\pm=(1\pm s^\star)m$ with $s^\star=\sqrt{\alpha(2-\alpha)}\in(0,1)$ strictly interior; the two hypotheses then overlap on both rare atoms, the affinity retains the cross term $2\sqrt{p_+p_-}=2(\alpha-1)m$, and separating them required the Skellam analysis of Section~\ref{sec:bounded-alpha-moment-1}. As $\alpha\downarrow1$ we get $s^\star\uparrow1$: the lighter rare atom $p_-=(1-s^\star)m$ disappears, so apart from the shared origin the two hypotheses keep only a single \emph{disjoint} rare atom each, at $+a(\Delta)$ and $-a(\Delta)$, and the affinity reduces to the shared mass $1-m(\Delta)$ with no cross term.

The upper bound in Lemma \ref{prop:slow-upper} is matched, to leading order, by a Hellinger separation lower bound.

\begin{theorem}[Sharp slow-growth exponent]\label{thm:slow-sharp}
Under Assumption~\ref{assu:slow-growth},
\[
  r(\Delta)\;\sim\;m(\Delta),\qquad \Delta\downarrow 0 .
\]
\end{theorem}

\begin{proof}
An upper bound $r(\Delta)\le m(\Delta)(1+o(1))$ follows from Lemma~\ref{prop:slow-upper}. For the matching lower bound, the definition~\eqref{eq:intro-hellinger} of $r(\Delta)$ and the elementary inequality $-\log u\ge 1-u$ give
\[
  r(\Delta)\;=\;-\log\!\!\sup_{\mb P_-\in\mc C_{-\Delta},\,\mb P_+\in\mc C_{\Delta}}\!\!\rho(\mb P_-,\mb P_+)\;=\;\inf_{\mb P_-\in\mc C_{-\Delta},\,\mb P_+\in\mc C_{\Delta}}\!\bigl(-\log\rho(\mb P_-,\mb P_+)\bigr)\;\ge\;\inf_{\mb P_-\in\mc C_{-\Delta},\,\mb P_+\in\mc C_{\Delta}}\!\bigl(1-\rho(\mb P_-,\mb P_+)\bigr) .
\]
By Lemma~\ref{prop:slow-hellinger} (specialized to $\mu=0$) the right-hand side is at least $m(\Delta)(1+o(1))$, whence $r(\Delta)\ge m(\Delta)(1+o(1))$. Combining the two bounds gives $r(\Delta)\sim m(\Delta)$.
\end{proof}

\begin{lemma}[Hellinger separation modulus]\label{prop:slow-hellinger}
Under Assumption~\ref{assu:slow-growth},
\[
  \liminf_{\Delta\downarrow0}\;
  \inf_{\mu\in\Re}\;
  \inf_{\mb P_{\pm}\in\mc C_{\mu\pm\Delta}}\;
  \frac{1-\rho(\mb P_-,\mb P_+)}{m(\Delta)}\;\ge\;1 .
\]
\end{lemma}
\begin{proof}
  The proof of Lemma~\ref{prop:slow-hellinger}, together with the supporting deterministic estimates for the slow-growth tail, is given in Appendix~\ref{app:slow-growth}.
\end{proof}

As opposed to the classes of Sections~\ref{sec:bound-vari-class}--\ref{sec:bounded-alpha-moment-1}, we do not construct an estimator attaining the rate $r(\Delta)$; we only invoke the matching theorem~\ref{thm:matching}, which guarantees an $M$-estimator attaining $r(\Delta)$ to \emph{exist}, and even pins it down as the solution of the two-dimensional convex optimization problem over the multipliers $(\lambda_1,\lambda_2)\in\Re\times\Re_+$ of Section~\ref{ssec:sandwich}.
Nevertheless, we must point out here that the exponent $r(\Delta)$ is genuinely sub-polynomial in $\Delta$: as $\Delta\downarrow0$ it decays like $m(\Delta)=\Delta/\Lambda^{-1}(B/\Delta)$, far faster than the $(\Delta/\sigma)^{\alpha/(\alpha-1)}$ of the bounded $\alpha$-moment class, reflecting how much harder concentration becomes once the moment function only marginally exceeds linear growth. To make the latter more tangible we consider two particular examples.

\paragraph{Examples.}
The following canonical slow tails all satisfy Assumption~\ref{assu:slow-growth}.

\emph{Log-gamma tails.}
Fix $\gamma>0$ and let $\phi$ be any even function that is flat near the origin, agreeing with $\abs x(\log\abs x)^\gamma$ for large $\abs x$ (a compact zero-modification of the loggamma tail). Then $\Lambda(x)=(\log x)^\gamma$, with $\eta(x)=\gamma/\log x\to0$, $\int^\infty\!\eta(u)/u\,\d u=\gamma\log\log x\to\infty$, and $x\eta'(x)/\eta(x)=-1/\log x\to0$. This verifies (A2). Since $\Lambda^{-1}(z)=\exp(z^{1/\gamma})$, Theorem~\ref{thm:slow-sharp} gives the following.

\begin{corollary}[Log-gamma slow growth]\label{cor:slow-loggamma}
For every $\gamma>0$ and every $\phi$ as above,
\[
  r(\Delta)\;\sim\;\Delta\,\exp\bigl\{-(B/\Delta)^{1/\gamma}\bigr\},\qquad \Delta\downarrow0 .
\]
\end{corollary}

\emph{Iterated logarithms.}
The assumption also covers the still slower tails $\Lambda(x)=\log^{\circ k}x$ for any fixed $k\ge1$. Write $L_j(x)\defn\log x\,\log^{\circ2}x\cdots\log^{\circ j}x$ for the product of the first $j$ iterated logarithms (with $L_0\defn1$). A direct computation gives, for large $x$,
\[
  \eta(x)=\frac{1}{L_k(x)}\to0,\qquad
  \int^\infty\frac{\eta(u)}{u}\,\d u=\log^{\circ(k+1)}x\to\infty,\qquad
  \frac{x\eta'(x)}{\eta(x)}=-\sum_{j=1}^{k}\frac{1}{L_j(x)}\to0,
\]
which verifies (A2) for every $k\ge1$. Writing $\exp^{\circ k}$ for the inverse of $\log^{\circ k}$, we have $\Lambda^{-1}(z)=\exp^{\circ k}(z)$, so $a(\Delta)=\exp^{\circ k}(B/\Delta)$ and Theorem~\ref{thm:slow-sharp} gives the following.

\begin{corollary}[Iterated-logarithm slow growth]\label{cor:slow-iterated}
For every $k\ge1$ and every even $\phi$ that is flat near the origin with $\Lambda(x)=\log^{\circ k}x$ for large $\abs x$,
\[
  r(\Delta)\;\sim\;\frac{\Delta}{\exp^{\circ k}(B/\Delta)},\qquad \Delta\downarrow0 .
\]
\end{corollary}

\subsubsection{The Fixed-Confidence Regime}
\label{sec:slow-fixed-beta}

We start by recording the margin required by the fixed-margin-optimal $M$-estimator, the one the matching theorem~\ref{thm:matching} guarantees to exist without us exhibiting it explicitly, to achieve a fixed confidence level, and then prove a matching lower bound.

\begin{corollary}[Achievable fixed-$\beta$ margin]\label{cor:slow-achievable}
Fix $\beta\in(0,1)$ and write $\xi_n\defn\log(1/\beta)/n$. The fixed-margin-optimal $M$-estimator of Section~\ref{ssec:sandwich} attains confidence level $\beta$ at margin
\[
  \Delta_n\;=\;m^{-1}(\xi_n)\,(1+o(1)),\qquad n\to\infty .
\]
\end{corollary}

The proof is deferred to Appendix~\ref{app:deferred}.

For the matching lower bound, the disjointness of the rare atoms in the near least-favorable pair~\eqref{eq:slow-pair} makes the testing problem collapse to an elementary binomial: a single non-zero observation reveals the hypothesis.

\begin{proposition}[Non-asymptotic fixed-$\beta$ lower bound]\label{prop:slow-fixed-beta}
For every $n\ge 1$, every $\beta\in(0,\tfrac12)$, every margin $\Delta_n\in(0,\Delta_0)$, and every measurable estimator $\mu_n$ with worst-case confidence level~\eqref{eq:beta-n-def} obeying $\beta_n(\mu_n;\Delta_n)\le\beta$, we must have
\[
  \beta\;\ge\;\tfrac12\,(1-m(\Delta_n))^{\,n}.
\]
\end{proposition}

The proof is deferred to Appendix~\ref{app:deferred}.

The bound is fully non-asymptotic: $\tfrac12(1-m(\Delta_n))^n$ is the Bayes error of a coin informative with probability $m(\Delta_n)$ per draw. Inverting it yields a lower bound on the attainable margin.

\begin{corollary}[Fixed-$\beta$ margin lower bound]\label{cor:slow-fixed-beta}
Fix $\beta\in(0,\tfrac12)$ and let $\xi_n=\log(1/\beta)/n$ as in Corollary~\ref{cor:slow-achievable}. Any estimator sequence $(\mu_n)$ with $\beta_n(\mu_n;\Delta_n)\le\beta$ for all $n$ must satisfy
\[
  \Delta_n\;\ge\;m^{-1}(\xi_n)\,(1-o(1)),\qquad n\to\infty .
\]
\end{corollary}

The proof is deferred to Appendix~\ref{app:deferred}.

The lower bound matches the margin of the fixed-margin-optimal $M$-estimator of Corollary~\ref{cor:slow-achievable} to leading order. The fixed-margin-optimal $M$-estimator is therefore also fixed-confidence optimal on the slow-growth class, attaining the minimax margin $m^{-1}(\xi_n)$.

The matching here is stronger than in the earlier classes. For the bounded-variance and bounded $\alpha$-moment classes of Sections~\ref{sec:bound-vari-class}--\ref{sec:bounded-alpha-moment-1} the achievable and minimax margins agree on their leading {constant} only in the high-confidence limit $\beta\downarrow0$ (Catoni's $\sqrt2$ and the $L(\alpha)$ of \citet{lee2020optimal,bhatt2022nearly}); on the slow-growth class they instead match to leading order at \emph{every} fixed $\beta\in(0,\tfrac12)$. The strengthening is real but modest: it looks \emph{much} stronger than it is, because the leading constant does not depend on $\beta$ at all, leaving no $\beta$-dependent constant to match across. Indeed, $\beta$ enters the budget $\xi_n=\log(1/\beta)/n$ only as the multiplicative constant $\log(1/\beta)$, while $m^{-1}(\xi)=\xi\,\phi^{-1}(B/\xi)$ is slowly varying as $\xi\downarrow0$ and hence asymptotically invariant under constant rescaling of its argument, $m^{-1}(c\,\xi_n)\sim m^{-1}(\xi_n)$ for every fixed $c>0$.
Specializing to the two tails of Section~\ref{sec:slow-margin} makes this margin concrete: the log-gamma tail $\phi(x)\sim\abs x(\log\abs x)^\gamma$ gives $\Delta_n\sim B/(\log n)^\gamma$, and the iterated-logarithm tail $\Lambda(x)=\log^{\circ k}x$ the still slower $\Delta_n\sim B/\log^{\circ k} n$. In both, the confidence $\beta$ has dropped out of the leading term, so the attainable margin is dictated by the sample size alone, and shrinks only (iterated-)logarithmically in $n$.

\section{Discussion}
\label{sec:discussion}

We have kept the development to one-dimensional location, symmetric moment classes, and i.i.d.\ data in order to expose the convex mechanism cleanly. In discussing extensions it is worth separating the two halves of the result. The \emph{construction} of the estimator, and the upper bound it certifies, rests only on the Chernoff bound and the Lagrangian-duality argument of Section~\ref{sec:synthesis}, and carries over to several richer settings essentially unchanged. The \emph{matching lower bound}, by contrast, is tied to the one-dimensional symmetric i.i.d.\ structure through the two-point Hellinger argument of Section~\ref{sec:lower-bound}; whether it extends is a genuine question in most of the settings below, not a corollary of the present analysis.

\paragraph{Extensions of the synthesis.} Three generalizations leave the convex synthesis problem intact, so it continues to characterize the achievable rate. For \emph{asymmetric tails}, the Lagrangian-dual derivation goes through whenever $\psi$ is non-decreasing away from the origin, now with possibly different envelopes on the two sides. For \emph{multiple moment constraints}, imposing $\int\phi_k(X)\,\d\mb P\le B_k$ for $k=1,\dots,K$ adds one multiplier per constraint and leaves the sandwich structure unchanged. For \emph{martingale data}, the exponential supermartingales $M_k^\pm$ of Remark~\ref{rmk:martingale} extend the confidence guarantee of Theorem~\ref{thm:feasibility-guarantee} to martingale-difference sequences; replacing the Chernoff bound by Ville's maximal inequality $\mb P[\sup_{k\geq 0}M_k^\pm\ge1/\beta]\le\beta$ moreover renders it \emph{anytime-valid}, holding simultaneously over all sample sizes and hence under optional stopping, as in the extensions of \citet{catoni2012challenging} and \citet{bhatt2022nearly}.

\paragraph{Open directions for matching optimality.} In other settings the estimator extends but exact optimality does not follow from our analysis. For \emph{multivariate location}, a one-dimensional confidence estimator can be lifted through the median-of-projections reduction of \citet{prasad2020robust}; for \emph{regression}, one may apply the synthesis to the residuals of a regression $M$-estimator. In both cases the estimator is available, but the two-point lower bound and the resulting exact match are particular to the one-dimensional symmetric problem, and whether the synthesized rate remains optimal is open. 

\section{Conclusions}
\label{sec:conclusions}

We have shown that, for symmetric moment classes, the synthesis of a confidence-optimal $M$-estimator reduces to a tractable convex optimization problem over a pair of dual multipliers, and that the resulting estimating function is sandwiched between two explicit envelopes determined by these multipliers. The synthesis rate matches a two-point Hellinger lower bound that holds for \emph{every} measurable estimator, yielding matching exponential upper and lower bounds on the fixed-margin confidence rate for general symmetric moment classes.

Specializing the framework recovers Catoni's optimal sub-Gaussian constant for the bounded-variance class and extends it to the heavier-tailed bounded $\alpha$-moment classes, sharpening the lower-bound constant of \citet{devroye2015subgaussianmeanestimators} to the exact value attained at a (near) least-favorable pair and thereby establishing the leading constant of \citet{lee2020optimal,bhatt2022nearly} as tight. Pushing beyond polynomial tails, we treat the slowly varying moment classes, a substantially heavier regime in which the moment function only marginally exceeds linear growth and the optimal exponent degenerates to a sub-polynomial rate. Across these classes the (near) least-favorable distributions driving the matching bounds take three qualitatively distinct shapes: a symmetric two-atom pair with vanishing bias in the light-tailed regime, analyzed through a central limit theorem; a three-atom pair carrying an additional central mass and a strictly interior bias in the heavy polynomial regime, requiring a Skellam analysis; and a degenerate two-atom pair retaining only the central mass and a single informative outlier in the slowly varying regime, where an elementary binomial argument suffices. For each of these concrete classes the fixed-margin-optimal estimator also matches the fixed-confidence margin to leading order: the sharp high-confidence leading constant for the bounded-variance and bounded $\alpha$-moment classes, and fixed-$\beta$ leading-order minimaxity for the slowly varying classes.

\paragraph*{Acknowledgements}
Bart P.G.\ van Parys gratefully acknowledges funding from NWO Vidi grant VI.Vidi.243.021.
The authors used Anthropic's Claude for copyediting.

\bibliography{main}

\appendix

\section{Deferred Proofs}
\label{app:deferred}

This appendix collects the proofs deferred from the main text, organized by the section in which each result is stated.

\subsection{Proofs of Section~\ref{sec:synthesis}}

\begin{proof}[Proof of Lemma~\ref{lem:lagrangian-dual}]
Since $L$ is measurable and bounded below, $\E{\mb P}{L(X)}\in(-\infty,+\infty]$ is well defined for every $\mb P\in\mc C_0$, as are the constraint integrals $\E{\mb P}{X}=0$ and $\E{\mb P}{\phi(X)}\le B<\infty$ that define $\mc C_0$. Fix $\mb P\in\mc C_0$ and $(\lambda_1,\lambda_2)\in\Re\times\Re_+$, and set $M:=\sup_{x\in\Re}\bigl(L(x)-\lambda_1 x-\lambda_2(\phi(x)-B)\bigr)$. If $M=+\infty$ the bound is trivial; otherwise $L(x)\le M+\lambda_1 x+\lambda_2(\phi(x)-B)$ for all $x$, so integrating against $\mb P$ and using $\E{\mb P}{X}=0$ together with $\lambda_2\bigl(\E{\mb P}{\phi(X)}-B\bigr)\le 0$ gives $\E{\mb P}{L(X)}\le M$. Taking the supremum over $\mb P\in\mc C_0$ and then the infimum over $(\lambda_1,\lambda_2)$ proves the inequality.
\end{proof}

\subsection{Proofs of Section~\ref{sec:lower-bound}}

\begin{proof}[Proof of Proposition~\ref{prop:relax-gap}]
Since $\mc C_{\pm\Delta}\subseteq\mc C^{\mathrm{relax}}_{\pm\Delta}$, $\sup_{\mc C^{\mathrm{relax}}}\rho\ge\sup_{\mc C}\rho$ and hence $r_{\mathrm{relax}}(\Delta)\le r(\Delta)$. For the reverse, fix $\epsilon>0$ and, using Proposition~\ref{prop:mirror} (the relaxed classes $\mc C^{\mathrm{relax}}_{\pm\Delta}$ are convex and reflection-conjugate, since $\mc C_0=\bar{\mc C}_0$), a sign-symmetric pair $(\mb P^\star_{-\Delta}, \mb P^\star_{\Delta})$ feasible in~\eqref{eq:rstar-relax} with $\rho(\mb P^\star_{-\Delta},\mb P^\star_\Delta)\ge e^{-r_{\mathrm{relax}}(\Delta)}-\epsilon$. It suffices to obtain $\mb P^\star_{\pm\Delta}\in\mc C_{\pm\Delta}$ without decreasing $\rho$, since the resulting pair is then feasible in~\eqref{eq:intro-hellinger} and gives $e^{-r(\Delta)}\ge e^{-r_{\mathrm{relax}}(\Delta)}-\epsilon$. \emph{Unit mass.} If $\int\d\mb P^\star_{\pm\Delta}<1$, add an atom of the deficit mass at $\pm\Delta$: there the mean integrand $X\mp\Delta$ vanishes and the moment integrand $\phi(X\mp\Delta)-B$ equals $\phi(0)-B<0$, so $\int(X\mp\Delta)\,\d\mb P^\star_{\pm\Delta}$ is unchanged and $\int(\phi(X\mp\Delta)-B)\,\d\mb P^\star_{\pm\Delta}$ only decreases (both constraints remain satisfied) while $\rho$ does not decrease (monotone in each argument) and sign-symmetry is preserved. We may thus take $\mb P^\star_{\pm\Delta}$ to be probability measures with $\E{\mb P^\star_{-\Delta}}{X}\le-\Delta$, $\E{\mb P^\star_\Delta}{X}\ge\Delta$ and $\E{\mb P^\star_{\pm\Delta}}{\phi(X\mp\Delta)}\le B$. \emph{Tight mean.} If $\E{\mb P^\star_{-\Delta}}{X}=-\Delta$ then $\mb P^\star_{\pm\Delta}\in\mc C_{\pm\Delta}$ already; otherwise $\E{\mb P^\star_{-\Delta}}{X}=-\Delta-\delta$ for some $\delta> 0$, and we tighten the mean by a shrinkage map. For $t\geq 0$ define the measurable shrinkage map
\[
  f_t(x) \;:=\; \begin{cases} \min(x+t,-\Delta) & x<-\Delta,\\ x & x\in[-\Delta,\Delta],\\ \max(x-t,\Delta) & x>\Delta, \end{cases}
\]
which pushes the tails inward by amount $t$ but never crosses into the interior $(-\Delta,\Delta)$, and let $T_t\mb P:=(f_t)_*\mb P$. Three properties hold for any $t\geq 0$:

(i) \emph{Affinity non-decreasing:} pushing both measures through the common map $f_t$ can only increase the affinity, $\rho(T_t\mb P^\star_{-\Delta}, T_t\mb P^\star_\Delta)\ge\rho(\mb P^\star_{-\Delta},\mb P^\star_\Delta)$. This is the data-processing inequality, in the direction fixed by the identity $H^2(\mb P,\mb Q)=2\bigl(1-\rho(\mb P,\mb Q)\bigr)$: the squared Hellinger distance is the $f$-divergence with convex $f(t)=(\sqrt t-1)^2$, hence non-increasing under the common pushforward $(f_t)_*$, so the affinity $\rho=1-\tfrac12H^2$ is non-decreasing.

(ii) \emph{Moments slacken:} $|f_t(x)+\Delta|\le|x+\Delta|$ for all $x$ (direct case-check), so by evenness and monotonicity of $\phi$ on $[0,\infty)$,
\[
  \int\!\phi(y+\Delta)\,\d(T_t\mb P^\star_{-\Delta})(y) \;=\; \int\!\phi(f_t(x)+\Delta)\,\d\mb P^\star_{-\Delta}(x) \;\le\; \int\!\phi(x+\Delta)\,\d\mb P^\star_{-\Delta}(x) \;\le\; B,
\]
and symmetrically for $\mb P^\star_\Delta$.

(iii) \emph{Mean continuous and crosses $-\Delta$:} the map $t\mapsto m(t):=\E{T_t\mb P^\star_{-\Delta}}{X}$ is continuous (dominated convergence with $|f_t(x)|\le|x|\vee\Delta$ integrable under $\mb P^\star_{-\Delta}$), with $m(0)=-\Delta-\delta<-\Delta$ and $\lim_{t\to\infty}m(t)=L$ where $L=\int h\,\d\mb P^\star_{-\Delta}\ge -\Delta$ with $h(x)\defn(-\Delta)\vee(x\wedge\Delta)$ the pointwise limit of $f_t$. The inequality $L\ge -\Delta$ is strict whenever $\mb P^\star_{-\Delta}$ has positive mass in $(-\Delta,\infty)$, which must hold whenever $\rho(\mb P^\star_{-\Delta},\mb P^\star_\Delta)>0$: otherwise $\mb P^\star_{-\Delta}$ would be supported in $(-\infty,-\Delta]$ and by mirror symmetry $\mb P^\star_\Delta$ in $[\Delta,\infty)$, disjoint supports forcing $\rho(\mb P^\star_{-\Delta},\mb P^\star_\Delta)=0$. Hence $L>-\Delta$, and by the intermediate-value theorem there exists $t^\star\geq 0$ with $m(t^\star)=-\Delta$; by sign-symmetry the corresponding mean of $T_{t^\star}\mb P^\star_\Delta$ equals $\Delta$.

The pair $(T_{t^\star}\mb P^\star_{-\Delta}, T_{t^\star}\mb P^\star_\Delta)$ is feasible in~\eqref{eq:intro-hellinger} with $\rho\ge e^{-r_{\mathrm{relax}}(\Delta)}-\epsilon$, so $e^{-r(\Delta)}\ge e^{-r_{\mathrm{relax}}(\Delta)}-\epsilon$. Letting $\epsilon\downarrow0$ gives $r(\Delta)\le r_{\mathrm{relax}}(\Delta)$.
\end{proof}

\subsection{Proofs of Section~\ref{sec:illustrations-2}}

\subsubsection{Proofs of Section~\ref{sec:bounded-alpha-moment-large}}

\begin{proof}[Proof of Proposition~\ref{prop:CLT-lb-fixed-beta}]
Let $\mb P^\star_{\pm\Delta_n}$ be the two-atom near least-favorable pair of Proposition~\ref{prop:two-atom-large-alpha}, available whenever $\Delta_n\le\Delta_0$. Steps 1--3 below show that $\beta\ge\mc B^{\mathrm{bin}}_n(\Delta_n)$ for every such $n$, by bounding the confidence of any estimator from below by an explicit two-point Bayes error.

\emph{Step 1 (two-point reduction).} The non-asymptotic inequality~\eqref{eq:LB-reduction} from the proof of Proposition~\ref{prop:LB}, applied to the pair $\mb P^\star_{\pm\Delta'}$ at a margin $\Delta'>\Delta_n$ where the near least-favorable pair still exists, bounds $\beta$ from below by the Bayes error $\mc R^\star_n(\mb P^\star_{-\Delta'},\mb P^\star_{\Delta'})$ of testing the two hypotheses on $n$ i.i.d.\ samples under a uniform prior (see Equation~\eqref{eq:bayes-error-def}).

\emph{Step 2 (the optimal test is a sign test).} Because each hypothesis is supported on only the two atoms $\pm a(\Delta')$, recording which atom each sample hits turns the data into $n$ i.i.d.\ signs, and the number $N_+$ of $+a(\Delta')$ outcomes is a sufficient statistic. The problem is thus to distinguish two binomials,
\[
  N_+\sim\mathrm{Bin}\bigl(n,\tfrac{(1+s(\Delta'))}{2}\bigr)\ \text{under}\ \mb P^\star_{\Delta'}\qquad\text{versus}\qquad N_+\sim\mathrm{Bin}\bigl(n,\tfrac{(1-s(\Delta'))}{2}\bigr)\ \text{under}\ \mb P^\star_{-\Delta'},
\]
two mirror binomials with success probabilities $\tfrac{(1\pm s(\Delta'))}{2}$. Writing $s\defn s(\Delta')$, the likelihood ratio between the two hypotheses at the observed count $N_+=k$ is
\[
  \frac{\d\mb P^\star_{\Delta'}}{\d\mb P^\star_{-\Delta'}}(k)\;=\;\frac{\bigl(\tfrac{(1+s)}{2}\bigr)^{k}\bigl(\tfrac{(1-s)}{2}\bigr)^{n-k}}{\bigl(\tfrac{(1-s)}{2}\bigr)^{k}\bigl(\tfrac{(1+s)}{2}\bigr)^{n-k}}\;=\;\Bigl(\tfrac{(1+s)}{(1-s)}\Bigr)^{2k-n},
\]
which, since $\tfrac{(1+s)}{(1-s)}>1$, is strictly increasing in the centered count $S_n\defn\sum_{i=1}^n X_i/a(\Delta')=2N_+-n\in\{-n,\dots,n\}$. The Bayes-optimal test of Step~1 therefore thresholds $S_n$; at the uniform prior the threshold sits at likelihood ratio $1$, i.e.\ $S_n=0$ with fair-coin tie-breaking, the standard sign test. Its error is the Bayes error of Step~1, now in closed form, $\mc R^\star_n(\mb P^\star_{-\Delta'},\mb P^\star_{\Delta'})=\mb P^\star_{\Delta'}[S_n<0]+\tfrac12\mb P^\star_{\Delta'}[S_n=0]$. We will only need the weaker lower bound obtained by discarding the non-negative tie term, namely the binomial left tail
\[
  \mc B^{\mathrm{bin}}_n(\Delta')\;\defn\;\mb P^\star_{\Delta'}[S_n<0]=\Pr\bigl[N_+<\tfrac n2\bigr]\;\le\;\mc R^\star_n(\mb P^\star_{-\Delta'},\mb P^\star_{\Delta'}).
\]

\emph{Step 3 (pass to the margin $\Delta_n$).} For fixed $n$, $\mc B^{\mathrm{bin}}_n(\Delta')=\sum_{k<n/2}\binom nk p^k(1-p)^{n-k}$ is a finite sum of binomial probabilities over the $\Delta'$-independent range $\{k:k<n/2\}$, with success probability $p=\tfrac{(1+s(\Delta'))}{2}$; it is therefore a polynomial in $p$, hence continuous in $s(\Delta')$. Since $\Delta'\mapsto s(\Delta')$ is continuous (Proposition~\ref{prop:two-atom-large-alpha}), so is $\Delta'\mapsto\mc B^{\mathrm{bin}}_n(\Delta')$, and letting $\Delta'\downarrow\Delta_n$ gives
\[
\beta\;\ge\;\mc B^{\mathrm{bin}}_n(\Delta_n),\qquad S_n=\sum_{i=1}^n X_i/a(\Delta_n).
\]
Under $\mb P^\star_{\Delta_n}$ this $S_n$ is a sum of $n$ i.i.d.\ $\pm 1$ variables of bias $s_n\defn s(\Delta_n)$.

Let $z^\star\defn\liminf_{n\to\infty}\sqrt n\,\Delta_n/\sigma\in[0,\infty]$. If $z^\star=\infty$ the bound is immediate, so assume $z^\star<\infty$ and pass to a subsequence $n_k\to\infty$ along which $\sqrt{n_k}\,\Delta_{n_k}/\sigma\to z^\star$. Then $\Delta_{n_k}\to 0$, so $\Delta_{n_k}\le\Delta_0$ for $k$ large and Steps~1--3 give $\beta\ge\mc B^{\mathrm{bin}}_{n_k}(\Delta_{n_k})$. By Proposition~\ref{prop:two-atom-large-alpha} the biases satisfy $s_{n_k}=\Delta_{n_k}/\sigma\,(1+o(1))$, so $\sqrt{n_k}\,s_{n_k}\to z^\star$. For each $k$, write $S_{n_k}=\sum_{i=1}^{n_k}Y^{(k)}_i$ with $Y^{(k)}_i\in\{-1,+1\}$ iid, $\Pr[Y_i^{(k)}=+1]=(1+s_{n_k})/2$, so $\E{}{Y^{(k)}_i}=s_{n_k}$ and $\mathrm{Var}(Y^{(k)}_i)=1-s_{n_k}^2$. Center and rescale each row to unit total variance,
\[
  X_{k,i}\;\defn\;\frac{Y^{(k)}_i-s_{n_k}}{\sqrt{n_k(1-s_{n_k}^2)}},\qquad 1\le i\le n_k,
\]
so that $\E{}{X_{k,i}}=0$, $\mathrm{Var}(X_{k,i})=1/n_k$, and $\sum_{i=1}^{n_k}\mathrm{Var}(X_{k,i})=1$. The entries in our triangular array are uniformly small: $|X_{k,i}|\le c_k\defn 2/\sqrt{n_k(1-s_{n_k}^2)}$, with $c_k\to 0$ since $n_k\to\infty$ and $s_{n_k}\to 0$. Hence for every $\epsilon>0$, once $c_k\le\epsilon$ the event $|X_{k,i}|>\epsilon$ does not occur and the Lindeberg-Feller condition holds, i.e., $\sum_{i=1}^{n_k}\E{}{X_{k,i}^2\,\mathbf 1\{|X_{k,i}|>\epsilon\}}\to 0$. The Lindeberg--Feller CLT therefore gives
\[
  \sum_{i=1}^{n_k}X_{k,i}=\frac{S_{n_k}-n_k s_{n_k}}{\sqrt{n_k(1-s_{n_k}^2)}}\;\Rightarrow\;N(0,1)\quad\text{under }\mb P^\star_{\Delta_{n_k}},
\]
Write $W_k\defn\sum_{i=1}^{n_k}X_{k,i}$ for this normalized sum and $t_k\defn n_k s_{n_k}/\sqrt{n_k(1-s_{n_k}^2)}=\sqrt{n_k}\,s_{n_k}/\sqrt{1-s_{n_k}^2}$ for the standardized mean, so that $\{S_{n_k}<0\}=\{W_k<-t_k\}$ and hence
\[
  \mc B^{\mathrm{bin}}_{n_k}(\Delta_{n_k})=\mb P^\star_{\Delta_{n_k}}[W_k<-t_k].
\]
Since $s_{n_k}\to 0$ and $\sqrt{n_k}s_{n_k}\to z^\star$ we have $t_k\to z^\star$; as $W_k\Rightarrow N(0,1)$ and the standard normal distribution function $\Phi$ is continuous, the distribution functions converge uniformly (P\'olya's theorem), so evaluating at the convergent thresholds $-t_k\to-z^\star$ gives $\mb P^\star_{\Delta_{n_k}}[W_k<-t_k]\to\Phi(-z^\star)$, whence
\[
\mc B^{\mathrm{bin}}_{n_k}(\Delta_{n_k})\;\longrightarrow\;\Phi(-z^\star).
\]
Since $\beta\ge\mc B^{\mathrm{bin}}_{n_k}(\Delta_{n_k})$ for every $k$, letting $k\to\infty$ gives $\beta\ge\Phi(-z^\star)$, that is $z^\star\ge\Phi^{-1}(1-\beta)$. Hence $\liminf_{n\to\infty}\sqrt n\,\Delta_n/\sigma\ge\Phi^{-1}(1-\beta)$.
\end{proof}

\subsubsection{Proofs of Section~\ref{sec:bounded-alpha-moment-1}}

\begin{proof}[Proof of Lemma~\ref{lem:F-star-asymp}]
  By construction $\E{\mb P^\star_{\pm\Delta}}{X}=\pm(p_+-p_-)a$, so setting $d\defn p_+-p_->0$ and $a\defn\Delta/d$ fixes the mean. The $\alpha$-moment about $\pm\Delta$ then evaluates to
  \[
    \E{\mb P^\star_\Delta}{\abs{X-\Delta}^\alpha} \,=\, \Delta^\alpha\!\left[\frac{p_-(1+d)^\alpha+p_+(1-d)^\alpha}{d^\alpha} + p_0\right],
  \]
  with the mirror identity for $\mb P^\star_{-\Delta}$. Choosing $m=m(\varepsilon)$ so that this $\alpha$-moment meets the budget $\Delta^\alpha\varepsilon^{-\alpha}$ with equality reduces feasibility to the single equation
  \begin{equation}\label{eq:two-point-alpha-constraint}
    \frac{p_-(1+d)^\alpha+p_+(1-d)^\alpha}{d^\alpha} + p_0 \;=\; \varepsilon^{-\alpha}.
  \end{equation}

  Set $s^\star\defn\sqrt{\alpha(2-\alpha)}\in(0,1)$ and parametrize the pair $\mb P_{\pm \Delta}$ by $m\in(0,\tfrac{1}{2})$ via $p_\pm=(1\pm s^\star)m$, $p_0=1-2m$, so that $d=2ms^\star$. Since $p_\pm=(1\pm s^\star)m$,
  \[
    p_-(1+d)^\alpha+p_+(1-d)^\alpha=m\bigl[(1-s^\star)(1+d)^\alpha+(1+s^\star)(1-d)^\alpha\bigr]=2m\,\tilde h(m),
  \]
  where $\tilde h(m)\defn\tfrac12\bigl[(1-s^\star)(1+2ms^\star)^\alpha+(1+s^\star)(1-2ms^\star)^\alpha\bigr]$ is $C^\infty$ on $[0,\tfrac12)$ with $\tilde h(0)=1$. As $d^\alpha=(2ms^\star)^\alpha=(2m)^\alpha(s^\star)^\alpha$, the left-hand side of~\eqref{eq:two-point-alpha-constraint} is
  \[
    \frac{p_-(1+d)^\alpha+p_+(1-d)^\alpha}{d^\alpha}+p_0=\frac{2m\,\tilde h(m)}{(2m)^\alpha(s^\star)^\alpha}+(1-2m)=\frac{(2m)^{1-\alpha}\tilde h(m)}{(s^\star)^\alpha}+1-2m.
  \]

  \emph{Existence and leading order via a rescaled implicit function theorem.} We solve the moment constraint~\eqref{eq:two-point-alpha-constraint}, now in the form
  \[
    \frac{(2m)^{1-\alpha}\tilde h(m)}{(s^\star)^\alpha}+1-2m=\varepsilon^{-\alpha},
  \]
  for $m=m(\varepsilon)$, after rescaling to the anticipated order by $t\defn 2m/\varepsilon^q$. Since $q(1-\alpha)=-\alpha$ we have $(2m)^{1-\alpha}=t^{1-\alpha}\varepsilon^{-\alpha}$, and multiplying through by $\varepsilon^\alpha$ recasts the equation, for $\varepsilon>0$, as $G(t,\varepsilon)=0$ with
  \[
    G(t,\varepsilon)\defn\frac{t^{1-\alpha}\,\tilde h(t\abs{\varepsilon}^q/2)}{(s^\star)^\alpha}+\abs{\varepsilon}^\alpha\bigl(1-t\abs{\varepsilon}^q\bigr)-1.
  \]
  The absolute values leave the equation unchanged for the margins $\varepsilon>0$ of interest, but render $G$ real-valued and smooth on a \emph{two-sided} neighborhood of $\varepsilon=0$, as the standard implicit function theorem requires. At $\varepsilon_0=0$, $G(t,\varepsilon_0)=t^{1-\alpha}/(s^\star)^\alpha-1$ vanishes at $t_0\defn(s^\star)^{-q}$. On a neighborhood of $(t_0,\varepsilon_0)=((s^\star)^{-q},0)$ the map $G$ is $C^1$: $t^{1-\alpha}$ and $\tilde h$ are smooth, $\varepsilon\mapsto\abs{\varepsilon}^\alpha$ is $C^1$ at $0$ (as $\alpha>1$), and $\varepsilon\mapsto\abs{\varepsilon}^q$ is $C^2$ (as $q>2$). As $\partial_t G\,|_{(t_0,\varepsilon_0)}=(1-\alpha)t_0^{-\alpha}/(s^\star)^\alpha\ne0$, the $C^1$ implicit function theorem~\citep{krantz2013implicit} yields a unique $C^1$ solution $t(\varepsilon)$ near $\varepsilon=\varepsilon_0=0$ with $t(\varepsilon_0)=t_0$ and $G(t(\varepsilon),\varepsilon)=0$; restricting to $\varepsilon>0$ recovers the solutions of the moment constraint. Hence for every sufficiently small $\varepsilon>0$ the mass $m(\varepsilon)=t(\varepsilon)\varepsilon^q/2\in(0,\tfrac12)$ solves~\eqref{eq:two-point-alpha-constraint} (so $\mb P^\star_{\pm\Delta}\in\mc C_{\pm\Delta}$) and
  \[
    2m(\varepsilon)=t(\varepsilon)\,\varepsilon^q=(s^\star)^{-q}\varepsilon^q\bigl(1+o(1)\bigr)\qquad(\varepsilon\downarrow0).
  \]

  \emph{Objective value.} Recall that the mirror pair's affinity is $\rho(\mb P^\star_{-\Delta},\mb P^\star_\Delta)=p_0+2\sqrt{p_+p_-}$. Substituting $p_0=1-2m$ and $\sqrt{p_+p_-}=m\sqrt{1-(s^\star)^2}$ gives $p_0+2\sqrt{p_+p_-}=1-2m+2m\sqrt{1-(s^\star)^2}$, and with $\sqrt{1-(s^\star)^2}=\alpha-1$ (since $(s^\star)^2=\alpha(2-\alpha)=1-(\alpha-1)^2$),
  \[
    p_0+2\sqrt{p_+p_-} \,=\, 1 - 2(2-\alpha)\,m(\varepsilon) \,=\, 1 - \frac{2-\alpha}{(s^\star)^q}\,\varepsilon^q + o(\varepsilon^q).
  \]
  It remains to identify the leading coefficient with $K(\alpha)$. Comparing the definitions of $K(\alpha)$ in~\eqref{eq:K-def} and of $L(\alpha)$ in~\eqref{eq:L-def} yields the rate--margin identity $K(\alpha)=L(\alpha)^{-q}$. From the explicit formula for $C(\alpha)$ (Section~\ref{sec:Bhatt-model}) a direct computation gives $L(\alpha)=s^\star/(2-\alpha)^{(\alpha-1)/\alpha}$, hence
  \[
    K(\alpha) \,=\, L(\alpha)^{-q} \,=\, \frac{(2-\alpha)^{(\alpha-1)q/\alpha}}{(s^\star)^q} \,=\, \frac{2-\alpha}{(s^\star)^q},
  \]
  using $(\alpha-1)q/\alpha=1$. Substituting completes the proof.
\end{proof}

\begin{proof}[Proof of Lemma~\ref{lem:tv-poisson}]
  We use the Poisson approximation theorem of \citet{lecam1960} in its single-variable coupling form: for any $q\in[0,1]$, a $\mathrm{Bernoulli}(q)$ variable $B$ and a $\mathrm{Poisson}(q)$ variable $Z$ can be realized on a common probability space so that
  \begin{equation}\label{eq:lecam-coupling}
    \Pr[B\neq Z]\;=\;\mathrm{TV}\bigl(\mathrm{Bernoulli}(q),\mathrm{Poisson}(q)\bigr)\;=\;q\bigl(1-e^{-q}\bigr)\;\le\;q^2 .
  \end{equation}
  The middle equality is the maximal coupling; the value follows by writing total variation as $\tfrac12\sum_{k\ge0}|p_k-p_k'|$, where $p_k$ and $p_k'$ are the $\mathrm{Bernoulli}(q)$ and $\mathrm{Poisson}(q)$ masses. On the Bernoulli side only $p_0=1-q$ and $p_1=q$ are nonzero, while $p_k'=e^{-q}q^k/k!$, so
  \[
    \begin{aligned}
      \mathrm{TV}\bigl(\mathrm{Bernoulli}(q),\mathrm{Poisson}(q)\bigr)
      &=\tfrac12\Bigl[\bigl|(1-q)-e^{-q}\bigr|+\bigl|q-qe^{-q}\bigr|+\textstyle\sum_{k\ge2}e^{-q}q^k/k!\Bigr]\\
      &=\tfrac12\Bigl[\bigl(e^{-q}-(1-q)\bigr)+q\bigl(1-e^{-q}\bigr)+\bigl(1-(1+q)e^{-q}\bigr)\Bigr]\\
      &=\tfrac12\cdot 2q\bigl(1-e^{-q}\bigr)=q\bigl(1-e^{-q}\bigr)\le q^2,
    \end{aligned}
  \]
  using $e^{-q}\ge 1-q$ at $k=0$ and $\sum_{k\ge2}e^{-q}q^k/k!=1-(1+q)e^{-q}$ at the tail.

  Realize the counts through an iid sample $Y_1,\ldots,Y_n\in\{-1,0,1\}$ with $\Pr[Y_i=\pm1]=p_\pm$ and $\Pr[Y_i=0]=p_0$, and set $U_i\defn\mathbf 1\{Y_i=+1\}\sim\mathrm{Bernoulli}(p_+)$ and $V_i\defn\mathbf 1\{Y_i=-1\}\sim\mathrm{Bernoulli}(p_-)$, so that $n_\pm=\sum_{i=1}^n U_i,\ \sum_{i=1}^n V_i$ and $S_n=\sum_{i=1}^n(U_i-V_i)$.

  Fix a trial $i$. We approximate the law of the indicator pair $(U_i,V_i)$ by that of a pair of \emph{independent} Poissons $\tilde Y_i^+\sim\mathrm{Poisson}(p_+)$, $\tilde Y_i^-\sim\mathrm{Poisson}(p_-)$ in two steps.
  \begin{itemize}
    \item[(a)] \emph{Decouple.} Let $(U_i',V_i')$ be \emph{independent} Bernoullis with the same marginals $p_+,p_-$. The two pairs live on $\{0,1\}^2$, assigning different probability masses according to the table
    \[
      \begin{array}{c|cccc}
        & (0,0) & (1,0) & (0,1) & (1,1)\\[1pt]\hline
        (U_i,V_i) & 1-p_+-p_- & p_+ & p_- & 0 \\[1pt]
        (U_i',V_i') & (1-p_+)(1-p_-) & p_+(1-p_-) & (1-p_+)p_- & p_+p_-
      \end{array}
    \]
    Each of the four entries differs in absolute value by exactly $p_+p_-$, so
    \[
      \mathrm{TV}\bigl(\mathcal L(U_i,V_i),\mathcal L(U_i',V_i')\bigr)=\tfrac12\!\!\sum_{(a,b)\in\{0,1\}^2}\!\!\bigl|\Pr[(U_i,V_i)=(a,b)]-\Pr[(U_i',V_i')=(a,b)]\bigr|=2p_+p_-.
    \]
    \item[(b)] \emph{Poissonize.} Couple $U_i'$ with $\tilde Y_i^+$ by the maximal coupling of~\eqref{eq:lecam-coupling}, so $$\Pr[U_i'\ne\tilde Y_i^+]=\mathrm{TV}(\mathrm{Bernoulli}(p_+),\mathrm{Poisson}(p_+)),$$ and, independently, couple $V_i'$ with $\tilde Y_i^-$ likewise. As the two couplings are independent this is a valid coupling of $(U_i',V_i')$ with $(\tilde Y_i^+,\tilde Y_i^-)$, and the pair disagrees only if a coordinate does, $\{(U_i',V_i')\ne(\tilde Y_i^+,\tilde Y_i^-)\}=\{U_i'\ne\tilde Y_i^+\}\cup\{V_i'\ne\tilde Y_i^-\}$. By the coupling bound for total variation followed by a union bound over these two events,
    \begin{align*}
      \mathrm{TV}\bigl(\mathcal L(U_i',V_i'),\mathcal L(\tilde Y_i^+,\tilde Y_i^-)\bigr) \le & \Pr[(U_i', V_i')\ne(\tilde Y_i^+, \tilde Y_i^-)]\\
      = & \Pr[U_i'\ne\tilde Y_i^+]+\Pr[V_i'\ne\tilde Y_i^-] \\
      \le &\;p_+^2+p_-^2,
    \end{align*}
    the last step applying~\eqref{eq:lecam-coupling} with $q=p_+$ and with $q=p_-$.
  \end{itemize}
  The triangle inequality combines the two steps,
  \begin{equation}\label{eq:per-trial-tv}
    \mathrm{TV}\bigl(\mathcal L(U_i,V_i),\mathcal L(\tilde Y_i^+,\tilde Y_i^-)\bigr)\;\le\;2p_+p_-+p_+^2+p_-^2\;=\;(p_++p_-)^2\;=\;4m^2,
  \end{equation}
  using $p_++p_-=2m$. Carry out these couplings independently across the $n$ trials; then $\tilde n_\pm\defn\sum_{i=1}^n\tilde Y_i^\pm\sim\mathrm{Poisson}(np_\pm)$ are independent, and the disagreement event $\{(n_+,n_-)\ne(\tilde n_+,\tilde n_-)\}$ is in $\bigcup_{i=1}^n\{(U_i,V_i)\ne(\tilde Y_i^+,\tilde Y_i^-)\}$, so a union bound over~\eqref{eq:per-trial-tv} gives $\mathrm{TV}\bigl(\mathcal L(n_+,n_-),\mathcal L(\tilde n_+,\tilde n_-)\bigr)\le 4nm^2$. Finally $S_n=n_+-n_-$ and $\tilde n_+-\tilde n_-\sim\mathrm{Skellam}(np_+,np_-)$, so the data-processing inequality for the map $(a,b)\mapsto a-b$ yields
  \[
    \mathrm{TV}\bigl(\mathcal L(S_n),\mathrm{Skellam}(np_+,np_-)\bigr)\;\le\;4nm^2,
  \]
  which is~\eqref{eq:tv-skellam}.
\end{proof}

\begin{proof}[Proof of Lemma~\ref{lem:BS-properties}]
Treat the parameter $\lambda$ as a time variable, and let $S(\lambda)$ be the continuous-time Markov chain on $\mathbb Z$ started at $S(0)=0$ that jumps $+1$ at rate $1+s^\star$ and $-1$ at rate $1-s^\star$, i.e.\ with $Q$-matrix
  \[
    q_{k,k+1}=1+s^\star,\qquad q_{k,k-1}=1-s^\star,\qquad q_{kk}=-2,\qquad q_{kj}=0\ \ (|j-k|\ge2).
  \]
  The $+1$ and $-1$ jumps occur along two independent Poisson clocks of rates $1+s^\star$ and $1-s^\star$, so by time $\lambda$ the numbers of up- and down-jumps are independent $\mathrm{Poisson}(\lambda(1+s^\star))$ and $\mathrm{Poisson}(\lambda(1-s^\star))$ variables; their difference $S(\lambda)$ therefore has law $\mathrm{Skellam}(\lambda(1+s^\star),\lambda(1-s^\star))$, exactly the law appearing in $\mc B_S(\lambda)$. The rates are bounded, so writing $p_k(\lambda)\defn\Pr[S(\lambda)=k]$, the forward equation $P'(\lambda)=P(\lambda)Q$~\citep[Thm.~2.8.2(c)]{norris1997markov} reads $p_k'(\lambda)=\sum_{j\in\mathbb Z}p_j(\lambda)q_{jk}$. Hence, for any bounded $f:\mathbb Z\to\Re$, the double sum being absolutely convergent (as $f$ is bounded and $\sum_{k\in\mathbb Z}|q_{jk}|=4$),
  \[
    \frac{\d}{\d\lambda}\mb E\bigl[f(S(\lambda))\bigr]=\sum_{k\in\mathbb Z} f(k)\,p_k'(\lambda)=\sum_{j\in\mathbb Z} p_j(\lambda)\sum_{k\in\mathbb Z} q_{jk}f(k)
  \]
  where, using $q_{jj}=-2=-(1+s^\star)-(1-s^\star)$,
  \[
    \sum_{k\in\mathbb Z} q_{jk}f(k)=(1+s^\star)\bigl(f(j+1)-f(j)\bigr)+(1-s^\star)\bigl(f(j-1)-f(j)\bigr).
  \]
  Thus
  \[
    \frac{\d}{\d\lambda}\mb E\bigl[f(S(\lambda))\bigr]=(1+s^\star)\,\mb E\bigl[f(S+1)-f(S)\bigr]+(1-s^\star)\,\mb E\bigl[f(S-1)-f(S)\bigr].
  \]
  Consider now in particular $f(z)=\mathbf 1\{z<0\}+\tfrac12\mathbf 1\{z=0\}$, so that $\mb E[f(S(\lambda))]=\mc B_S(\lambda)$. Here $f(z+1)-f(z)=-\tfrac12$ for $z\in\{-1,0\}$ and $0$ otherwise, while $f(z-1)-f(z)=+\tfrac12$ for $z\in\{0,1\}$ and $0$ otherwise, so
  \[
    \mc B_S'(\lambda)=-\frac{1+s^\star}{2}\bigl(\mb P_S[S=-1]+\mb P_S[S=0]\bigr)+\frac{1-s^\star}{2}\bigl(\mb P_S[S=0]+\mb P_S[S=1]\bigr).
  \]
  From the Skellam mass $p_S(k;\lambda)=e^{-2\lambda}\bigl(\tfrac{(1+s^\star)}{(1-s^\star)}\bigr)^{k/2}I_{|k|}\bigl(2\lambda\rho_s\bigr)$ (and $I_{|1|}=I_{|-1|}$) we have $\mb P_S[S=1]=\tfrac{(1+s^\star)}{(1-s^\star)}\,\mb P_S[S=-1]$, hence $(1-s^\star)\mb P_S[S=1]=(1+s^\star)\mb P_S[S=-1]$ and the $S=\pm1$ contributions cancel, leaving
  \[
    \mc B_S'(\lambda)=-s^\star\,\mb P_S[S=0]=-s^\star e^{-2\lambda}I_0\bigl(2\lambda\rho_s\bigr),\qquad\lambda>0 .
  \]
  At $\lambda=0$ the law degenerates, $S\sim\mathrm{Skellam}(0,0)=\delta_0$, so $\mb P_S[S<0]=0$, $\mb P_S[S=0]=1$, and $\mc B_S(0)=\tfrac12$. The standard Laplace transform $\int_0^\infty e^{-pu}I_0(au)\,\d u=(p^2-a^2)^{-1/2}$ ($p>a\ge0$)~\citep[Equation~6.611.4]{gradshteyn2007table}, with $p=2$, $a=2\rho_s$ and $p^2-a^2=4(s^\star)^2$, gives $s^\star\!\int_0^\infty e^{-2u}I_0(2u\rho_s)\,\d u=\tfrac{s^\star}{(2s^\star)}=\tfrac12=\mc B_S(0)$. Integrating the derivative from $0$ to $\lambda$,
  \[
    \mc B_S(\lambda)=\mc B_S(0)+\int_0^\lambda\mc B_S'(u)\,\d u=s^\star\!\int_0^\infty e^{-2u}I_0(2u\rho_s)\,\d u-s^\star\!\int_0^\lambda e^{-2u}I_0(2u\rho_s)\,\d u=s^\star\!\int_\lambda^\infty e^{-2u}I_0(2u\rho_s)\,\d u,
  \]
  which is~\eqref{eq:BS-integral}. The integrand $e^{-2u}I_0(2u\rho_s)$ is positive, so $\mc B_S$ is $C^1$ and strictly decreasing on $[0,\infty)$, with $\mc B_S(\infty)=0$; with $\mc B_S(0)=\tfrac12$ it is a continuous strictly decreasing bijection of $[0,\infty)$ onto $(0,1/2]$, and the inverse $\lambda^\star(\beta)=\mc B_S^{-1}(\beta)$ is well-defined and continuous on $(0,1/2)$.
\end{proof}

\begin{proof}[Proof of Proposition~\ref{prop:skellam-lb-fixed-beta}]
  Suppose, towards contradiction, that along a subsequence $n=n_k\to\infty$,
  \begin{equation}\label{eq:contra-i}
    n_k^{(\alpha-1)/\alpha}\Delta_{n_k}/\sigma \;\leq\; (1-\delta)\,s^\star\bigl(2\lambda^\star(\beta)\bigr)^{(\alpha-1)/\alpha}
  \end{equation}
  for some $\delta>0$. Writing $\varepsilon_k\defn\Delta_{n_k}/\sigma$, the hypothesis~\eqref{eq:contra-i} reads $\varepsilon_k\le(1-\delta)\,s^\star\bigl(2\lambda^\star(\beta)\bigr)^{(\alpha-1)/\alpha}\,n_k^{-(\alpha-1)/\alpha}$, so $\varepsilon_k\to0$ as $k\to\infty$ (since $(\alpha-1)/\alpha>0$) and the small-margin scaling~\eqref{eq:m-scaling} applies. The value $\lambda_k\defn n_k\,m(\varepsilon_k)$ satisfies
  \[
    \lambda_k \;=\; \tfrac12(s^\star)^{-q}n_k\varepsilon_k^q(1+o(1)) \;\leq\; \tfrac12(s^\star)^{-q}\bigl[(1-\delta)s^\star\bigr]^q\bigl(2\lambda^\star(\beta)\bigr)^{q(\alpha-1)/\alpha}\bigl(1+o(1)\bigr).
  \]
  Using $q(\alpha-1)/\alpha=1$ this simplifies to $\lambda_{k}\le(1-\delta)^q\lambda^\star(\beta)(1+o(1))$, so $\limsup_{k\to\infty}\lambda_k\le(1-\delta)^q\lambda^\star(\beta)<\lambda^\star(\beta)$. Passing to a further subsequence, $\lambda_k\to\lambda_\infty\in[0,(1-\delta)^q\lambda^\star(\beta)]$.

  We now compare the trinomial Bayes error $\mc B_{n_k}(\Delta_{n_k})=\mb P^\star_{\Delta_{n_k}}[S_{n_k}<0]+\tfrac12\mb P^\star_{\Delta_{n_k}}[S_{n_k}=0]$ (defined in~\eqref{eq:skellam-lb-nonasymp}, with $S_n=\sum_{i=1}^n X_i/a(\Delta_n)\in\{-n,\dots,n\}$) to its Skellam counterpart $\mc B_S(\lambda)=\mb P_S[S<0]+\tfrac12\mb P_S[S=0]$ for $S\sim\mathrm{Skellam}(\lambda(1+s^\star),\lambda(1-s^\star))$ (Lemma~\ref{lem:BS-properties}). Both are the expectation of the same bounded functional $f(z)\defn\mathbf 1\{z<0\}+\tfrac12\mathbf 1\{z=0\}\in[0,1]$: writing $\mc S_k\defn\mathrm{Skellam}(\lambda_k(1+s^\star),\lambda_k(1-s^\star))$ for the Skellam law (its parameter matched to the binomial via $\lambda_k=n_km(\varepsilon_k)$), we have $\mc B_{n_k}(\Delta_{n_k})=\E{\mc L(S_{n_k})}{f}$ and $\mc B_S(\lambda_k)=\E{\mc S_k}{f}$. Lemma~\ref{lem:tv-poisson} bounds the total variation between these two laws by $4n_km(\varepsilon_k)^2=4\lambda_k^2/n_k$, which vanishes since $\lambda_k$ is bounded while $n_k\to\infty$. As $f\in[0,1]$, a difference of expectations is at most the total variation between the underlying laws, so
  \[
    \bigl|\mc B_{n_k}(\Delta_{n_k}) - \mc B_S(\lambda_k)\bigr|
    \;=\; \bigl|\E{\mc L(S_{n_k})}{f} - \E{\mc S_k}{f}\bigr|
    \;\leq\; \mathrm{TV}\bigl(\mc L(S_{n_k}),\mc S_k\bigr)
    \;\leq\; \frac{4\lambda_k^2}{n_k}
    \;\to\; 0.
  \]
  Continuity of $\mc B_S$ at $\lambda_\infty$ (Lemma~\ref{lem:BS-properties}) gives $\mc B_S(\lambda_k)\to\mc B_S(\lambda_\infty)$, and combining the two yields $\mc B_{n_k}(\Delta_{n_k})\to\mc B_S(\lambda_\infty)$. By strict monotonicity, $\mc B_S(\lambda_\infty)>\mc B_S(\lambda^\star(\beta))=\beta$. But Proposition~\ref{prop:skellam-lb-nonasymp} together with the hypothesis $\beta_{n_k}(\mu_{n_k};\Delta_{n_k})\le\beta$ gives $\beta\ge\beta_{n_k}(\mu_{n_k};\Delta_{n_k})\ge\mc B_{n_k}(\Delta_{n_k})\to\mc B_S(\lambda_\infty)>\beta$; a contradiction. 
\end{proof}

\begin{proof}[Proof of Corollary~\ref{cor:skellam-lb-asymp}]
  Recall $\mc B_S(\lambda)=s^\star\int_\lambda^\infty e^{-2u}I_0(2u\rho_s)\,\d u$ with $\rho_s=\sqrt{1-(s^\star)^2}$ and $I_0$ the modified Bessel function of the first kind, and abbreviate $a\defn2(1-\rho_s)>0$. The integral representation $I_0(z)=\tfrac1\pi\int_0^\pi e^{z\cos\theta}\,\d\theta$ gives elementary two-sided bounds. From $\cos\theta\le1$, $I_0(z)\le e^z$, whence
  \[
    \mc B_S(\lambda)\;\le\;s^\star\!\int_\lambda^\infty e^{-au}\,\d u\;=\;C\,e^{-a\lambda},\qquad C\defn\frac{s^\star}{2(1-\rho_s)} .
  \]
  For a matching lower bound, restrict the representation to $\theta\in[0,z^{-1/2}]$ where $\cos\theta\ge1-\theta^2/2\ge1-\tfrac{1}{(2z)}$, so
  \[
    I_0(z)\;\ge\;\tfrac1\pi\!\int_0^{z^{-1/2}}\!e^{z\cos\theta}\,\d\theta\;\ge\;\tfrac1\pi\,z^{-1/2}e^{z-1/2}\;=\;\frac{e^{-1/2}}{\pi}\,\frac{e^z}{\sqrt z}\qquad(z\ge1/\pi^2).
  \]
  Taking $z=2u\rho_s\ge1/\pi^2$ (i.e.\ $u\ge\tfrac{1}{(2\pi^2\rho_s)}$) gives $e^{-2u}I_0(2u\rho_s)\ge\kappa\,u^{-1/2}e^{-au}$ with $\kappa\defn\tfrac{e^{-1/2}}{(\pi\sqrt{2\rho_s})}$. Restricting the integral to $[\lambda,2\lambda]$ and using $u^{-1/2}\ge(2\lambda)^{-1/2}$ there, with
  \[
    \int_\lambda^{2\lambda}e^{-au}\,\d u=\frac{e^{-a\lambda}-e^{-2a\lambda}}{a}=\frac{e^{-a\lambda}}{a}\bigl(1-e^{-a\lambda}\bigr)\ge\frac{e^{-a\lambda}}{2a}\qquad(a\lambda\ge\log 2,\ \text{so } e^{-a\lambda}\le\tfrac12),
  \]
  we obtain
  \[
    \mc B_S(\lambda)\;\ge\;s^\star\kappa\,(2\lambda)^{-1/2}\,\frac{e^{-a\lambda}}{2a}\;=\;c\,\lambda^{-1/2}e^{-a\lambda},\qquad c\defn\frac{s^\star\kappa}{2\sqrt2\,a},
  \]
  for all $\lambda\ge\max\{\tfrac{1}{(2\pi^2\rho_s)},\tfrac{\log 2}{a}\}$. We invert these bounds directly. By Lemma~\ref{lem:BS-properties} the inverse $\lambda^\star(\beta)=\mc B_S^{-1}(\beta)\to\infty$ as $\beta\downarrow0$ (as $\mc B_S(\infty)=0$); in particular $\lambda^\star(\beta)\ge\max\{\tfrac{1}{(2\pi^2\rho_s)},\tfrac{\log 2}{a}\}$ for $\beta$ small enough, so the lower bound is available at $\lambda=\lambda^\star(\beta)$. Evaluating the sandwich there, where $\mc B_S(\lambda^\star(\beta))=\beta$, and taking logarithms in $c\,\lambda^\star(\beta)^{-1/2}e^{-a\lambda^\star(\beta)}\le\beta\le C\,e^{-a\lambda^\star(\beta)}$,
  \[
    \log(1/\beta)+\log c-\tfrac12\log\lambda^\star(\beta)\;\le\;a\lambda^\star(\beta)\;\le\;\log(1/\beta)+\log C .
  \]
  The right inequality already gives $\lambda^\star(\beta)\le a^{-1}\bigl(\log(1/\beta)+\log C\bigr)=O(\log(1/\beta))$, hence $\log\lambda^\star(\beta)=O(\log\log(1/\beta))=o(\log(1/\beta))$; the two-sided bound then collapses to $a\lambda^\star(\beta)=\log(1/\beta)(1+o(1))$. With $a=2(1-\rho_s)=2(2-\alpha)$ from~\eqref{eq:rho-s},
  \[
    2\lambda^\star(\beta) \;=\; \frac{\log(1/\beta)}{2-\alpha}\,(1+o(1)),
  \]
  and substituting,
  \[
    s^\star\bigl(2\lambda^\star(\beta)\bigr)^{(\alpha-1)/\alpha} \;=\; \frac{s^\star}{(2-\alpha)^{(\alpha-1)/\alpha}}\bigl(\log(1/\beta)\bigr)^{(\alpha-1)/\alpha}(1+o(1)) \;=\; L(\alpha)\bigl(\log(1/\beta)\bigr)^{(\alpha-1)/\alpha}(1+o(1)),
  \]
  using $L(\alpha)=s^\star/(2-\alpha)^{(\alpha-1)/\alpha}$, established in the proof of Lemma~\ref{lem:F-star-asymp}, which by its definition~\eqref{eq:L-def} equals the upper-bound constant $\alpha\,C(\alpha)^{1/\alpha}/(\alpha-1)^{(\alpha-1)/\alpha}$.
\end{proof}

\subsubsection{Proofs of Section~\ref{sec:slow-growth}}

\begin{proof}[Proof of Corollary~\ref{cor:slow-achievable}]
By Theorem~\ref{thm:feasibility-guarantee} this estimator obeys $\beta_n(\mu_n;\Delta)\le e^{-n r_M(\Delta)}$, and $r_M(\Delta)=r(\Delta)=m(\Delta)(1+o(1))$ as $\Delta\downarrow0$ (Theorems~\ref{thm:matching} and~\ref{thm:slow-sharp}). For $\Delta$ small enough that $B/\Delta\ge\Lambda(x_\infty)$, so that $\Lambda^{-1}(B/\Delta)$ is defined, the map $\Delta\mapsto m(\Delta)=\Delta/\Lambda^{-1}(B/\Delta)$ is continuous and strictly increasing with $m(\Delta)\downarrow0$, and its inverse is
\begin{equation}\label{eq:slow-m-inverse}
  m^{-1}(\rho)\;=\;\rho\,\phi^{-1}(B/\rho),
\end{equation}
since, for $\rho$ small enough that $B/\rho\ge\phi(x_\infty)$ where $\phi$ restricts to a strictly increasing bijection of $[x_\infty,\infty)$ onto $[\phi(x_\infty),\infty)$, the value $a\defn\phi^{-1}(B/\rho)\ge x_\infty$ gives $\Lambda(a)=\phi(a)/a=B/(\rho a)$, so $\Delta\defn\rho a$ has $a(\Delta)=a$ and $m(\Delta)=\Delta/a=\rho$. Because $\phi(x)=\abs x\,\Lambda(\abs x)$ is regularly varying of index $1$ at infinity, so is its inverse $\phi^{-1}$ \citep[Proposition~2.6(v)]{resnick2007heavy}; hence $\rho\mapsto\phi^{-1}(B/\rho)$ is regularly varying of index $-1$ as $\rho\downarrow0$, and $m^{-1}(\rho)=\rho\,\phi^{-1}(B/\rho)$ is \emph{slowly varying} there.

Fix any $\zeta>0$ and put $\Delta_n\defn m^{-1}\bigl((1+\zeta)\xi_n\bigr)$, so $m(\Delta_n)=(1+\zeta)\xi_n$ and $\Delta_n\downarrow0$. Then $r_M(\Delta_n)=(1+\zeta)\xi_n(1+o(1))\ge\xi_n$ for all large $n$, whence $\beta_n(\mu_n;\Delta_n)\le e^{-n\xi_n}=\beta$. Finally, since $m^{-1}$ is slowly varying and $1+\zeta$ is a fixed constant,
\[
  \Delta_n=m^{-1}\bigl((1+\zeta)\xi_n\bigr)=m^{-1}(\xi_n)\,(1+o(1)),
\]
the asserted margin, with closed form $\xi_n\,\phi^{-1}(B/\xi_n)$ by~\eqref{eq:slow-m-inverse}.
\end{proof}

\begin{proof}[Proof of Proposition~\ref{prop:slow-fixed-beta}]
Fix $\Delta'\in(\Delta_n,\Delta_0]$ (nonempty since $\Delta_n<\Delta_0$) and consider $\mb P^\star_{\pm\Delta'}\defn(1-m(\Delta'))\delta_0+m(\Delta')\delta_{\pm a(\Delta')}\in\mc C_{\pm\Delta'}$, the pair~\eqref{eq:slow-pair} at margin $\Delta'$. Specializing the non-asymptotic inequality~\eqref{eq:LB-reduction} to $(\mu,\mb P)=(\pm\Delta',\mb P^\star_{\pm\Delta'})$ and applying the Bayes-risk bound at uniform prior gives $\beta\ge\mc R^\star_n(\mb P^\star_{-\Delta'},\mb P^\star_{\Delta'})$. Under the change of variable $Y_i\defn X_i/a(\Delta')\in\{-1,0,+1\}$ the two hypotheses are $Y_i\sim((1-m')\,\text{at }0,\ m'\,\text{at }+1)$ and its mirror, with $m'=m(\Delta')$. As the rare atoms are disjoint, the single-observation log-likelihood ratio is $+\infty$ on $\{Y_i=+1\}$, $-\infty$ on $\{Y_i=-1\}$, and $0$ on $\{Y_i=0\}$: the Bayes-optimal test errs only on the event that all $n$ observations equal $0$, which has probability $(1-m')^n$ under either hypothesis, and then errs with probability $\tfrac12$. Hence $\mc R^\star_n=\tfrac12(1-m(\Delta'))^n$, and letting $\Delta'\downarrow\Delta_n$ with continuity of $\Delta\mapsto m(\Delta)$ gives the claim.
\end{proof}

\begin{proof}[Proof of Corollary~\ref{cor:slow-fixed-beta}]
Suppose, towards contradiction, that $\liminf_{n\to\infty} \Delta_n/m^{-1}(\xi_n)<1$; then there are $\delta\in(0,1)$ and a subsequence $n_k\to\infty$ with $\Delta_{n_k}\le(1-\delta)\,m^{-1}(\xi_{n_k})$. Since $\xi_{n_k}\to0$ and $m^{-1}(\rho)\downarrow0$ as $\rho\downarrow0$, this forces $\Delta_{n_k}\to0$; in particular $\Delta_{n_k}<\Delta_0$ for $k$ large, so Proposition~\ref{prop:slow-fixed-beta} gives $\tfrac12(1-m(\Delta_{n_k}))^{n_k}\le\beta$, i.e.\ $n_k\log(1/(1-m(\Delta_{n_k})))\ge\log(1/(2\beta))$. As $\Delta_{n_k}\to0$ forces $m(\Delta_{n_k})\to0$, we have $\log(1/(1-m(\Delta_{n_k})))=-\log(1-m(\Delta_{n_k}))=m(\Delta_{n_k})(1+o(1))$, so  $n_k\,m(\Delta_{n_k})(1+o(1))\ge\log(1/(2\beta))$, hence $m(\Delta_{n_k})\ge\tfrac{\log(1/(2\beta))}{n_k}(1-o(1))=c_\beta\,\xi_{n_k}(1-o(1))$, where $c_\beta\defn 1-\log2/\log(1/\beta)\in(0,1)$ is a fixed constant; fixing any $\kappa\in(0,c_\beta)$ gives $m(\Delta_{n_k})\ge\kappa\,\xi_{n_k}$ for $k$ large. Since $m$ is increasing and $m^{-1}$ is slowly varying (see proof of Corollary~\ref{cor:slow-achievable}, via \citealp[Proposition~2.6(v)]{resnick2007heavy}) with $\kappa$ fixed, $\Delta_{n_k}\ge m^{-1}(\kappa\,\xi_{n_k})=m^{-1}(\xi_{n_k})(1+o(1))$, contradicting $\Delta_{n_k}\le(1-\delta)\,m^{-1}(\xi_{n_k})$. Hence $\liminf_{n\to\infty} \Delta_n/m^{-1}(\xi_n)\ge1$.
\end{proof}

\section{Right-Continuity of $r$}
\label{app:right-continuity}

The following three lemmas support the right-continuity step $\inf_{\Delta'>\Delta}r(\Delta')\le r(\Delta)$ in the proof of Proposition~\ref{prop:LB}. Throughout, $r(\Delta)$ is the lower-bound problem~\eqref{eq:intro-hellinger} and $\phi$ is the even, continuous, super-linear function of~\eqref{eq:moment-class-intro}, non-decreasing on $[0,\infty)$ with $\phi(0)<B$.

\begin{lemma}[Finiteness]\label{lem:r-finite}
For every $\Delta>0$, $r(\Delta)<\infty$.
\end{lemma}
\begin{proof}
Fix $p\in(0,1)$ and set $\mb P_{\pm\Delta}=p\,\delta_0+(1-p)\,\delta_{\pm\Delta/(1-p)}$. Then $\E{\mb P_{\pm\Delta}}{X}=\pm\Delta$, and as the two measures share only the atom at $0$ (of mass $p$) their affinity is $\rho(\mb P_{-\Delta},\mb P_\Delta)=\sqrt{p\cdot p}=p$. By continuity of $\phi$,
\[
\int\phi(x-\Delta)\,\d\mb P_\Delta(x)=p\,\phi(\Delta)+(1-p)\,\phi \Bigl(\tfrac{p\Delta}{(1-p)}\Bigr)\xrightarrow[p\downarrow0]{}\phi(0)<B,
\]
and symmetrically for $\mb P_{-\Delta}$. Hence for $p$ small enough the pair is feasible in~\eqref{eq:intro-hellinger} at level $\Delta$, giving $r(\Delta)=-\log\sup\rho\le-\log p=\log(1/p)<\infty$.
\end{proof}

\begin{lemma}[Compact strictly-feasible near-maximisers]\label{lem:compact-strict}
For every $\Delta>0$ and $\eta>0$ there exist compactly supported $\mb P_{-\Delta},\mb P_\Delta$ with $\mb P_{\pm\Delta}\in\mc C_{\pm\Delta}$, strict moment constraints $\int\phi(x\mp\Delta)\,\d\mb P_{\pm\Delta}<B$ and
\[
\rho(\mb P_{-\Delta},\mb P_\Delta)\ge e^{-r(\Delta)}-\eta.
\]
\end{lemma}
\begin{proof}
By Lemma~\ref{lem:r-finite} $r(\Delta)<\infty$, so by definition of the supremum in~\eqref{eq:intro-hellinger} there is a feasible pair $(\widetilde{\mb P}_{-\Delta},\widetilde{\mb P}_\Delta)$ at level $\Delta$ with $\rho(\widetilde{\mb P}_{-\Delta},\widetilde{\mb P}_\Delta)\ge e^{-r(\Delta)}-\eta/4$. We transform it in three steps into the desired $(\mb P_{-\Delta},\mb P_\Delta)$, handling the two components $\widetilde{\mb P}_{-\Delta}$ and $\widetilde{\mb P}_\Delta$ in parallel; throughout, an expression with $\pm$ and the matching $\mp$ is to be read for both signs.

\emph{Strict moments.} Replace $\widetilde{\mb P}_{\pm\Delta}$ by $(1-a)\widetilde{\mb P}_{\pm\Delta}+a\,\delta_{\pm\Delta}$. The mean is unchanged, and since $\phi(0)<B$,
\[
\int\phi(x\mp\Delta)\,\d\bigl[(1-a)\widetilde{\mb P}_{\pm\Delta}+a\,\delta_{\pm\Delta}\bigr]=(1-a)\!\int\phi(x\mp\Delta)\,\d\widetilde{\mb P}_{\pm\Delta}+a\,\phi(0)\le(1-a)B+a\,\phi(0)<B,
\]
so both moment constraints become strict. The mixed pair has affinity at least $\rho((1-a)\widetilde{\mb P}_{-\Delta}+a\,\delta_{-\Delta}, (1-a)\widetilde{\mb P}_{\Delta}+a\,\delta_{\Delta}) \geq (1-a)\,\rho(\widetilde{\mb P}_{-\Delta},\widetilde{\mb P}_\Delta)$, so for $a$ small it stays at least $e^{-r(\Delta)}-\eta/2$.

\emph{Truncation.} For $R>0$ let $A=\{|x|\le R\}$, so that $\widetilde{\mb P}_{\pm\Delta}(A)\to1$ as $R\to\infty$, and condition both measures on $A$: $\widetilde{\mb P}_{\pm\Delta,A}=\widetilde{\mb P}_{\pm\Delta}(\cdot\mid A)$, each supported in $[-R,R]$. Then
\[
\rho(\widetilde{\mb P}_{-\Delta,A},\widetilde{\mb P}_{\Delta,A})=\frac{\int_A\sqrt{\d\widetilde{\mb P}_{-\Delta}\,\d\widetilde{\mb P}_\Delta}}{\sqrt{\widetilde{\mb P}_{-\Delta}(A)\,\widetilde{\mb P}_\Delta(A)}}\;\ge\;\int_A\sqrt{\d\widetilde{\mb P}_{-\Delta}\,\d\widetilde{\mb P}_\Delta}\xrightarrow[R\to\infty]{}\rho(\widetilde{\mb P}_{-\Delta},\widetilde{\mb P}_\Delta).
\]
The conditioned means and moments converge to their original values as $R\to\infty$.

\emph{Exact mean.} Conditioning has perturbed the mean: by the truncation step the conditioned mean $m_\pm:=\E{\widetilde{\mb P}_{\pm\Delta,A}}{X}$ converges to $\pm\Delta$ as $R\to\infty$, so the discrepancy $e_\pm:=m_\pm\mp\Delta$ tends to $0$. If $e_\pm=0$ set $\mb P_{\pm\Delta}=\widetilde{\mb P}_{\pm\Delta,A}$. Otherwise restore the exact mean by adding one small atom: put $\theta_\pm=|e_\pm|^{1/2}\in(0,1)$ and
\[
y_\pm=\frac{\pm\Delta-(1-\theta_\pm)m_\pm}{\theta_\pm},\qquad
\mb P_{\pm\Delta}=(1-\theta_\pm)\widetilde{\mb P}_{\pm\Delta,A}+\theta_\pm\,\delta_{y_\pm},
\]
where $y_\pm$ is chosen exactly so that $(1-\theta_\pm)m_\pm+\theta_\pm y_\pm=\pm\Delta$, i.e.\ $\E{\mb P_{\pm\Delta}}{X}=\pm\Delta$. As $R\to\infty$ this corrective atom is negligible in both mass and location:
\[
\theta_\pm=|e_\pm|^{1/2}\to0,\qquad
\bigl|y_\pm-(\pm\Delta)\bigr|=(1-\theta_\pm)\frac{|e_\pm|}{\theta_\pm}=(1-\theta_\pm)\,|e_\pm|^{1/2}\to0 ,
\]
so $y_\pm\to\pm\Delta$. Three consequences for $R$ large. First, since $\widetilde{\mb P}_{\pm\Delta,A}$ is supported in $[-R,R]$ and $y_\pm\to\pm\Delta$, the measure $\mb P_{\pm\Delta}$ is supported in $[-R,R]\cup\{y_\pm\}$, hence compactly supported. Second, by continuity of $\phi$,
\[
\int\phi(x\mp\Delta)\,\d\mb P_{\pm\Delta}=(1-\theta_\pm)\!\int\phi(x\mp\Delta)\,\d\widetilde{\mb P}_{\pm\Delta,A}+\theta_\pm\,\phi\bigl(y_\pm\mp\Delta\bigr)\xrightarrow[R\to\infty]{}\int\phi(x\mp\Delta)\,\d\widetilde{\mb P}_{\pm\Delta}<B,
\]
the limit being strict from the strict-moments step (the atom term vanishes as $\theta_\pm\to0$ and $\phi(y_\pm\mp\Delta)\to\phi(0)$), so the moment constraint stays strict. Third,
\[
\rho(\mb P_{-\Delta},\mb P_\Delta)\ge\rho\bigl((1-\theta_-)\widetilde{\mb P}_{-\Delta,A},(1-\theta_+)\widetilde{\mb P}_{\Delta,A}\bigr)=\sqrt{(1-\theta_-)(1-\theta_+)}\,\rho(\widetilde{\mb P}_{-\Delta,A},\widetilde{\mb P}_{\Delta,A}),
\]
with $\theta_\pm\to0$. Combining the three steps (taking $a$ small, then $R$ large, then the corrective atoms negligible), the pair $(\mb P_{-\Delta},\mb P_\Delta)$ is compactly supported, has strict moments, and satisfies $\rho(\mb P_{-\Delta},\mb P_\Delta)\ge e^{-r(\Delta)}-\eta$.
\end{proof}

\begin{lemma}[Scaling]\label{lem:scaling}
Let $S_t(x)=tx$ for $t>0$. The common pushforward $(\mb P_{-\Delta},\mb P_\Delta)\mapsto\bigl((S_t)_*\mb P_{-\Delta},(S_t)_*\mb P_\Delta\bigr)$ leaves the affinity $\rho(\mb P_{-\Delta},\mb P_\Delta)$ unchanged. If the pair is compactly supported, feasible in~\eqref{eq:intro-hellinger} at level $\Delta$, and has strict moment constraints, then its common pushforward is feasible at level $t\Delta$ for all $t$ sufficiently close to $1$.
\end{lemma}
\begin{proof}
The affinity is invariant under the common bijection $S_t$ (change of variables). The means scale by $t$, so $(S_t)_*\mb P_{\pm\Delta}$ has mean $\pm t\Delta$. For the moment constraint,
\[
\int\phi(x-t\Delta)\,\d\bigl((S_t)_*\mb P_\Delta\bigr)(x)=\int\phi\bigl(t(y-\Delta)\bigr)\,\d\mb P_\Delta(y),
\]
and symmetrically for $\mb P_{-\Delta}$. Write $g(t):=\int\phi\bigl(t(y-\Delta)\bigr)\,\d\mb P_\Delta(y)$ for the right-hand side, and let $[-M,M]$ contain the support of $\mb P_\Delta$. For $t$ in the neighborhood $[\tfrac12,\tfrac32]$ of $1$ and $y\in[-M,M]$ the argument $t(y-\Delta)$ lies in the compact interval $K:=[-\tfrac32(M+\Delta),\tfrac32(M+\Delta)]$, on which the continuous function $\phi$ attains a finite bound $C:=\sup_K\phi$. The integrand $y\mapsto\phi\bigl(t(y-\Delta)\bigr)$ is therefore dominated by the constant $C$, which is $\mb P_\Delta$-integrable, while for each fixed $y$ it converges to $\phi(y-\Delta)$ as $t\to1$ by continuity of $\phi$. Dominated convergence gives
\[
g(t)\xrightarrow[t\to1]{}g(1)=\int\phi(y-\Delta)\,\d\mb P_\Delta(y),
\]
so $g$ is continuous at $t=1$. Since the pair is strictly feasible at level $\Delta$, $g(1)<B$, and hence $g(t)<B$ for all $t$ sufficiently close to $1$; the same applies to $\mb P_{-\Delta}$.
\end{proof}

\section{Slow-Growth Lower Bound}
\label{app:slow-growth}

This appendix proves Lemma~\ref{prop:slow-hellinger}. Throughout, Assumption~\ref{assu:slow-growth} is in force, and we use the shifted class $\mc C_\mu$ of~\eqref{eq:shifted-class}, the lower-bound problem $r(\Delta)$ of~\eqref{eq:intro-hellinger}, the scale and mass $a(\Delta),m(\Delta)$ of~\eqref{eq:slow-critical-explicit}, and the mirror affinity $\rho(\cdot,\cdot)$ of~\eqref{eq:rho-def}. For $z\ge\Lambda(x_\infty)$ write
\begin{equation}\label{eq:slow-def-a-F}
  A(z):=\Lambda^{-1}(z),\qquad F(z):=\frac{1}{A(z)},
\end{equation}
so that $a(\Delta)=A(B/\Delta)=\Lambda^{-1}(B/\Delta)$ and $m(\Delta)=\Delta/a(\Delta)$.

\begin{lemma}[Superpolynomial growth of $A$]\label{lem:slow-inverse-growth}
Under Assumption~\ref{assu:slow-growth}(A2), for every $q>0$,
\[
  \frac{A(z)}{z^q}\longrightarrow\infty,\qquad z\to\infty;
\]
in particular $z/A(z)\to0$.
\end{lemma}
\begin{proof}
This is a standard property of inverses of slowly varying functions, and follows directly from \cite[Proposition~2.6(i,v)]{resnick2007heavy} (equivalently, $A$ is rapidly varying of index $+\infty$, being the inverse of a slowly varying function).
\end{proof}

\begin{lemma}[Eventual convexity of the reciprocal inverse]\label{lem:slow-F-convex}
Under Assumption~\ref{assu:slow-growth}, the function $F(z)=1/\Lambda^{-1}(z)$ is eventually convex.
\end{lemma}
\begin{proof}
For $x\ge x_\infty$, differentiating the von Mises representation~(A2) gives $\Lambda'(x)=\eta(x)\Lambda(x)/x>0$, hence $\phi'(x)=\Lambda(x)(1+\eta(x))$ and
\[
  \phi''(x)=2\Lambda'(x)+x\Lambda''(x)=\frac{\Lambda(x)\,\eta(x)}{x}\Bigl[1+\eta(x)+\frac{x\eta'(x)}{\eta(x)}\Bigr] .
\]
By~(A2)(i),(iii), $\eta(x)\to0$ and $x\eta'(x)/\eta(x)\to0$, so the bracket tends to $1$; choose $R\ge x_\infty$ so large that $\eta$ is positive and continuously differentiable on $[R,\infty)$ and the bracket is positive there. Then $\Lambda$ is $C^2$ on $[R,\infty)$ (as $\eta$ is $C^1$) and the tail $\phi(x)=x\Lambda(x)$ is convex there, $\phi''(x)\ge0$ for $x\ge R$.

Write $x=A(z)=\Lambda^{-1}(z)$, well-defined for $z\ge\Lambda(x_\infty)$ since by~(A2) the map $\Lambda$ is a strictly increasing bijection of $[x_\infty,\infty)$ onto $[\Lambda(x_\infty),\infty)$, so that $z=\Lambda(x)$ and $F(z)=1/x$. By the inverse-function rule $\frac{\d x}{\d z}=A'(z)=\frac{1}{\Lambda'(x)}$ and the chain rule we get
\[
  F'(z)=-\frac{1}{x^2}\frac{\d x}{\d z}=-\frac{1}{x^2\Lambda'(x)} .
\]
Writing $F'(z)=-[g(x)]^{-1}$ with $g(x):=x^2\Lambda'(x)$, and applying the same process,
\[
  F''(z)=\frac{1}{\Lambda'(x)}\frac{\d}{\d x}\!\left(-\frac{1}{g(x)}\right)=\frac{1}{\Lambda'(x)}\frac{g'(x)}{g(x)^2},\qquad g'(x)=x\bigl(2\Lambda'(x)+x\Lambda''(x)\bigr),
\]
so that $F''(z)=\bigl(2\Lambda'(x)+x\Lambda''(x)\bigr)/\bigl(x^3(\Lambda'(x))^3\bigr) = \phi''(x)/\bigl(x^3(\Lambda'(x))^3\bigr)$. The numerator is nonnegative and the denominator positive, so $F''(z)\ge0$ for $z\ge\Lambda(R)$.
\end{proof}

\begin{lemma}[Derivative regularity implies inverse stability]\label{lem:slow-D-IS}
Under Assumption~\ref{assu:slow-growth}, for every fixed $K<\infty$,
\[
  \sup_{0\le s\le K\sqrt{z/A(z)}}\left|\frac{A(z/(1-s))}{A(z)}-1\right|\longrightarrow0,\qquad z\to\infty .
\]
\end{lemma}
\begin{proof}
Put $\delta(x):=\sqrt{\Lambda(x)/x}$. Since $\Lambda$ is slowly varying, $\delta(x)\to0$. Moreover, by~(A2)(iii) the von Mises condition $x\eta'(x)/\eta(x)\to0$ holds, so $\eta$ is itself slowly varying via its Karamata representation \citep{resnick2007heavy}. Fix $\varepsilon>0$. We have
\[
  \int_x^{(1+\varepsilon)x}\frac{\eta(u)}{u}\,\d u=\int_1^{1+\varepsilon}\frac{\eta(tx)}{t}\,\d t=\eta(x)\int_1^{1+\varepsilon}\frac{\eta(tx)/\eta(x)}{t}\,\d t .
\]
By the uniform convergence theorem for slowly varying functions~\cite[Proposition~2.4]{resnick2007heavy}, $\eta(tx)/\eta(x)\to1$ uniformly for $t\in[1,1+\varepsilon]$ as $x\to\infty$. Writing $\eta(tx)/\eta(x)=1+\varrho(t,x)$, the inner integral is
\[
  \int_1^{1+\varepsilon}\frac{\d t}{t}+\int_1^{1+\varepsilon}\frac{\varrho(t,x)}{t}\,\d t=(1+o(1))\log(1+\varepsilon),
\]
since $\bigl|\int_1^{1+\varepsilon}\tfrac{\varrho(t,x)}{t}\,\d t\bigr|\le\bigl(\sup_{t\in[1,1+\varepsilon]}|\varrho(t,x)|\bigr)\int_1^{1+\varepsilon}\tfrac{\d t}{t}=o(1)\,\log(1+\varepsilon)$, the supremum vanishing as $x\to\infty$ by uniform convergence.
Hence
\[
  \int_x^{(1+\varepsilon)x}\frac{\eta(u)}{u}\,\d u=(1+o(1))\eta(x)\log(1+\varepsilon).
\]
We claim $\delta(x)=o(\eta(x))$, i.e.\ $\eta(x)/\delta(x)\to\infty$; indeed, since $\delta(x)=\sqrt{\Lambda(x)/x}$,
\[
  \log\frac{\eta(x)}{\delta(x)}=\log\eta(x)-\log\delta(x)=\log\eta(x)-\tfrac12\log\Lambda(x)+\tfrac12\log x .
\]
Both $\eta$ and $\Lambda$ are slowly varying, and hence $\log\eta(x)=o(\log x)$ and $\log\Lambda(x)=o(\log x)$, so
\[
  \log\frac{\eta(x)}{\delta(x)}=\tfrac12\log x+o(\log x)\longrightarrow\infty ,
\]
giving $\delta(x)=o(\eta(x))$.
Fix $K\in(0,\infty)$. For all large $x$ we have both $K\delta(x)<1/2$ and $x\ge x_\infty$; then for $s\in[0,K\delta(x)]$ the argument $\Lambda(x)/(1-s)\ge\Lambda(x)\ge\Lambda(x_\infty)$ lies in the domain $[\Lambda(x_\infty),\infty)$ of $A=\Lambda^{-1}$. Since $A=\Lambda^{-1}$ is increasing, $s\mapsto A(\Lambda(x)/(1-s))$ is increasing on $[0,K\delta(x)]$ and is at least $A(\Lambda(x))=x$, so
\begin{equation}\label{eq:slow-supKbound}
  \sup_{0\le s\le K\delta(x)}\left|\frac{A(\Lambda(x)/(1-s))}{x}-1\right|=\frac{A(\Lambda(x)/(1-K\delta(x)))}{x}-1 .
\end{equation}
Let $y(x):=A(\Lambda(x)/(1-K\delta(x)))$. Then $y(x)\ge x$ and  $K\delta(x)<1/2$ yields
\[
  \log\Lambda(y(x))-\log\Lambda(x)=\log\frac1{1-K\delta(x)}\le 2K\delta(x)=o(\eta(x)).
\]
We claim $y(x)/x\to1$ for $x\to\infty$. If not, then for some $\varepsilon>0$ and a subsequence, $y(x)\ge(1+\varepsilon)x$. Along this subsequence we may assume $x\ge x_\infty$, so that $x_\infty\le x\le y(x)$ and the von Mises representation~(A2) applies on $[x,y(x)]$, giving
\[
  \log\Lambda(y(x))-\log\Lambda(x)=\int_x^{y(x)}\frac{\eta(u)}{u}\,\d u\ge\int_x^{(1+\varepsilon)x}\frac{\eta(u)}{u}\,\d u=(1+o(1))\eta(x)\log(1+\varepsilon),
\]
contradicting $\log\Lambda(y(x))-\log\Lambda(x)=o(\eta(x))$. Thus $y(x)/x\to1$, and~\eqref{eq:slow-supKbound} gives $\sup_{0\le s\le K\delta(x)}|A(\Lambda(x)/(1-s))/x-1|\to0$. Writing $z=\Lambda(x)$, so $x=A(z)$ and $\delta(x)=\sqrt{z/A(z)}$, yields the assertion.
\end{proof}

For $0\le t\le1$, set $G(t):=1-\sqrt{1-t^2}$. We shall use
\begin{equation}\label{eq:slow-G-quadratic}
  G(t)=\frac{t^2}{1+\sqrt{1-t^2}}\ge\frac{t^2}{2},\qquad 0\le t\le1 .
\end{equation}

\begin{lemma}[Rapid comparison]\label{lem:slow-rapid}
Let $\varepsilon_R:=\sup_{u\ge R}\eta(u)$. Then $\varepsilon_R\downarrow0$, and for all sufficiently large $R$, all $x\ge R$, and all $0<t\le1$,
\[
  G(t)\ge e^{-\varepsilon_R}\,x\,t\,F\!\left(\frac{\Lambda(x)}{t}\right).
\]
\end{lemma}
\begin{proof}
The statement on $\varepsilon_R$ follows from (A2). Fix $R>x_\infty$ so large that $\varepsilon_R<1/2$. Let $x\ge R$, $0<t\le1$, and put $y(x):=A(\Lambda(x)/t)$. Then $x_\infty<x\le y(x)$ (the second inequality since $t\le1$ and $A$ is increasing), so the von Mises representation~(A2) applies on $[x,y(x)]$ and gives
\[
  -\log t=\log\Lambda(y(x))-\log\Lambda(x)=\int_x^{y(x)}\frac{\eta(u)}{u}\,\d u\le\varepsilon_R\log(y(x)/x).
\]
Hence $x/y(x)\le t^{1/\varepsilon_R}$. Set $w=\sqrt{1-t^2}\in[0,1]$. Then, using the equality in~\eqref{eq:slow-G-quadratic},
\[
  \frac{t^{1/\varepsilon_R}}{G(t)/t}=(1-w^2)^{(1/\varepsilon_R-1)/2}(1+w).
\]
Using $\log(1-w^2)\le-w^2$ and $\log(1+w)\le w$, $\log[t^{1/\varepsilon_R}/(G(t)/t)]\le-\frac{1-\varepsilon_R}{2\varepsilon_R}w^2+w$. The right-hand side is a concave quadratic in $w$, with maximum $\varepsilon_R/[2(1-\varepsilon_R)]\le\varepsilon_R$ for $\varepsilon_R\le1/2$. Therefore $t^{1/\varepsilon_R}\le e^{\varepsilon_R}G(t)/t$. Since $F(\Lambda(x)/t)=1/y(x)$,
\[
  x\,t\,F\!\left(\frac{\Lambda(x)}{t}\right)=\frac{xt}{y(x)}\le t^{1+1/\varepsilon_R}\le e^{\varepsilon_R}G(t),
\]
implying the assertion.
\end{proof}

\begin{lemma}[A variation of Jensen's inequality]\label{lem:slow-allocation}
Let $B>0$, $L>0$, and let $F$ be positive, non-increasing, and convex on $[y_0,\infty)$. Let $\mb M$ be a finite nonnegative measure and $Y$ a measurable function with $Y\ge y_0$ $\mb M$-almost everywhere. If $\mb M(\Re)\ge L$ and $\int Y\,\d\mb M\le B$, then $\int F(Y)\,\d\mb M\ge L\,F(B/L)$.
\end{lemma}
\begin{proof}
Let $N=\mb M(\Re)$. Since $Y\ge y_0$, the average $N^{-1}\int Y\,\d\mb M$ belongs to $[y_0,\infty)$, and Jensen's inequality gives $\int F(Y)\,\d\mb M\ge N F(N^{-1}\int Y\,\d\mb M)$. As $F$ is non-increasing and $\int Y\,\d\mb M\le B$, this is $\ge N F(B/N)$. Because $N\ge L$, $B/N\le B/L$; since $F$ is non-increasing and positive, $N F(B/N)\ge N F(B/L)\ge L F(B/L)$.
\end{proof}

\begin{proof}[Proof of Lemma~\ref{prop:slow-hellinger}]
Fix $\Delta>0$, $\mu\in\Re$, and feasible $\mb P_-\in\mc C_{\mu-\Delta}$, $\mb P_+\in\mc C_{\mu+\Delta}$. Let $\mb S:=(\mb P_-+\mb P_+)/2$. There is a measurable random variable $\tau$ with $|\tau|\le1$ defined by $\tfrac{\d\mb P_+}{\d\mb S}=1+\tau$ and $\tfrac{\d\mb P_-}{\d\mb S}=1-\tau$. We have 
\begin{equation}\label{eq:slow-D-as-G}
  D:=1-\rho(\mb P_-,\mb P_+)=\int G(|\tau|)\,\d\mb S .
\end{equation}
We derive a non-asymptotic lower bound for $D$, independent of $\mb P_\pm$ and $\mu$.

\medskip\noindent\emph{Step 1: centered mean separation.}
Since $\mb P_\pm$ are probability measures, $\int\tau\,\d\mb S=\tfrac12(\mb P_+(\Re)-\mb P_-(\Re))=0$. The mean constraints give $2\Delta=\int x\,\mb P_+(\d x)-\int x\,\mb P_-(\d x)=2\int(x-\mu)\tau(x)\,\mb S(\d x)$, hence
\begin{equation}\label{eq:slow-mean-separation}
  \Delta\le\int|x-\mu|\,|\tau(x)|\,\mb S(\d x).
\end{equation}
Fix $R\ge2x_\infty$ large enough that $F$ is convex on $[\Lambda(R/2),\infty)$ (Lemma~\ref{lem:slow-F-convex}) and that $\sup_{u\ge R/2}\eta(u)<1/2$, so that Lemma~\ref{lem:slow-rapid} is applicable at every scale $\ge R/2$ (as used in Step~3). Since $R/2\ge x_\infty$, the identity $\phi(y)=|y|\Lambda(|y|)$ then holds for all $|y|\ge R/2$. We fix $R$ first and then let $\Delta\downarrow0$; hence we may assume $\Delta<R/2$.

\medskip\noindent\emph{Step 2: the small-$x$ region $|x-\mu|\le R$.}
Let $A_R:=\{|x-\mu|\le R\}$ and $c_R:=\int_{A_R}|x-\mu|\,|\tau(x)|\,\mb S(\d x)$. By Cauchy--Schwarz and $\mb S(A_R)\le1$,
\[
  c_R^2\le\Bigl(\int_{A_R}|x-\mu|^2\,\mb S(\d x)\Bigr)\Bigl(\int_{A_R}|\tau|^2\,\mb S(\d x)\Bigr)\le R^2\int_{A_R}|\tau|^2\,\mb S(\d x),
\]
so by~\eqref{eq:slow-G-quadratic},
\begin{equation}\label{eq:slow-compact-cost}
  \int_{A_R}G(|\tau|)\,\d\mb S\ge\tfrac12\int_{A_R}|\tau|^2\,\d\mb S\ge\frac{c_R^2}{2R^2}.
\end{equation}

\medskip\noindent\emph{Step 3: the large-$x$ region $|x-\mu|\ge R$.}
Let $E_R:=\{|x-\mu|>R,\ |\tau(x)|>0\}$ and, on $E_R$, $u_\Delta(x):=|x-\mu|-\Delta$, so $u_\Delta(x)>R/2$ as we assume $\Delta<R/2$. Define the finite nonnegative measure $\mb M(\d x):=\mathbf 1_{E_R}(x)u_\Delta(x)|\tau(x)|\,\mb S(\d x)$ and
\[
  Y(x):=\begin{cases}\Lambda(u_\Delta(x))/|\tau(x)|,&x\in E_R,\\[2pt]\Lambda(R/2),&x\notin E_R,\end{cases}
\]
so $Y\ge\Lambda(R/2)$. For $x\in E_R$, Lemma~\ref{lem:slow-rapid} with spatial argument $u_\Delta(x)$ and $t=|\tau(x)|>0$ gives $G(|\tau(x)|)\ge e^{-\varepsilon_{R/2}}u_\Delta(x)|\tau(x)|F(\Lambda(u_\Delta(x))/|\tau(x)|)$ with $\varepsilon_{R/2}:=\sup_{u\ge R/2}\eta(u)$. Hence
\begin{equation}\label{eq:slow-tail-rapid}
\begin{aligned}
  \int_{E_R}G(|\tau|)\,\d\mb S&\ge e^{-\varepsilon_{R/2}}\int_{E_R}F\!\left(\frac{\Lambda(u_\Delta(x))}{|\tau(x)|}\right)u_\Delta(x)|\tau(x)|\,\mb S(\d x)\\
  &=e^{-\varepsilon_{R/2}}\int F(Y)\,\d\mb M .
\end{aligned}
\end{equation}
We next estimate $\mb M(\Re)$. By definition of $\mb M$, with $u_\Delta=|x-\mu|-\Delta$,
\[
  \mb M(\Re)=\int_{E_R}|x-\mu|\,|\tau|\,\d\mb S-\Delta\int_{E_R}|\tau|\,\d\mb S .
\]
In the first integral the integrand vanishes where $\tau=0$, so, splitting off $A_R$ and using~\eqref{eq:slow-mean-separation},
\[
  \int_{E_R}|x-\mu|\,|\tau|\,\d\mb S=\int_{|x-\mu|>R}|x-\mu|\,|\tau|\,\d\mb S=\int|x-\mu|\,|\tau|\,\d\mb S-c_R\ge\Delta-c_R,
\]
while, again because the integrand vanishes where $\tau=0$, $\int_{E_R}|\tau|\,\d\mb S=\int_{|x-\mu|>R}|\tau|\,\d\mb S$. Combining,
\begin{equation}\label{eq:slow-M-mass-pre}
  \mb M(\Re)\ge\Delta-c_R-\Delta\int_{|x-\mu|>R}|\tau(x)|\,\mb S(\d x).
\end{equation}
By~\eqref{eq:slow-G-quadratic} and~\eqref{eq:slow-D-as-G}, $\int|\tau|^2\,\d\mb S\le2D$, so Cauchy--Schwarz gives $\int_{|x-\mu|>R}|\tau|\,\d\mb S\le\sqrt{2D}$, whence
\begin{equation}\label{eq:slow-M-mass}
  \mb M(\Re)\ge\Delta-c_R-\Delta\sqrt{2D}.
\end{equation}
For the budget, since $\mb P_-\in\mc C_{\mu-\Delta}$ and $\mb P_+\in\mc C_{\mu+\Delta}$,
\begin{align*}
  B\ge&\,\frac12\!\int\phi(x-\mu+\Delta)\,\mb P_-(\d x)+\frac12\!\int\phi(x-\mu-\Delta)\,\mb P_+(\d x)\\
  =& \int\Bigl[\frac{(1-\tau)}2\phi(x-\mu+\Delta)+\frac{(1+\tau)}2\phi(x-\mu-\Delta)\Bigr]\mb S(\d x)\\
  =& \int\Bigl[\frac{(1-\tau)}2\phi(|x-\mu+\Delta|)+\frac{(1+\tau)}2\phi(|x-\mu-\Delta|)\Bigr]\mb S(\d x).
\end{align*}
For every $x$, the triangle inequality gives $|x-\mu\pm\Delta|\ge|x-\mu|-\Delta=u_\Delta(x)$. On $E_R$, $u_\Delta(x)>R/2\ge x_\infty>0$, so both $|x-\mu\pm\Delta|$ and $u_\Delta(x)$ lie in $[0,\infty)$ where $\phi$ is non-decreasing; the integrand above is at least $\phi(u_\Delta(x))$ on $E_R$. As that integrand is nonnegative everywhere, we may restrict the budget integral to $E_R$, where $\phi(u_\Delta)=u_\Delta\Lambda(u_\Delta)$ as $u_\Delta>x_\infty$. Hence
\begin{equation}\label{eq:slow-Y-budget}
  \int Y\,\d\mb M=\int_{E_R}u_\Delta(x)\Lambda(u_\Delta(x))\,\mb S(\d x)=\int_{E_R}\phi(u_\Delta(x))\,\mb S(\d x)\le B.
\end{equation}

\medskip\noindent\emph{Step 4: the master inequality.}
Define $L_R:=(\Delta-c_R-\Delta\sqrt{2D})_+$ and $\Psi_R:=L_R F(B/L_R)$ if $L_R>0$, else $\Psi_R:=0$. If $L_R>0$, then $\mb M(\Re)\ge L_R$ by~\eqref{eq:slow-M-mass}, $Y\ge\Lambda(R/2)$, and~\eqref{eq:slow-Y-budget} holds, so Lemma~\ref{lem:slow-allocation} gives $\int F(Y)\,\d\mb M\ge L_R F(B/L_R)=\Psi_R$; if $L_R=0$ this is trivial. Combining with~\eqref{eq:slow-D-as-G},~\eqref{eq:slow-compact-cost}, and~\eqref{eq:slow-tail-rapid},
\begin{equation}\label{eq:slow-master-bound}
  D\ge\frac{c_R^2}{2R^2}+e^{-\varepsilon_{R/2}}\Psi_R .
\end{equation}

\medskip\noindent\emph{Step 5: scalar uniformity argument.}
Put $z:=B/\Delta$, so that $m(\Delta)=\Delta/A(z)$, and fix $\theta\in(0,1)$. We distinguish two cases according to the size of $D$ relative to $m(\Delta)$.

\emph{Case 1 ($D\ge\theta m(\Delta)$).} Dividing by $m(\Delta)$,
\begin{equation}\label{eq:slow-case-large-D}
  \frac{D}{m(\Delta)}\ge\theta .
\end{equation}

\emph{Case 2 ($D<\theta m(\Delta)$).} The master bound~\eqref{eq:slow-master-bound}, dropping the nonnegative $e^{-\varepsilon_{R/2}}\Psi_R$ term, gives $D\ge c_R^2/(2R^2)$; with $D<\theta m(\Delta)$ this yields $c_R<R\sqrt{2\theta m(\Delta)}$. Define $s_\Delta:=(c_R+\Delta\sqrt{2D})/\Delta$. From $D<\theta m(\Delta)$ and $c_R<R\sqrt{2\theta m(\Delta)}$, and substituting $m(\Delta)=\Delta/A(z)$ with $\Delta=B/z$,
\begin{equation}\label{eq:slow-s-bound}
  s_\Delta\le R\sqrt{\frac{2\theta m(\Delta)}{\Delta^2}}+\sqrt{2\theta m(\Delta)}=R\sqrt{\frac{2\theta}{B}}\sqrt{\frac{z}{A(z)}}+\frac{\sqrt{2\theta B}}{z}\sqrt{\frac{z}{A(z)}} .
\end{equation}
For $0<\Delta<B$ we have $z=B/\Delta>1$; set $K_{R,\theta}:=R\sqrt{2\theta/B}+\sqrt{2\theta B}$. Then, for such $\Delta$,
\begin{equation}\label{eq:slow-s-admissible}
  0\le s_\Delta\le K_{R,\theta}\sqrt{z/A(z)}=:S_\Delta(R,\theta).
\end{equation}
By Lemma~\ref{lem:slow-inverse-growth}, $z/A(z)\to0$, so $S_\Delta(R,\theta)\to0$ and $S_\Delta(R,\theta)<1$ for small $\Delta$; then~\eqref{eq:slow-s-admissible} gives $L_R=\Delta(1-s_\Delta)>0$, so $\Psi_R=L_R F(B/L_R)$. Set $\omega_\Delta(R,\theta):=\sup_{0\le s\le S_\Delta(R,\theta)}|A(z/(1-s))/A(z)-1|$. Lemma~\ref{lem:slow-D-IS} with $K=K_{R,\theta}$ gives $\omega_\Delta(R,\theta)\to0$. Since $s_\Delta\le S_\Delta(R,\theta)$ and $B/L_R=z/(1-s_\Delta)$,
\begin{equation}\label{eq:slow-tail-asymptotic-uniform}
  \Psi_R=\frac{\Delta(1-s_\Delta)}{A(z/(1-s_\Delta))}=m(\Delta)\frac{1-s_\Delta}{A(z/(1-s_\Delta))/A(z)}\ge m(\Delta)\frac{1-S_\Delta(R,\theta)}{1+\omega_\Delta(R,\theta)} .
\end{equation}
Combining~\eqref{eq:slow-master-bound} and~\eqref{eq:slow-tail-asymptotic-uniform} in this case,
\begin{equation}\label{eq:slow-case-small-D}
  \frac{D}{m(\Delta)}\ge e^{-\varepsilon_{R/2}}\frac{1-S_\Delta(R,\theta)}{1+\omega_\Delta(R,\theta)} .
\end{equation}
The alternatives~\eqref{eq:slow-case-large-D} and~\eqref{eq:slow-case-small-D} show that every admissible triple $(\mu,\mb P_-,\mb P_+)$ satisfies
\[
  \frac{1-\rho(\mb P_-,\mb P_+)}{m(\Delta)}=\frac{D}{m(\Delta)}\ge\min\!\left\{\theta,\ e^{-\varepsilon_{R/2}}\frac{1-S_\Delta(R,\theta)}{1+\omega_\Delta(R,\theta)}\right\},
\]
with right-hand side independent of $\mu,\mb P_-,\mb P_+$. For every sufficiently large $R$, taking $\liminf_{\Delta\downarrow0}$ and using $S_\Delta(R,\theta)\to0$, $\omega_\Delta(R,\theta)\to0$ gives
\[
  \liminf_{\Delta\downarrow0}\inf_{\mu\in\Re}\inf_{\mb P_{\pm}\in\mc C_{\mu\pm\Delta}}\tfrac{(1-\rho(\mb P_-,\mb P_+))}{m(\Delta)}\ge\min\{\theta,e^{-\varepsilon_{R/2}}\}.
\]
Letting $\theta\uparrow1$ and then $R\to\infty$, and using $\varepsilon_{R/2}\downarrow0$, proves the claim.
\end{proof}

\end{document}